\documentclass[a4paper,11pt,oneside]{article}

\usepackage[top=2cm, bottom=2cm, left=2cm, right=2cm]{geometry}

\usepackage{graphicx}
\usepackage{multicol}
\usepackage[bottom]{footmisc}
\usepackage[binary-units=true]{siunitx}
\usepackage{algorithm}
\usepackage[noend]{algpseudocode}
\usepackage{pgfplots}
\usepackage{pgfplotstable}
\usepackage{booktabs}
\usepackage{authblk}
\usepackage{amsmath,amsthm,mathtools}
\usepackage{amsfonts,amssymb}
\usepackage{hyperref}
\usepackage{float} 
\usepackage{algorithm} 
\usepackage{algpseudocode}
\usepackage{pgfplots}
\pgfplotsset{compat=1.15}
\usepackage{mathrsfs}
\usetikzlibrary{arrows}
\usepackage{tikz-3dplot}

\newif\ifMAKEPICS
\MAKEPICSfalse 

\ifMAKEPICS
\usepackage[cleanup,subfolder]{gnuplottex}
\usepackage{xparse}

\ExplSyntaxOn
\DeclareExpandableDocumentCommand{\convertlen}{ O{cm} m }
{
	\dim_to_decimal_in_unit:nn { #2 } { 1 #1 } cm
}
\ExplSyntaxOff
\fi

\newtheorem{Algorithm}{Algorithm}
\newtheorem{Lemma}{Lemma} 
\newtheorem{Theorem}{Theorem} 
\newtheorem{Remark}{Remark}
\newtheorem{Assumption}{Assumption}
\newtheorem{Definition}{Definition}
\newtheorem{Proposition}{Proposition}

\newtheorem{Corollary}{Corollary}

\begin{document}
\title{A Posteriori Single- and Multi-Goal Error Control and Adaptivity
for Partial Differential Equations
}

\author[1,4]{{B. Endtmayer}}
\author[2,5]{{U. Langer}}
\author[3]{{T. Richter}}
\author[2]{{A. Schafelner}}
\author[1,4]{{T. Wick}}

\affil[1]{Leibniz University Hannover, Institute of Applied Mathematics,
            Germany}
\affil[2]{Johannes Kepler University Linz, Institute of Numerical Mathematics,
                Austria}
\affil[3]{Otto von Guericke University Magdeburg, Institute of Analysis and Numerics,
            Germany}
\affil[4]{Leibniz University Hannover, Cluster of Excellence PhoenixD,
			Germany}
\affil[5]{Austrian Academy of Sciences, RICAM, 
                Linz, Austria}
	
\date{}

\maketitle

\begin{abstract}
This work reviews goal-oriented a posteriori error control, adaptivity and solver control 
for finite element approximations to boundary and initial-boundary value problems for 
stationary and non-stationary partial differential equations,
respectively. In particular,
coupled field problems with different physics may require simultaneously the accurate evaluation of several quantities of interest, which is achieved 
with multi-goal oriented error control. Sensitivity measures are obtained by solving an adjoint problem. Error localization is achieved with the help of 
a partition-of-unity. We also review and extend theoretical results for 
efficiency and reliability by employing a saturation assumption. The resulting 
adaptive algorithms allow to balance discretization and non-linear iteration errors, and are demonstrated for four applications: 
Poisson's problem,
non-linear elliptic boundary value problems,
stationary incompressible Navier-Stokes equations,
and regularized parabolic $p$-Laplace initial-boundary value problems.
Therein, different finite element discretizations in two different software libraries are utilized,
which are partially accompanied with open-source implementations on GitHub.
\\
\textbf{Keywords:}  
Goal-oriented error control; 
multi-goal error control; 
adjoint problems; 
dual-weighted residual method; 
partial differential equations; 
adaptive finite element methods
\end{abstract}

%
\section{Introduction}
\label{Sec:Introduction}
This work is devoted to a goal-oriented a posteriori error control of single and multiple 
quantities of interest (QoI) and Adaptive Finite Element Methods (AFEM)
for solving stationary and non-stationary, linear and non-linear partial differential equations (PDEs)
and systems of PDEs. AFEM were certainly inspired by 
the pioneer work \cite{ELRWS:BabuskaRheinboldt:1978SINUM} of Babu\v{s}ka and Rheinboldt. Further important studies on a posteriori error controlled 
adaptive finite element methods are \cite{Doerfler:1996a,CaVer99,MR1770058,MR2050077,MR2324418} to name a few.
The first comprehensive monograph \cite{Verfuerth:1996a} by Verf\"urth was another AFEM milestone. 
We refer to the survey papers
\cite{ErikssonEstepHansboJohnson1995},
\cite{BeRa01},
\cite{ChamoinLegoll2023},
and to the monographs 
\cite{AinsworthOden:2000},
\cite{ELRWS:BabuskaStrouboulis:2001Book},
\cite{BaRa03},
\cite{ELRWS:Han:2004Book},
\cite{NeittaanmaekiRepin:2004a},
\cite{ELRWS:Repin:2008Book},
\cite{ErEstHaJoh09},
\cite{ReSau20}
for an overview of a posteriori error estimates and adaptive finite element techniques.

The governing discretization in this work is the finite element method 
\cite{ELRWS:BrennerScott:2008_Book,Braess,Ciarlet:2002:FEM:581834,Hu00_book,CaOd84,ErnGuer2021all_three}, but isogeometric 
analysis could be used in a similar fashion \cite{HughesCottrellBazilevs:2005a,cottrellHughesIGA2009,IGA}. We consider stationary settings as well as time-dependent cases, modeled as space-time 
problems and solved by space-time discretization methods;~see, e.g., \cite{ELRWS:Johnson:2009Book,ELRWS:HughesHulbert:1988CMAME,Thomee1997,Steinbach:2015a,NoSauWie17,SteinbachYang:2019a,LaStein19,TezTa19,LaNoSauWie22} 
and the references therein.

In the following, we review the important steps regarding the development 
of goal-oriented techniques. 
Starting with the Acta Numerica paper \cite{ErikssonEstepHansboJohnson1995} 
{by Eriksson, Estep, Hansbo and Johnson}
as a survey on adaptive methods, Becker and Rannacher proposed an automated procedure for self-controlled 
adaptivity with the dual-weighted residual method (DWR) 
in their Acta Numerica paper \cite{BeRa01}. This was followed by the 
book \cite{BaRa03}. Further studies of the early developments include 
\cite{NoSchmiSiebVeeser2006,GilPie2004,PruOdWeBaBo03,GilesSuli2002,BraackErn02,Giles2001,GilPie2000,OdPru99,RaSu97,BeRa96}. An important step towards time-dependent problems 
in space-time formulations with full space-time adaptivity was the work by 
Schmich and Vexler \cite{SchVe08}.
These previously mentioned studies concentrated on single goal-oriented error estimation. 
In the year $2003$, Hartmann and Houston proposed 
goal-oriented a posteriori error estimation for multiple quantities 
of interest \cite{HaHou03}, followed by Hartmann's paper \cite{Ha08} and shortly later 
\cite{PARDO20101953}.
We started ourselves with multigoal-oriented error estimation with \cite{EnWi17,EnLaWi18}.

Conceptional developments on goal-oriented error estimation include 
a safe-guarded method established in \cite{NochettoVeeserVerani2009}
and guaranteed bounds derived in \cite{AinRan2012}.
Abstract analyses and reformulations 
were presented in \cite{FeiPraeZee16,KerPruChaLaf2017}.
First convergence rates of goal-oriented error estimation were obtained in \cite{MoSte09},
weighted marking \cite{Becker:2011:WMG:2340478.2340490},
bounding techniques for goal-oriented error estimates for elliptic problems \cite{LadPleCha2013},
goal oriented flux reconstruction \cite{MozPru2015},
goal-oriented error estimation for the fractional-step-theta scheme \cite{MeiRi14},
a partition-of-unity localization including effectivity estimates for single goal functionals 
\cite{RiWi15_dwr},
partition-of-unity localization for multiple goal functionals \cite{EnWi17},
space-time partition-of-unity localization \cite{endtmayer2023goal,ThiWi24},
equilibrated flux \cite{BrPillSb2009}, 
linearization errors \cite{granzow2023linearization}.
Theoretical work showing optimality of adaptive algorithms 
was summarized in four axioms of adaptivity in \cite{MR3170325}
and dates back to residual-based a posteriori error estimates 
in \cite{Doerfler:1996a}. First rigorous convergence results for 
goal-oriented estimators go back to \cite{MoSte09,Becker:2011:WMG:2340478.2340490}, then \cite{FeiPraeZee16,HoPoZhu2015,HolPo2016}, and recent findings 
are \cite{MR4235813,becker2022rate,becker2023goal}.
We also mention our own prior work on such theoretical advancements 
that include effective estimates with common upper bounds of the dual-weighted 
residual estimator and the error indicators \cite{RiWi15_dwr}.
Efficiency and reliability estimates for the dual-weighted residual method
using a saturation assumption were shown in \cite{EnLaWi20}, and 
smart algorithms switching between solving high-order adjoint 
problems or using interpolations were derived in \cite{endtmayer2021reliability}.

A key step in the development of the dual-weighted residual 
method was the introduction of a partition-of-unity (PU) localization
by Richter and Wick \cite{RiWi15_dwr}; a prototype open-source implementation 
based on deal.II \cite{BangerthHartmannKanschat2007,dealII90,deal2020} 
can be found on GitHub\footnote{\url{https://github.com/tommeswick/PU-DWR-Poisson}}. 
For the general 
idea of the PU-FEM, we refer the reader to \cite{BabMel96,BabMel97}.
The PU-DWR method opened the way for addressing 
goal-oriented error control in practical applications of 
time-dependent, non-linear, coupled PDEs and multiphysics applications in a much more convenient way.
Along with the PU-DWR method, 
an indicator index was introduced, which measures the quality of error indicators.
The PU-DWR method was extended to multigoal-oriented error control 
in \cite{EnWi17}. 
The extension of PU-DWR to space-time goal-oriented error estimation on fully
unstructured simplicial decompositions of the space-time cylinder was done 
in \cite{endtmayer2023goal} and a $cGdG$ (continuous Galerkin in space, discontinuous 
Galerkin in time) discretization was done for parabolic problems in \cite{ThiWi24}, including 
open-source developments in Zenodo \cite{thiele_2024_10641119,thiele_2024_10641104} 
and GitHub\footnote{\url{https://github.com/jpthiele/pu-dwr-diffusion}},\footnote{\url{https://github.com/jpthiele/pu-dwr-combustion}}, 
and finally, space-time PU-DWR for the incompressible Navier-Stokes equations in \cite{RoThiKoeWi23}, again 
with codes on GitHub\footnote{\url{https://github.com/mathmerizing/dwr-instatfluid}}.

Single- and multigoal-oriented error control and adaptivity 
has found numerous applications.
These include Poisson's problem \cite{BeckerRannacher1995,BeRa96,OdPru99,EnWi17},
stationary linear elasticity \cite{RaSu97},
space-time adaptivity for parabolic equations \cite{SchVe08},
space-time elasticity and the wave equation \cite{Rade09,BaGeRa10},
multi-rate discretizations of coupled parabolic/hyperbolic problems \cite{SoszynskaRichter2021},
$p$-Laplace problems \cite{EnLaWi18,endtmayer2021reliability,endtmayer2023goal, RanVi2013},
elasto-plasticity \cite{RannacheSuttmaeier1998,RaSu99,RaSu00} and visco-plasticity \cite{MeRi2020},
incompressible flow \cite{becker2002optimal,BraaRich2006a,Besier2009,BeRa12,endtmayer2021hierarchical,EnLaWi20,endtmayer2021reliability,RoThiKoeWi23},
transport proplems \cite{KuzminKorotov2010},
transport problems with coupled flow \cite{bause2021flexible,bruchhauser2022goal,bruchhauser2023cost},
Boussinesq equations and coupled flow with temperature \cite{BeEnLaWi2021a,BeuDeEndtMoWi24,endtmayerPAMM2022},
reactive flows \cite{B59,Ri05,B3}, 
uncertain inputs \cite{bespalov2019goal},
Maxwell's equations \cite{INGELSTROM20032597},
elliptic eigenvalue problems \cite{Heuveline2001,cockburn2022adjoint}, 
modeling errors \cite{OdPr02,BraackErn02,Rade19},
applications to polygonal meshes \cite{WeiWi18},
conforming and nonconforming approximations with inexact solvers \cite{MaVohYou20},
anisotropic mesh refinement \cite{Richter10},
heterogeneous multiscale problems \cite{chung2016goal,MaiRa16,MaiRa18,LautschRichter2020,Dominguez2023},
fluid-structure interaction \cite{GraeBa06,Ri11,Du06,Du07,DuRaRi09,Wi11_phd,Wi2012,FickBrummelenZee2010,ZeeBrummelenAkkermanBorst2011,Ri17_fsi,FaiWi18,ahuja2022multigoal},
phase-field slits and fracture \cite{Wi16_dwr_pff,Wi20,Wi21_CMAM},
sea ice simulations \cite{MeRi2020},
optimal control \cite{BeKaRa00,BeBaMeRa07,MeiVe07,Me08,HiHo08a,VeWo08,HiHo10,RannancherVexler2010,HiHiKa18,EnLaNeiWoWi2020},
obstacle/contact problems and variational inequalities \cite{BlumSuttmeier99,BlumSu00,blum2003posteriori,SchroeRade2011,rademacher2015dual,Su08,Wi20},
finite cell methods \cite{StRaSchroe19,DiStolfo2022},
balancing discretization and iteration errors \cite{MeiRaVih109,RaWeWo10,RanVi2013,EnLaWi18,dolejvsi2021goal,dolejvsi2023goal},
financial mathematics \cite{GoRaWo15},
goal-oriented model order reduction \cite{Meyer2003,fischer2024adaptive,fischer2024dwr,KerChaLaPr19,ChaLe21}, 
neural network enhanced dual-weighted residual technologies \cite{BREVIS2021186,minakowski2021error,Roth2022},
adaptive multiscale predictive modeling \cite{oden_2018}, and 
open-source software developments for 
goal-oriented error estimation 
\cite{DOpElib}\footnote{\url{https://github.com/winnifried/dopelib}}, 
\cite{KoeBruBau2019a}\footnote{\url{https://github.com/dtm-project/dwr-diffusion}}, 
\cite{ThiWi24,Thi24_phd}\footnote{\url{https://github.com/jpthiele/pu-dwr-diffusion}},\footnote{\url{https://github.com/jpthiele/pu-dwr-combustion}}.

In this work, we undertake an extensive review on single- and multigoal-oriented a posteriori 
error estimation and adaptivity. Our results include classical proofs for 
educational purposes. But we also add new findings: first, we establish 
efficiency and reliability estimates 
with only one saturation assumption in which the 
strengthened condition from \cite{EnLaWi20} 
is no longer necessary.
Second, we provide details on non-standard discretizations such 
as non-consistent and non-conformal methods. Third, 
the code of the basic PU-DWR method \cite{RiWi15_dwr} is published open-source 
on GitHub. Fourth, in extension of \cite{endtmayer2023goal}, the space-time 
PU-DWR method is investigated for multiple goals with singularities.
Certain results are illustrated with the elliptic model problem. Our main theoretical results 
include estimates and algorithms for balancing discretization and non-linear iteration errors.
Our algorithms are illustrated with newly designed numerical tests, not published in the 
literature elsewhere, that have both research and educational character in order 
to explain the mechanisms of goal-oriented error control and mesh adaptivity.
These results include stationary, non-linear situations, incompressible flow, and 
a space-time $p$-Laplace problem.

The outline of this work is as follows. 
In Section~\ref{sec_notation_abstract_setting}, the notation and abstract setting are introduced. 
Then, in Section~\ref{Sec:single-goal}, single goal-oriented error control is discussed. 
Section~\ref{Sec:Nonstandard} focuses 
on error estimation for non-standard discretizations.
Next, in Section~\ref{Sec:Multigoal}, multigoal-oriented error estimators are explained. 
In Section~\ref{sec_applications}, three applications are considered, and computationally analyzed. 
The work is concluded in Section~\ref{Sec:Conclusions}, 
and some current research directions are outlined.

%
%
\section{Notation, Abstract Setting, Finite Element Discretization}
\label{sec_notation_abstract_setting}
Throughout this work, let $d$ be the space dimension.
We denote by $Q:=\Omega\times I \subset \mathbb{R}^{d+1}$ the space-time domain (cylinder), where $\Omega \subset \mathbb{R}^{d}$
is the bounded and Lipschitz spatial domain with the boundary $\Gamma = \partial \Omega$, 
and $I:=(0,T)$ denotes the temporal domain (time interval). 
Here, $T>0$ is the end time value.
Furthermore, we use the standard notations for Lebesgue, Sobolev, and Bochner spaces like 
$L_p(\Omega)$, 
$W_p^k(\Omega),\mathring{W}_p^k(\Omega), H^k(\Omega)=W_2^k(\Omega), H^k_0(\Omega)=\mathring{W}_2^k(\Omega)$,
$L_p(0,T;\mathring{W}_p^1(\Omega))$ etc.; see, e.g., 
\cite{Braess,ELRWS:BrennerScott:2008_Book,Zeidler:1990a}.

\subsection{Abstract Setting}\label{subsec: Abstract Setting}
Let $U$ and $V$ be reflexive Banach spaces
with their dual spaces $U^*$ and $V^*$, respectively.
We consider the abstract operator equation:
Find $u \in U$ such that
\begin{equation}\label{eq: General Modelproblem}
	\mathcal{A}(u)=0 \qquad \text{in }V^*,
\end{equation}
where $\mathcal{A}:U \mapsto V^*$ represents a non-linear partial differential operator. 

Let us assume $U_h$ and $V_h$ are conforming discrete spaces, i.e.\ we have the properties $U_h \subseteq U$ and $V_h \subseteq V$ and $\operatorname{dim}(U_h) < \infty$ and $\operatorname{dim}(V_h)< \infty$. 
With this, we can perform a Galerkin-Petrov discretization of 
the operator equation \eqref{eq: General Modelproblem} yielding the 
discrete problem: Find $u_h \in U_h$ such that
\begin{equation}\label{eq: discrete primal}
	\mathcal{A}(u_h)(v_h)=0 \qquad \forall v_h \in V_h.
\end{equation}
Once a basis is chosen, this leads to a non-linear system of finite  equations.
Afterwards, this system of non-linear equations can be solved with some non-linear solver, e.g.\ 
Picard's iteration or Newton's method \cite{Deuflhard2011}.

\subsection{Finite Element Discretization}
One possible discretization technique is the Finite Element Method (FEM) \cite{ELRWS:BrennerScott:2008_Book,Braess,Ciarlet:2002:FEM:581834,Hu00_book,CaOd84,ErnGuer2021all_three}.
In this section, we assume that $\Omega$ is a polygonal Lipschitz domain.
As an example for the conforming discrete subspaces from Section~\ref{subsec: Abstract Setting}, we provide two possible finite element discretizations.

\subsubsection{Finite Elements on Simplices $P_k$}
We can decompose our domain into {shape-regular} simplicial elements $K \in \mathcal{T}_h$, where 
$\bigcup_{K \in \mathcal{T}_h}\overline{K}=\overline{\Omega}$ and for all $K,K' \in \mathcal{T}_h:$  $K \cap K' = \emptyset$ if and only if $K\not=K'$. 
Let $\mathbb{P}_k(\hat{K})$ be the space of polynomials of  total degree $k$ on the reference domain $\hat{K}$ and 
\begin{equation*}
	{
	P_k:=\{v_h \in \mathcal{C}(\Omega): v_h|_K = v_h(\varphi_K(\cdot)) \in \mathbb{P}_k(\hat{K})\quad \forall K \in \mathcal{T}_h \}
	}
\end{equation*}
the finite element space of continuous finite elements of degree $k$ on  $\mathcal{T}_h$,
where $\varphi_K: \hat{K} \rightarrow K$ is a regular mapping of the reference element $\hat{K}$ onto $K$,
e.g., a affine-linear or isoparametric mapping.
For more information, we refer to \cite{ELRWS:BrennerScott:2008_Book,Braess,ErnGuer2021all_three}.

\subsubsection{Finite Elements on Hypercubes $Q_k$}
\label{sec_FEM_Qk}
Another possible way is to decompose our domain into hypercubal elements $K \in \mathcal{T}_h$, where 
$\bigcup_{K \in \mathcal{T}_h}\overline{K}=\overline{\Omega}$ and for all $K,K' \in \mathcal{T}_h:$  $K \cap K' = \emptyset$ if and only if $K\not=K'$. 
Let $\mathbb{Q}_k(\hat{K})$ be the tensor product space of polynomials of degree $k$ on the reference domain $\hat{K}$ and 
	$$Q_k:=\{v_h \in \mathcal{C}(\Omega): v_h|_K = v_h(\varphi_K(\cdot)) \in \mathbb{Q}_k(\hat{K})\quad \forall K \in \mathcal{T}_h \}$$
the finite element space of continuous finite elements of degree $k$ on  $\mathcal{T}_h$, 
where $\varphi_K: \hat{K} \rightarrow K$ is a regular mapping of the reference element $\hat{K}$ onto $K$,
e.g., a multi-linear 
or isoparametric mapping.
We refer to \cite{ELRWS:BrennerScott:2008_Book,Braess,ErnGuer2021all_three} for more information.

To facilitate adaptive mesh
refinement and to avoid connecting elements, we use the concept of 
hanging nodes. Elements are allowed to have nodes that lie on the midpoints
of the faces or edges of neighboring cells. In our implementations, at most, one hanging node 
is allowed on each face or edge. In three dimensions, this concept is 
generalized to subplanes and faces because we must deal with
two types of lower manifolds. To enforce global
continuity (i.e., global conformity), the degrees of freedom 
located on the interface between different refinement levels have to satisfy
additional constraints. They are determined by interpolation of neighboring 
degrees of freedom. Therefore, hanging nodes do not carry any degrees of freedom.
For more details on this, we refer to \cite{CaOd84}.

\section{Single-Goal Oriented Error Control}\label{Sec:single-goal}
In this section, we first provide some background 
information and we explain the need for goal-oriented error estimation. 
Moreover, we also introduce and explain adjoint-based error estimation employing the so-called 
dual-weighted residual (DWR) method.

\subsection{Motivation and Preliminaries}
First, we address the purpose to construct error estimators 
and adaptive schemes. Error estimators allow to determine 
approximately the error between approximations and the (unknown) true solution 
of a given problem statement. One distinguishes:
\begin{itemize} 
\item a priori estimates (estimated before the actual approximate 
solution is known);
\item a posteriori error estimates after the approximate 
solution has been computed.
\end{itemize}
Based on these error estimations, the error may be further localized in order 
to `refine' algorithms. Mostly, mesh refinement in space and/or time is of interest,
but also model errors or iteration errors can be controlled and balanced
\cite{AinsworthOden:2000,ReSau20,BeRa01,BaRa03,ErikssonEstepHansboJohnson1995}.
These localized errors can be used to adaptively steer the algorithm, by enhancing 
the accuracy of the approximate solution, while keeping the computational cost reasonable.

In anticipation of practical realization and optimality of 
adaptive procedures, 
the basic algorithm reads:
\begin{Algorithm}[AFEM - adaptive finite elements]\index{AFEM! Basic algorithm}
\index{Adaptive finite elements}
\index{Finite elements! Adaptivity}
The basic algorithm for AFEM reads:
\begin{enumerate}
\item \textbf{Solve} the differential equation on the current mesh ${\cal T}$;
\item \textbf{Estimate} the error via a posteriori error estimation to obtain $\eta$;
\item \textbf{Mark} the elements by localizing the error estimator;
\item \textbf{Refine/coarsen} the elements with the highest/lowest
error contributions using a certain refinement strategy.
\end{enumerate}
\end{Algorithm}
In the 70s, a priori and a posteriori error estimates 
were derived based on global norms, e.g., the $L^2$-norm, the $H^1$-norm 
or even classical $C^0$- and $C^1$-norms.
With this, the entire solution 
is controlled and resulting adaptive schemes act correspondingly. 
However, in many applications, only certain parts of the numerical 
solution are of interest. Regarding the geometry (domains $Q, \Omega, I$),
such parts can be subdomains $\tilde Q\subset Q$, 
line evaluations $\Gamma\subset \bar{Q}$ or even simply points $(x,t)\in Q$ at which 
solution information (values, derivatives) are evaluated.
In case of ordinary differential equations (ODEs) 
or partial differential equation (PDEs), not all solution components 
may be of interest simultaneously.
Clearly, these restrictions cannot be modeled in terms of
global norms, but only in terms of so-called goal
functionals that specify certain
quantities of interest. 

In this work, we denote this quantity of interest by $J$. 
Even though we are interested in $J(u)$, all we can obtain is $J(u_h)$, where $u_h$ solves \eqref{eq: discrete primal}, or $J(\tilde{u})$, where $\tilde{u}$ is an approximation of $u_h$. Therefore, we focus on error control of $J(u)-J(u_h)$, i.e. goal-oriented error estimation for $J$. 
There are many techniques for goal-oriented error estimation as, for instance, presented in the works \cite{BruSchweBau18,FeiPraeZee16,becker2002optimal,becker2022rate,becker2023goal,dolejvsi2021goal,dolejvsi2023goal,creuse2023goal,KerPruChaLaf2017,Ri11,Richter10}.

\subsection{Adjoint Problem}
In order to realize goal-oriented error control in this work,
we use the DWR method, which requires solving an adjoint problem. This adjoint formulation follows from the Lagrange formalism and is given by:
Find $z \in V$
such that
\begin{equation} \label{eq: General Adjointproblem}
	\mathcal{A}'(u)(v,z)=J'(u)(v) \qquad \forall v \in U,
\end{equation}
where $u$ is the solution of \eqref{eq: General Modelproblem} and $\mathcal{A}'$ and $J'$ describe the G\^ateaux-derivatives of $\mathcal{A}$ and $J$, respectively. 
The discrete adjoint problem is given by: Find $z_h \in V_h$ such that 
\begin{equation} \label{eq: discrete adjoint}
\mathcal{A}'(u_h)(v_h,z_h)=J'(u_h)(v_h) \qquad \forall v_h \in U_h,
\end{equation}
where $u_h$ is the solution of \eqref{eq: discrete primal}.
The reasoning and details on how the adjoint problem arises will 
become clear in the following.

\subsection{An Error Identity}
\label{SubSec:Error-Identity}
The adjoint problem \eqref{eq: General Adjointproblem} 
allows us to represent the error 
in a specific way as shown in the following theorem.
\begin{Theorem}[see \cite{RanVi2013,EnLaWi18}] \label{thm: error identity}
	Let $\tilde{u} \in U$ and $\tilde{z} \in V$ be arbitrary but fixed,
	and let $u \in U$ be the solution of the model problem \eqref{eq: General Modelproblem},
	and $z \in V$ be the solution of the adjoint problem \eqref{eq: General Adjointproblem}. 
	If
	$\mathcal{A} \in \mathcal{C}^3(U,V^*)$ and $J \in \mathcal{C}^3(U,\mathbb{R})$, 
	then 
		\begin{align} \label{eq:Error Representation}
				\begin{split}
					J(u)-J(\tilde{u})&= \frac{1}{2}\left(\rho(\tilde{u})(z-\tilde{z})+\rho^*(\tilde{u},\tilde{z})(u-\tilde{u}) \right)
					-\rho (\tilde{u})(\tilde{z}) + \mathcal{R}^{(3)}
					\end{split}
			\end{align}
	 for arbitrary but fixed  $\tilde{u} \in U$ and $ \tilde{z} \in V$,
		 where
		\begin{align}
				\label{eq:Error Estimator: primal}
				\rho(\tilde{u})(\cdot) &:= -\mathcal{A}(\tilde{u})(\cdot), \\
				\label{eq:Error Estimator: adjoint}
				\rho^*(\tilde{u},\tilde{z})(\cdot) &:= J'(\tilde{u})-\mathcal{A}'(\tilde{u})(\cdot,\tilde{z}), 	
			\end{align}
		and the remainder term
		\begin{equation}
					\begin{split}	\label{eq: Error Estimator: Remainderterm}
						\mathcal{R}^{(3)}:=\frac{1}{2}\int_{0}^{1}[J'''(\tilde{u}+se)(e,e,e)
						-\mathcal{A}'''(\tilde{u}+se)(e,e,e,\tilde{z}+se^*)
						-3\mathcal{A}''(\tilde{u}+se)(e,e,e^*)]s(s-1)\,ds,
						\end{split} 
			\end{equation}
	with $e=u-\tilde{u}$ and $e^* =z-\tilde{z}$.
	\begin{proof}
		The proof can be found in \cite{RanVi2013,EnLaWi18}. However, for completeness of our presentation, we will {also include} the proof here.
		Let us define $x := (u,z) \in  X:=U \times V$ and $\tilde{x}:=(\tilde{u},\tilde{z}) \in X$. Moreover, we define the Lagrange function as 
		\begin{equation*}
			\mathcal{L}(\hat{x}):= J(\hat{u})-\mathcal{A}(\hat{u})(\hat{z}) \quad \forall (\hat{u},\hat{z})=:\hat{x} \in X.
			\end{equation*}
		Since $\mathcal{A} \in \mathcal{C}^3(U,V^*)$ and $J \in
		\mathcal{C}^3(U,\mathbb{R})$ this Lagrange function is in $\mathcal{C}^3(X,\mathbb{R})$. Using the fundamental theorem of calculus,
we can write the difference $\mathcal{L}(x)-\mathcal{L}(\tilde{x})$ as
		\begin{equation*}
			\mathcal{L}(x)-\mathcal{L}(\tilde{x})=\int_{0}^{1} \mathcal{L}'(\tilde{x}+s(x-\tilde{x}))(x-\tilde{x})\,ds.
			\end{equation*}
		Using  the trapezoidal rule
		\begin{equation*}
			\int_{0}^{1}f(s)\,ds =\frac{1}{2}(f(0)+f(1))+ \frac{1}{2} \int_{0}^{1}f''(s)s(s-1)\,ds,
			\end{equation*}
		with $f(s):= \mathcal{L}'(\tilde{x}+s(x-\tilde{x}))(x-\tilde{x})$, we end up with
		\begin{align*}
			\mathcal{L}(x)-\mathcal{L}(\tilde{x}) =& \frac{1}{2}(\mathcal{L}'(x)(x-\tilde{x}) +\mathcal{L}'(\tilde{x})(x-\tilde{x})) + \underbrace{\frac{1}{2} \int_{0}^{1}\frac{d^3\mathcal{L}}{ds^3}(\tilde{x}+s(x-\tilde{x}))(x-\tilde{x}) s(s-1)\,ds}_{=\mathcal{R}^{(3)}}.
			\end{align*}
		By using the definition of $\mathcal{L}$, we obtain that
			\begin{equation*}
				J(u)-J(\tilde{u})=
				\mathcal{L}(x)-\mathcal{L}(\tilde{x}) +\underbrace{\mathcal{A}(u)(z) }_{=0} + \mathcal{A}(\tilde{u})(\tilde{z}) 
				=
				\frac{1}{2}(\mathcal{L}'(x)(x-\tilde{x}) +\mathcal{L}'(\tilde{x})(x-\tilde{x})) +\mathcal{A}(\tilde{u})(\tilde{z})+ \mathcal{R}^{(3)}.
			\end{equation*}
		We note that $\mathcal{L}'(x)(y)=0$ for all $v \in X$. 
		Therefore, the equation from above can be reduced to
		\begin{align*}
			J(u)-J(\tilde{u})=\mathcal{L}(x)-\mathcal{L}(\tilde{x}) +\underbrace{\mathcal{A}(u)(z) }_{=0} + \mathcal{A}(\tilde{u})(\tilde{z}) =& \frac{1}{2}\mathcal{L}'(\tilde{x})(x-\tilde{x}) +\mathcal{A}(\tilde{u})(\tilde{z})+ \mathcal{R}^{(3)}.
		\end{align*}
	Finally, with 
			\begin{align*}
				\mathcal{L}'(\tilde{x})(x-\tilde{x}) = &\underbrace{J'(\tilde{u})(e)-\mathcal{A}'(\tilde{u})(e,\tilde{z})}_{=\rho^*(\tilde{u},\tilde{z})(u-\tilde{u})}\underbrace{-A(\tilde{u})(e^*)}_{=\rho(\tilde{u})(z-\tilde{z})},
				\end{align*}
			we conclude the statement of the theorem.
	\end{proof}
\end{Theorem}
\begin{Remark}
A particular example for $\tilde{u}$ and $\tilde{z}$ are (finite element) approximations of $u_h$ and $v_h$.
\end{Remark}
Theoretically, this error identity already gives us an error estimator of the form
\begin{align} \label{eq: theory_Error_Estimator}
	\begin{split}
		J(u)-J(\tilde{u})&= \eta := \frac{1}{2}\left(\rho(\tilde{u})(z-\tilde{z})+\rho^*(\tilde{u},\tilde{z})(u-\tilde{u}) \right)
		-\rho (\tilde{u})(\tilde{z}) + \mathcal{R}^{(3)}.
	\end{split}
\end{align}
However, this error estimator $\eta$ is not computable, since neither $u$ nor $z$ are known in general.
%
%
\subsection{Error Estimation in Enriched Spaces}\label{SubSec:EEEnriched}
{In this subsection, we derive error estimates in enriched spaces that help us to obtain efficiency and reliability results for the error estimator $\eta$ using a saturation assumption. First, we introduce such} enriched spaces $U_h^{(2)}$ and $V_h^{(2)}$ with the properties $U_h \subseteq U_h^{(2)} \subseteq U$ and $V_h \subseteq V_h^{(2)} \subseteq V$, respectively. For instance, in case of finite elements, this space can be created by using uniform $h$- or uniform $p$-refinement.

The enriched spaces can be used to formulate the enriched primal problem, which is given by: Find $u_h^{(2)} \in U_h^{(2)}$ such that:
\begin{equation}\label{eq: enriched primal}
	\mathcal{A}(u_h^{(2)})(v_h^{(2)})=0 \qquad \forall v_h^{(2)} \in V_h^{(2)}.
\end{equation}
Naturally, this can also be done for the adjoint problem leading to the enriched adjoint problem: Find $z_h^{(2)} \in V_h^{(2)}$ such that
\begin{equation} \label{eq: enriched Adjointproblem}
	\mathcal{A}'(u_h^{(2)})(v_h^{(2)},z_h^{(2)})=J'(u_h^{(2)})(v_h^{(2)}) \qquad \forall v_h^{(2)} \in U_h^{(2)},
\end{equation}
where $u_h^{(2)}$ solves the enriched primal problem \eqref{eq: enriched primal}.
If we replace $u$ and $z$ in the right hand side of \eqref{eq: theory_Error_Estimator} with $u_h^{(2)}$ and $z_h^{(2)}$, we obtain the approximation
\begin{equation} \label{eq: def_eta(2)}
		J(u)-J(\tilde{u})\approx \eta^{(2)}:= \frac{1}{2}\left(\rho(\tilde{u})(z_h^{(2)}-\tilde{z})+\rho^*(\tilde{u},\tilde{z})(u_h^{(2)}-\tilde{u}) \right)
	-\rho (\tilde{u})(\tilde{z}) + \mathcal{R}^{(3)(2)},
\end{equation}
where
\begin{equation}
	\begin{split}	\label{eq: Error Estimator: Enriched Remainderterm}
		\mathcal{R}^{(3)(2)}:=\frac{1}{2}\int_{0}^{1}&[J'''(\tilde{u}+se_h^{(2)})(e_h^{(2)},e_h^{(2)},e_h^{(2)})
		-\mathcal{A}'''(\tilde{u}+se_h^{(2)})(e_h^{(2)},e_h^{(2)},e_h^{(2)},\tilde{z}+se_h^{*(2)}) \\
		&-3\mathcal{A}''(\tilde{u}+se_h^{(2)})(e_h^{(2)},e_h^{(2)},e_h^{*(2)})]s(s-1)\,ds,
	\end{split} 
\end{equation}
with $e_h^{(2)}=u_h^{(2)}-\tilde{u}$ and $e_h^{*(2)} =z_h^{(2)}-\tilde{z}$.
Such an enriched approach has already been used, e.g., in
\cite{BeRa01,BaRa03,endtmayer2021hierarchical,BruSchweBau16,KoeBruBau2019a,RiWi15_dwr,EnLaNeiWoWi2020,endtmayerPAMM2022,ErathGanterPraetorius2018}.
With the next theorem, we obtain some results for the error estimator resulting from the enriched approach. Additionally, the theorem allows us to relax differentiability conditions on the solution variables. Specifically, Fr\'echet-differentiability is only 
required on $U_h^{(2)}$ instead of $U$. For instance, this is important for the parabolic 
regularized $p$-Laplacian in Section~\ref{Subsec: Example4}.
\begin{Theorem}[see \cite{endtmayer2023goal}]\label{thm: Error Representation}
	Let $\mathcal{A}: U \mapsto V^*$ and $J:U\mapsto \mathbb{R}$.
	Moreover, let  $\mathcal{A}^{(2)} \in \mathcal{C}^3(U_h^{(2)},V_h^{(2)*})$ and $J_h \in \mathcal{C}^3(U_h^{(2)},\mathbb{R})$ such that 
	for all $v_h^{(2)},\psi_h^{(2)} \in U_h^{(2)}$ and $\phi_h^{(2)} \in V_h^{(2)}$, 
	the equalities 
	\begin{align}
	\label{eq: A=Ah}
	\mathcal{A}(v_h^{(2)})(\phi_h^{(2)})&=\mathcal{A}^{(2)}(v_h^{(2)})(\phi_h^{(2)}), \\
	\label{eq: A'=Ah'}
	\mathcal{A}'(v_h^{(2)})(\psi_h^{(2)},\phi_h^{(2)})&=(\mathcal{A}^{(2)})'(v_h^{(2)})(\psi_h^{(2)},\phi_h^{(2)}),\\
	\label{eq: J=Jh}
	J(\psi_h^{(2)})&=J_h(\psi_h^{(2)}),\\
	\label{eq: J'=Jh'}
	J'(\psi_h^{(2)})&=J'_h(\psi_h^{(2)}),
	\end{align}
	are fulfilled. 
	Here, $V_h^{(2)*}$ denotes the dual space of $V_h^{(2)}$. 
	Furthermore, let us assume that $J(u) \in \mathbb{R}$, where $u\in U$ solves the model problem \eqref{eq: General Modelproblem}.
	If $J(u) \in \mathbb{R}$, where $u\in U$ solves the model problem \eqref{eq: General Modelproblem}, $u_h^{(2)} \in U_h^{(2)}$ solves the enriched primal problem \eqref{eq: enriched primal} and $z_h^{(2)} \in V_h^{(2)}$ solves the enriched adjoint problem \eqref{eq: enriched Adjointproblem}, then for arbitrary but fixed  $\tilde{u} \in U_h^{(2)}$ and $ \tilde{z} \in V_h^{(2)}$
	the error representation formula
	\begin{align*} \label{eq:Error RepresentationEnriched}
	\begin{split}
	J(u)-J(\tilde{u})&= J(u)-J(u_h^{(2)})+ \frac{1}{2}\left(\rho(\tilde{u})(z_h^{(2)}-\tilde{z})+\rho^*(\tilde{u},\tilde{z})(u_h^{(2)}-\tilde{u}) \right)
	-\rho (\tilde{u})(\tilde{z}) + \mathcal{R}_h^{(3)},
	\end{split}
	\end{align*}
	holds, where
	$\rho(\tilde{u})(\cdot) := -\mathcal{A}(\tilde{u})(\cdot)$ and
	$\rho^*(\tilde{u},\tilde{z})(\cdot) := J'(\tilde{u})-\mathcal{A}'(\tilde{u})(\cdot,\tilde{z})$.
\begin{proof}
	A similar proof for this theorem is given in \cite{endtmayer2023goal}.
	Since $u_h^{(2)}$ solves the enriched primal problem \eqref{eq: enriched primal}, and \eqref{eq: A=Ah} holds, we get 
	\begin{equation*}
	\mathcal{A}(u_h^{(2)})(v_h^{(2)})=\mathcal{A}^{(2)}(u_h^{(2)})(v_h^{(2)})=0  \quad \forall v_h^{(2)} \in V_h^{(2)}.
	\end{equation*}
	For $z_h^{(2)}$ solving the enriched adjoint problem \eqref{eq: enriched Adjointproblem}, we conclude in combination with \eqref{eq: A'=Ah'} and \eqref{eq: J'=Jh'} that
	\begin{equation*}
	J'(u_h^{(2)})(v_h^{(2)})-\mathcal{A}'(u_h^{(2)})(v_h^{(2)},z_h^{(2)})= J'_h(u_h^{(2)})(v_h^{(2)})-(\mathcal{A}^{(2)})'(u_h^{(2)})(v_h^{(2)},z_h^{(2)})=0 \quad \forall v_h^{(2)} \in U_h^{(2)}.
	\end{equation*}
	This allows us to apply Theorem \ref{thm: error identity} with $u=u_h^{(2)}$, $z=z_h^{(2)}$, $\mathcal{A}=\mathcal{A}^{(2)}$ and $J=J_h$. 
	Therefore, we obtain 
	\begin{equation}
	\label{eq: pure discrete result}
	J_h(u_h^{(2)})-J_h(\tilde{u})= \frac{1}{2}\left(\rho_h(\tilde{u})(z_h^{(2)}-\tilde{z})+\rho^*_h(\tilde{u},\tilde{z})(u_h^{(2)}-\tilde{u})\right)
	-\rho_h (\tilde{u})(\tilde{z}) + \mathcal{R}_h^{(3)},
	\end{equation}
	where $\rho_h(\tilde{u})(\cdot) := -\mathcal{A}^{(2)}(\tilde{u})(\cdot)$, 
	$\rho^*_h(\tilde{u},\tilde{z})(\cdot) := J'_h(\tilde{u})-(\mathcal{A}^{(2)})'(\tilde{u})(\cdot,\tilde{z})$, and 
	\begin{equation*}
	\begin{split}
	\mathcal{R}_h^{(3)}:=\frac{1}{2}\int_{0}^{1}&[J_h'''(\tilde{u}+se_h^{(2)})(e_h^{(2)},e_h^{(2)},e_h^{(2)})
	-(\mathcal{A}^{(2)})'''(\tilde{u}+se_h^{(2)})(e_h^{(2)},e_h^{(2)},e_h^{(2)},\tilde{z}+se_h^{*(2)}) \\
	&-3(\mathcal{A}^{(2)})''(\tilde{u}+se_h^{(2)})(e_h^{(2)},e_h^{(2)},e_h^{*(2)})]s(s-1)\,ds,
	\end{split} 
	\end{equation*}
	with $e_h^{(2)}=u_h^{(2)}-\tilde{u}$ and $e_h^{*(2)} =z_h^{(2)}-\tilde{z}$.	

	Again using \eqref{eq: A=Ah}, \eqref{eq: A'=Ah'} and \eqref{eq: J'=Jh'}, we notice that $\rho_h=\rho$ and $\rho_h^*=\rho^*$ on the enriched spaces $U_h^{(2)}$ and $V_h^{(2)}$.  
	This leads to 
	the representation
	\begin{align*}
	J_h(u_h^{(2)})-J_h(\tilde{u})&= \frac{1}{2}\left(\rho_h(\tilde{u})(z_h^{(2)}-\tilde{z})+\rho^*_h(\tilde{u},\tilde{z})(u_h^{(2)}-\tilde{u})\right)
	-\rho_h (\tilde{u})(\tilde{z}) + \mathcal{R}_h^{(3)}\\
	&= \frac{1}{2}\left(\rho(\tilde{u})(z_h^{(2)}-\tilde{z})+\rho^*(\tilde{u},\tilde{z})(u_h^{(2)}-\tilde{u})\right)
	-\rho (\tilde{u})(\tilde{z}) + \mathcal{R}_h^{(3)}.
	\end{align*}
	In combination with \eqref{eq: pure discrete result} and \eqref{eq: J=Jh}, we can show
	\begin{align*}
	J(u)-J(\tilde{u})&=J(u)-	J(u_h^{(2)})+	J_h(u_h^{(2)})-J_h(\tilde{u})\\
	&=J(u)-	J(u_h^{(2)})+\frac{1}{2}\left(\rho(\tilde{u})(z_h^{(2)}-\tilde{z})+\rho^*(\tilde{u},\tilde{z})(u_h^{(2)}-\tilde{u})\right) 
	-\rho (\tilde{u})(\tilde{z}) + \mathcal{R}_h^{(3)},
		\end{align*}
	which completes the proof of the theorem.
\end{proof}
\end{Theorem}
\begin{Remark}
	With Theorem \ref{thm: Error Representation}, we know that it is sufficient to fulfill the differentiability conditions  $\mathcal{A}^{(2)} \in \mathcal{C}^3(U_h^{(2)},V_h^{(2)*})$ and $J_h \in \mathcal{C}^3(U_h^{(2)},\mathbb{R})$ instead of $\mathcal{A} \in \mathcal{C}^3(U,V^*)$ and $J \in \mathcal{C}^3(U,\mathbb{R})$. For instance, point evaluations are well defined on $U_h^{(2)}=Q_2$, but not in general on $U$. Additionally, the existence of a solution of the adjoint problem \eqref{eq: General Modelproblem} is not required anymore. Instead the existence of the solution $z_h^{(2)}\in V_h^{(2)}$ of the enriched adjoint problem \eqref{eq: enriched Adjointproblem} is mandatory.
\end{Remark}
\begin{Assumption}[Saturation Assumption]\label{assump: Saturation}
	Let $u_h^{(2)} \in U_h^{(2)}$ be the solution of the enriched primal problem \eqref{eq: enriched primal}. Then there exist $b_0\in (0,1)$ and $b_h\in (0,b_0)$ such that
	\begin{equation}
	|J(u)-J(u_h^{(2)})| \leq b_h \, |J(u)-J(\tilde{u})|,
	\end{equation}
	holds.
\end{Assumption}
Unfortunately, we are not aware of a general 
technique to verify
the saturation assumption for goal functionals.
However, it is a very common assumption in hierarchical based error estimation; 
see, e.g., 
\cite{BankWeiser1985,BoErKor1996,BankSmith1993,Verfuerth:1996a}. 
In the works \cite{BoErKor1996,EnLaWi20}, it was shown that the saturation assumption can fail. 
However, for specific functionals and PDEs, there are proofs for the saturation assumption; 
see, 
e.g., \cite{Rossi2002,Agouzal2002,DoerflerNochetto2002,AchAchAgou2004,CaGaGed16,BankParsaniaSauter2013,FerrazOrtnerPraetorius2010,ErathGanterPraetorius2018,KimKim2014}.
\begin{Definition}[Efficient and Reliable]
We say an error estimator $\eta$ is \textbf{efficient} with respect to $J$, if there exist a constant $\underline{c} \in \mathbb{R}$ with $\underline{c}>0$ such that
\begin{equation} \label{eq: efficient}
\underline{c}|\eta| \leq |J(u)-J(\tilde{u})|.
\end{equation}
We say an error estimator $\eta$ is \textbf{reliable} with respect to $J$, if there exist a constant $\overline{c} \in \mathbb{R}$ with $\overline{c}>0$ such that
\begin{equation} \label{eq: reliable}
\overline{c}|\eta| \geq |J(u)-J(\tilde{u})|.
\end{equation}
\end{Definition}
\begin{Theorem} \label{thm: Efficiency and Reliability}
	Let Assumption~\ref{assump: Saturation} be satisfied. Additionally let all assumptions of Theorem~\ref{thm: Error Representation} be fulfilled.
	Then the error estimator $\eta^{(2)}$ defined in \eqref{eq: def_eta(2)}, i.e.
	\begin{equation*}
	\eta^{(2)}:=\frac{1}{2}\left(\rho(\tilde{u})(z_h^{(2)}-\tilde{z})+\rho^*(\tilde{u},\tilde{z})(u_h^{(2)}-\tilde{u}) \right)
	-\rho (\tilde{u})(\tilde{z}) + \mathcal{R}^{(3)(2)},
	\end{equation*}
	is efficient and reliable with the constants 
	$\underline{c}=1/(1+b_h)$ and $\overline{c}=1/(1-b_h)$.
\end{Theorem}
	\begin{proof}
		From Theorem~\ref{thm: Error Representation}, we get that 
		\begin{align*}
		\begin{split}
		J(u)-J(\tilde{u})&= J(u)-J(u_h^{(2)})+ \frac{1}{2}\left(\rho(\tilde{u})(z_h^{(2)}-\tilde{z})+\rho^*(\tilde{u},\tilde{z})(u_h^{(2)}-\tilde{u}) \right)
		-\rho (\tilde{u})(\tilde{z}) + \mathcal{R}_h^{(3)},
		\end{split}
		\end{align*}
		which is equivalent to 
		\begin{align}
		\begin{split}
		J(u)-J(\tilde{u})&= J(u)-J(u_h^{(2)})+ \eta^{(2)}. \label{help: show efficiency and reliability}
		\end{split}
		\end{align}
		Equation \eqref{help: show efficiency and reliability} implies that 
		\begin{align}
		\begin{split}
		|J(u)-J(\tilde{u})|&= |J(u)-J(u_h^{(2)})+ \eta^{(2)}|. \label{help: show efficiency and reliability abs}
		\end{split}
		\end{align}
		Let us first proof that $\eta^{(2)}$ is reliable.
		Since the saturation assumption is valid, we know that $|J(u)-J(u_h^{(2)})| \leq b_h |J(u)-J(\tilde{u})|$ holds. Combining this with \eqref{help: show efficiency and reliability abs} we obtain		
                  \begin{align*}		
		|J(u)-J(\tilde{u})|&= |J(u)-J(u_h^{(2)})+ \eta^{(2)}|
		\leq |J(u)-J(u_h^{(2)})|+ |\eta^{(2)}|
		\leq  b_h |J(u)-J(\tilde{u})|+|\eta^{(2)}|.
		\end{align*}
		We finally get 
		\begin{align*}
		(1-b_h)|J(u)-J(\tilde{u})|\leq |\eta^{(2)}|,
		\end{align*}
		and consequently
		\begin{align*}
		|J(u)-J(\tilde{u})|\leq \frac{1}{1-b_h}|\eta^{(2)}|.
		\end{align*}		
	        Thus, reliability of $\eta^{(2)}$ follows with
		the constant $\overline{c}=1/(1-b_h)$.		
		Now let us prove
		that $\eta^{(2)}$ is efficient as well.
		We start again with \eqref{help: show efficiency and reliability abs}:
                \begin{align*}
		|J(u)-J(\tilde{u})| = |J(u)-J(u_h^{(2)})+ \eta^{(2)}|
		\geq  |\eta^{(2)}|-|J(u)-J(u_h^{(2)})|
		\geq  |\eta^{(2)}|-b_h |J(u)-J(\tilde{u})|.
		\end{align*}
		This leads to the estimates
		\begin{align*}
		(1+b_h)|J(u)-J(\tilde{u})|\geq |\eta^{(2)}|,
		\end{align*}
		and 
		\begin{align*}
		|J(u)-J(\tilde{u})|\geq \frac{1}{1+b_h}|\eta^{(2)}|.
		\end{align*}
		This proves the efficiency of $\eta^{(2)}$ with 
		the constant $\underline{c}=1/(1+b_h)$.
	\end{proof}
%
%
\subsection{Parts of the Error Estimator for Enriched Spaces}
If the saturation assumption is fulfilled, the previous theorem shows that the error estimator $\eta^{(2)}$ defined in \eqref{eq: def_eta(2)} is efficient and reliable. 
Furthermore, this error estimator is also computable and can be employed for measuring discretization and iteration errors.
In this subsection, we investigate the different parts of $\eta^{(2)}$ in more detail.
We split $\eta^{(2)}$ in the following way
\begin{equation}
\eta^{(2)}:=\underbrace{\frac{1}{2}\left(\rho(\tilde{u})(z_h^{(2)}-\tilde{z})+\rho^*(\tilde{u},\tilde{z})(u_h^{(2)}-\tilde{u}) \right)}_{:=\eta_h^{(2)}}
\underbrace{-\rho (\tilde{u})(\tilde{z})}_{:=\eta_k} + \underbrace{\mathcal{R}_h^{(3)}}_{:=\eta_{\mathcal{R}}^{(2)}}.
\end{equation}
%
%
\subsubsection{The Remainder Part $\eta_{\mathcal{R}}^{(2)}$ or Linearization Error Estimator.}
This part is of higher order, and usually is neglected in literature. In this work, we will neglect the term as well since it is of higher order regarding the error, and it is connected with high computational cost (depending on the problem). 
In \cite{EnLaWi20}, it was shown that the part indeed is of higher order and can be neglected 
for $p$-Laplace and the Navier-Stokes equations. However, this might not be true in general, as indicated in \cite{granzow2023linearization}.

%
%
\subsubsection{The Iteration Error Estimator $\eta_k$} \label{SubSubSec: etak}
The part 
\begin{equation}
\eta_k:=-\rho (\tilde{u})(\tilde{z}),
\end{equation}
represents the iteration error of the linear or also non-linear solver as presented in \cite{RaWeWo10,RanVi2013}.  It is the only part of the error estimator which does not depend on the enriched solutions $u_h^{(2)}$ and $z_h^{(2)}$.  Furthermore, $\eta_k$ vanishes if $\tilde{u}$ is the solution of the discrete primal problem \eqref{eq: discrete primal}. 
In many algorithms, it is used to stop the non-linear solver. 
A downside is that, for Newton's method, the adjoint solution $\tilde{z}$ is required in every Newton step.
\begin{Theorem}(Representation of Iteration error; see \cite{EnLaWi20}) \label{thm: Representation of Iteration error}
	Let us assume that $\tilde{z}$ solves 
	\begin{equation} \label{help: adjoint disrcrete}
	\mathcal{A}'(\tilde{u})(\tilde{v},\tilde{z})=J'(\tilde{u})(\tilde{v}) \qquad \forall \tilde{v} \in U_h,
	\end{equation}
	and let $\delta \tilde{u}$ solve 
	\begin{equation} \label{help: newton disrcrete}
	\mathcal{A}'(\tilde{u})(\delta\tilde{u},\tilde{v})=-\mathcal{A}(\tilde{u})(\tilde{v}) \qquad \forall \tilde{v} \in V_h.
	\end{equation}
	Then we have the representation
	\begin{equation}
	\rho(\tilde{u},\tilde{z})=J'(\tilde{u})(\delta\tilde{u})
	\end{equation}
\end{Theorem}	
\begin{proof}
		From \eqref{help: adjoint disrcrete}, \eqref{help: newton disrcrete} and the definition of $\rho$ in Theorem~\ref{SubSec:Error-Identity} and Theorem~\ref{thm: Error Representation} it immediately follows that		
		\begin{equation*}
		 \rho(\tilde{u},\tilde{z})=-\mathcal{A}(\tilde{u})(\tilde{z})=\mathcal{A}'(\tilde{u})(\delta \tilde{u},\tilde{z})=J'(\tilde{u})(\delta\tilde{u}).
		\end{equation*}
		This already concludes the proof. 
\end{proof}
This identity allows us to use the next Newton update instead of the adjoint solution $\tilde{z}$, reducing the number of required linear solves from $2n$ to $n+1$, where $n$ is the number of Newton steps. 
Naturally, this final Newton update $\delta \tilde{u}$ can be used to update the solution as well.
%
%
\subsubsection{The Discretization Error Estimator $\eta_h^{(2)}$}
The part 
\begin{equation} \label{eq: def primal and adjoint error estimator}
\eta_h^{(2)}:=\frac{1}{2}\bigl(\underbrace{\rho(\tilde{u})(z_h^{(2)}-\tilde{z})}_{:=\eta_{h,p}^{(2)}}+\underbrace{\rho^*(\tilde{u},\tilde{z})(u_h^{(2)}-\tilde{u})}_{:=\eta_{h,a}^{(2)}} \bigr),
\end{equation}
represents the discretization error of the error estimator as discussed in \cite{RanVi2013,EnLaWi20,endtmayer2020multi}. 
We would like to mention that, for linear goal functionals $J$ and affine linear operators $\mathcal{A}$, the primal part of the discretization error estimator $\eta_{h,p}$,
  and the adjoint part $\eta_{h,a}$, both defined in \eqref{eq: def primal and adjoint error estimator}, coincide. 
Later on in Subsection~\ref{SubSec:Localization}, this part will be localized and used to adapt the mesh.
\begin{Lemma} \label{lemma: eta2 to etah}
	Let $\eta^{(2)}$ be defined as in \eqref{eq: def_eta(2)} and let $\eta_h^{(2)}$ be defined as in \eqref{eq: def primal and adjoint error estimator}. Then
	\begin{equation}
	|\eta_h^{(2)}| - |\eta_k + \eta_{\mathcal{R}}^{(2)}|\leq|\eta^{(2)}|\leq|\eta_h^{(2)}| + |\eta_k + \eta_{\mathcal{R}}^{(2)}|.
	\end{equation}
\end{Lemma}	
	\begin{proof}
		We know that 		
		$|\eta^{(2)}|=|\eta_h^{(2)} + \eta_k + \eta_{\mathcal{R}}^{(2)}|$.		
		Therefore, 		
		we arrive at the estimates
		\begin{equation*}
		|\eta^{(2)}|\leq|\eta_h^{(2)}| + |\eta_k + \eta_{\mathcal{R}}^{(2)}|,
		\end{equation*}
		and
		\begin{equation*}
		|\eta^{(2)}|\geq|\eta_h^{(2)}| - |\eta_k + \eta_{\mathcal{R}}^{(2)}|,
		\end{equation*}
		which concludes the proof.
	\end{proof}

\begin{Theorem}[Almost Efficiency and Reliability]\label{thm: etah Efficiency and Reliability}
	The two-side discretization error estimate
	\begin{equation*}
	\frac{1}{1+b_h}\left(|\eta_h^{(2)}|-|\eta_k + \eta_{\mathcal{R}}^{(2)}|\right)\leq |J(u)-J(\tilde{u})| \leq\frac{1}{1-b_h}\left(|\eta_h^{(2)}|+|\eta_k + \eta_{\mathcal{R}}^{(2)}|\right)
	\end{equation*}
	is valid provided the saturation assumption \eqref{assump: Saturation} holds.
	Furthermore, if there exists a $\alpha_{\eta_h^{(2)}} \in (0,1)$ with
	\begin{equation}
	|\eta_k + \eta_{\mathcal{R}}^{(2)}| \leq \alpha_{\eta_h^{(2)}} |\eta_h^{(2)}|,
	\end{equation}
	then the discretization error estimator is efficient and reliable, i. e.
	\begin{equation}
	\underline{c}_h|\eta_h^{(2)}|\leq |J(u)-J(\tilde{u})| \leq\overline{c}_h|\eta_h^{(2)}|,
	\end{equation}
	with $\underline{c}_h:=(1-\alpha_{\eta_h^{(2)}})\frac{1}{1+b_h}$ and $\overline{c}_h:=(1+\alpha_{\eta_h^{(2)}})\frac{1}{1+b_h}$.
\end{Theorem}
	\begin{proof}
		From Theorem~\ref{thm: Efficiency and Reliability}, we know
		that $$\frac{1}{1+b_h}|\eta^{(2)}|\leq |J(u)-J(\tilde{u})|\leq \frac{1}{1-b_h}|\eta^{(2)}|,$$
		and, from Lemma~\ref{lemma: eta2 to etah},
		we know that 
		\begin{equation*}
		|\eta_h^{(2)}| - |\eta_k + \eta_{\mathcal{R}}^{(2)}|\leq|\eta^{(2)}|\leq|\eta_h^{(2)}| + |\eta_k + \eta_{\mathcal{R}}^{(2)}|.
		\end{equation*}
		Combining these two results leads to 
		\begin{align}
		\frac{1}{1+b_h}\left(|\eta_h^{(2)}|-|\eta_k + \eta_{\mathcal{R}}^{(2)}|\right)&\leq\frac{1}{1+b_h}|\eta^{(2)}|\nonumber\\
		 &\leq |J(u)-J(\tilde{u})| \label{eq: thm: inequality in proof}\\
		 &\leq \frac{1}{1-b_h}|\eta^{(2)}|\leq\frac{1}{1-b_h}\left(|\eta_h^{(2)}|+|\eta_k + \eta_{\mathcal{R}}^{(2)}|\right).\nonumber
		\end{align}
		This is the first statement of the theorem.
		
		From 
		\begin{equation}
		|\eta_k + \eta_{\mathcal{R}}^{(2)}| \leq \alpha_{\eta_h^{(2)}} |\eta_h^{(2)}|,
		\end{equation}
		we can deduce that
		\begin{equation*}
		|\eta_h^{(2)}|-|\eta_k + \eta_{\mathcal{R}}^{(2)}| \geq (1-\alpha_{\eta_h^{(2)}})|\eta_h^{(2)}|,
		\end{equation*}
		and 
		\begin{equation*}
		|\eta_h^{(2)}|+|\eta_k + \eta_{\mathcal{R}}^{(2)}| \leq (1+\alpha_{\eta_h^{(2)}})|\eta_h^{(2)}|.
		\end{equation*}
		Combining this with \eqref{eq: thm: inequality in proof}
		we obtain
		\begin{equation*}
		(1-\alpha_{\eta_h^{(2)}})\frac{1}{1+b_h}|\eta_h^{(2)}|\leq |J(u)-J(\tilde{u})| \leq(1+\alpha_{\eta_h^{(2)}})\frac{1}{1-b_h}|\eta_h^{(2)}|,
		\end{equation*}
		which provides us the second statement.
	\end{proof}
This shows that, for the discretization error estimator, we also get efficiency and reliability up to a higher-order term and a term which can be controlled by the accuracy of the non-linear solver. 
Additionally, if these two terms can be bounded by the inequality 
\begin{equation*}
|\eta_k + \eta_{\mathcal{R}}^{(2)}| \leq \alpha_{\eta_h^{(2)}} |\eta_h^{(2)}|,
\end{equation*}
then $|\eta_h^{(2)}|$ is an efficient and reliable error estimator as well.
%
%
\subsection{Effectivity Indices for Enriched Spaces}
\label{sec_Ieff_enriched_spaces}
In this subsection, we introduce and investigate effectivity 
indices. These were introduced in \cite{BaRei78} in order to measure
how well the error estimator approximates the true error. 
Ideally, the effectivity index approaches $1$ asymptotically under mesh 
refinement. However, due to cancellations with contributions with 
different signs, a (stronger) quality measure for mesh refinement 
utilizing the triangle inequality resulting into the indicator 
index was introduced in \cite{RiWi15_dwr}. As in both quality measures, the 
true error enters, i.e., the unknown solution $u$, either academic 
examples with known $u$ are taken, or numerical solutions obtained on highly 
refined meshes (in case the available computational power allows us to do so). 
Clearly, for complicated non-linear, coupled problems and multiphysics 
problems, the goal must be to have a previously tested error estimator, which 
is reliable and efficient, which is then applied to such complicated situations without the need to again measure effectivity and indicator indices.

\begin{Theorem}[Bounds on the effectivity index] \label{thm: Bounds: Ieff}
	We assume that $|J(u)-J(\tilde{u})|\not=0$.
	If
	the assumptions of Theorem~\ref{thm: Efficiency and Reliability} are fulfilled, 
	then, for the effectivity index $I_{\mathrm{eff},+}$ defined by
	\begin{equation} \label{def: Ieff+}
	I_{\mathrm{eff},+}:= \frac{\eta^{(2)}}{J(u)-J(\tilde{u})},
	\end{equation}
	we have the bounds
	\begin{equation} \label{eq: Ieff+ bounds}
	1-b_h\leq |I_{\mathrm{eff},+}| \leq 1+b_h.
	\end{equation}
	If the assumptions of Theorem~\ref{thm: etah Efficiency and Reliability} are fulfilled, 
	then, for the effectivity index $I_{\mathrm{eff}}$ defined by
	\begin{equation} \label{def: Ieff}
	I_{\mathrm{eff}}:= \frac{\eta_h^{(2)}}{J(u)-J(\tilde{u})},
	\end{equation}
	we have the bounds
	\begin{equation} \label{eq: Ieff bounds}
	\frac{1-b_h}{1+\alpha_{\eta_h^{(2)}}}\leq |I_{\mathrm{eff}}| \leq \frac{1+b_h}{1-\alpha_{\eta_h^{(2)}}}.
	\end{equation}
\end{Theorem}	
	\begin{proof}
		Here, we follow the ideas in \cite{EnLaWi20,endtmayer2020multi}.
		Theorem~\ref{thm: Efficiency and Reliability} provides the result
		\begin{align*}
		\frac{1}{1+b_h}|\eta^{(2)}|\leq|J(u)-J(\tilde{u})|\leq \frac{1}{1-b_h}|\eta^{(2)}|.
		\end{align*}
		Now we can divide the inequality from above by  $|J(u)-J(\tilde{u})|$, which leads to
		\begin{align*}
		\frac{1}{1+b_h}\left|\frac{\eta^{(2)}}{J(u)-J(\tilde{u})}\right|\leq 1 \leq \frac{1}{1-b_h}\left|\frac{\eta^{(2)}}{J(u)-J(\tilde{u})}\right|.
		\end{align*}
		From this, we can easily see the estimates
		\begin{equation*}
		1-b_h\leq |I_{\mathrm{eff},+}| \leq 1+b_h.
		\end{equation*}
		The second statement follows from Theorem~\ref{thm: etah Efficiency and Reliability}, i.e.
		\begin{equation*}
		(1-\alpha_{\eta_h^{(2)}})\frac{1}{1+b_h}|\eta_h^{(2)}|\leq |J(u)-J(\tilde{u})| \leq(1+\alpha_{\eta_h^{(2)}})\frac{1}{1-b_h}|\eta_h^{(2)}|,
		\end{equation*}
		where, by the same argument as above, we get
				\begin{align*}
		\frac{1-\alpha_{\eta_h^{(2)}}}{1+b_h}\left|\frac{\eta^{(2)}}{J(u)-J(\tilde{u})}\right|\leq 1 \leq \frac{1+\alpha_{\eta_h^{(2)}}}{1-b_h}\left|\frac{\eta^{(2)}}{J(u)-J(\tilde{u})}\right|.
		\end{align*}
		This is equivalent to 
		\begin{equation*} 
		\frac{1-b_h}{1+\alpha_{\eta_h^{(2)}}}\leq |I_{\mathrm{eff}}| \leq \frac{1+b_h}{1-\alpha_{\eta_h^{(2)}}}.
		\end{equation*}
	\end{proof}
Additionally to the effectivity indices above, we define the primal effectivity indices $I_{\mathrm{eff},p}$ and adjoint effectivity indices $I_{\mathrm{eff},a}$ as
		\begin{equation}
		I_{\mathrm{eff},p}:= \frac{\eta_{h,p}^{(2)}}{J(u)-J(\tilde{u})} \qquad \text{and} \qquad I_{\mathrm{eff},a}:= \frac{\eta_{h,a}^{(2)}}{J(u)-J(\tilde{u})}.
		\end{equation}

%
%
\subsection{Error Estimation using Interpolation Techniques}
\label{SubSec:Interpolation}
In Subsection~\ref{SubSec:EEEnriched}, we replaced $u$ and $z$ in \eqref{eq: theory_Error_Estimator} by solutions on the enriched space. In this section we will investigate replacing $u$ and $z$ by some arbitrary interpolations 
{$I_{h,u}: U_h \mapsto  U_h^{(2)}$}
and  
{$I_{h,z}: V_h \mapsto  V_h^{(2)}$} 
of $\tilde{u}$ and $\tilde{z}$, respectively. One such interpolation is presented in the work \cite{BeRa01,RanVi2013}. For instance, for $Q_1$ or $P_1$ elements, the nodes coincide with the nodes of $Q_2$ and $P_2$ finite elements on a coarser mesh, see Figure~\ref{fig: Interpol Q1} and Figure~\ref{fig: Interpol P1} for visualization, respectively.
\begin{figure} [H]
	\centering
	\begin{minipage}{0.4\linewidth}
		\definecolor{zzttqq}{rgb}{0.6,0.2,0}
\definecolor{xdxdff}{rgb}{0.49019607843137253,0.49019607843137253,1}
\definecolor{ududff}{rgb}{0.30196078431372547,0.30196078431372547,1}
\definecolor{uuuuuu}{rgb}{0.26666666666666666,0.26666666666666666,0.26666666666666666}
\begin{tikzpicture}[line cap=round,line join=round,>=triangle 45,x=2cm,y=2cm]
\clip(-0.1,-0.1) rectangle (2.1,2.1);
\fill[line width=2pt,color=zzttqq,fill=zzttqq,fill opacity=0.10000000149011612] (0,0) -- (2,0) -- (2,2) -- (0,2) -- cycle;
\fill[line width=2pt,color=zzttqq,fill=zzttqq,fill opacity=0.0] (0,1) -- (1,1) -- (1,2) -- (0,2) -- cycle;
\fill[line width=2pt,color=zzttqq,fill=zzttqq,fill opacity=0.0] (1,1) -- (2,1) -- (2,2) -- (1,2) -- cycle;
\fill[line width=2pt,color=zzttqq,fill=zzttqq,fill opacity=0.0] (0,0) -- (1,0) -- (1,1) -- (0,1) -- cycle;
\fill[line width=2pt,color=zzttqq,fill=zzttqq,fill opacity=0.0] (1,0) -- (2,0) -- (2,1) -- (1,1) -- cycle;
\draw [line width=2pt] (0,2)-- (2,2);
\draw [line width=2pt] (2,2)-- (2,0);
\draw [line width=2pt] (2,0)-- (0,0);
\draw [line width=2pt] (0,0)-- (0,2);
\draw [line width=2pt] (0,1)-- (2,1);
\draw [line width=2pt] (1,2)-- (1,0);
\draw [line width=2pt,color=uuuuuu] (0,0)-- (2,0);
\draw [line width=2pt,color=uuuuuu] (2,0)-- (2,2);
\draw [line width=2pt,color=uuuuuu] (2,2)-- (0,2);
\draw [line width=2pt,color=uuuuuu] (0,2)-- (0,0);
\draw [line width=2pt,color=uuuuuu] (0,1)-- (1,1);
\draw [line width=2pt,color=uuuuuu] (1,1)-- (1,2);
\draw [line width=2pt,color=uuuuuu] (1,2)-- (0,2);
\draw [line width=2pt,color=uuuuuu] (0,2)-- (0,1);
\draw [line width=2pt,color=uuuuuu] (1,1)-- (2,1);
\draw [line width=2pt,color=uuuuuu] (2,1)-- (2,2);
\draw [line width=2pt,color=uuuuuu] (2,2)-- (1,2);
\draw [line width=2pt,color=uuuuuu] (1,2)-- (1,1);
\draw [line width=2pt,color=uuuuuu] (0,0)-- (1,0);
\draw [line width=2pt,color=uuuuuu] (1,0)-- (1,1);
\draw [line width=2pt,color=uuuuuu] (1,1)-- (0,1);
\draw [line width=2pt,color=uuuuuu] (0,1)-- (0,0);
\draw [line width=2pt,color=uuuuuu] (1,0)-- (2,0);
\draw [line width=2pt,color=uuuuuu] (2,0)-- (2,1);
\draw [line width=2pt,color=uuuuuu] (2,1)-- (1,1);
\draw [line width=2pt,color=uuuuuu] (1,1)-- (1,0);
\begin{scriptsize}
\draw [fill=xdxdff] (0,0) circle (2.5pt);
\draw [fill=xdxdff] (2,2) circle (2.5pt);
\draw [fill=xdxdff] (0,2) circle (2.5pt);
\draw [fill=xdxdff] (2,0) circle (2.5pt);
\draw [fill=xdxdff] (1,1) circle (2.5pt);
\draw [fill=xdxdff] (0,1) circle (2.5pt);
\draw [fill=xdxdff] (2,1) circle (2.5pt);
\draw [fill=xdxdff] (1,2) circle (2.5pt);
\draw [fill=xdxdff] (1,0) circle (2.5pt);
\draw [fill=xdxdff] (2,2) circle (2.5pt);
\draw [fill=xdxdff] (0,2) circle (2.5pt);
\draw [fill=xdxdff] (1,2) circle (2.5pt);
\draw[color=uuuuuu] (1.5,1.5) node {$Q_1$};
\draw[color=uuuuuu] (0.5,1.5) node {$Q_1$};
\draw[color=uuuuuu] (1.5,0.5) node {$Q_1$};
\draw[color=uuuuuu] (0.5,0.5) node {$Q_1$};
\draw [fill=xdxdff] (0,2) circle (2.5pt);
\draw [fill=xdxdff] (2,2) circle (2.5pt);
\draw [fill=xdxdff] (1,2) circle (2.5pt);
\draw [fill=xdxdff] (1,1) circle (2.5pt);
\draw [fill=xdxdff] (0,1) circle (2.5pt);
\draw [fill=xdxdff] (2,1) circle (2.5pt);
\draw [fill=xdxdff] (1,1) circle (2.5pt);
\end{scriptsize}
\end{tikzpicture} 
	\end{minipage} 	{\centering\huge $\rightarrow$} \hfill
	\begin{minipage}{0.4\linewidth}
		\definecolor{zzttqq}{rgb}{0.6,0.2,0}
\definecolor{xdxdff}{rgb}{0.49019607843137253,0.49019607843137253,1}
\definecolor{ududff}{rgb}{0.30196078431372547,0.30196078431372547,1}
\definecolor{uuuuuu}{rgb}{0.26666666666666666,0.26666666666666666,0.26666666666666666}
\begin{tikzpicture}[line cap=round,line join=round,>=triangle 45,x=2cm,y=2cm]
\clip(-0.1,-0.1) rectangle (2.1,2.1);
\fill[line width=2pt,color=green,fill=green,fill opacity=0.1] (0,0) -- (2,0) -- (2,2) -- (0,2) -- cycle;
\draw [line width=2pt] (0,2)-- (2,2);
\draw [line width=2pt] (2,2)-- (2,0);
\draw [line width=2pt] (2,0)-- (0,0);
\draw [line width=2pt] (0,0)-- (0,2);
\begin{scriptsize}
\draw [fill=xdxdff] (0,0) circle (2.5pt);
\draw [fill=xdxdff] (2,2) circle (2.5pt);
\draw [fill=xdxdff] (0,2) circle (2.5pt);
\draw [fill=xdxdff] (2,0) circle (2.5pt);
\draw [fill=xdxdff] (1,1) circle (2.5pt);
\draw [fill=xdxdff] (0,1) circle (2.5pt);
\draw [fill=xdxdff] (2,1) circle (2.5pt);
\draw [fill=xdxdff] (1,2) circle (2.5pt);
\draw [fill=xdxdff] (1,0) circle (2.5pt);
\draw [fill=xdxdff] (2,2) circle (2.5pt);
\draw [fill=xdxdff] (0,2) circle (2.5pt);
\draw [fill=xdxdff] (1,2) circle (2.5pt);
\draw[color=uuuuuu] (1.25,1.25) node {$Q_2$};
\draw [fill=xdxdff] (0,2) circle (2.5pt);
\draw [fill=xdxdff] (2,2) circle (2.5pt);
\draw [fill=xdxdff] (1,2) circle (2.5pt);
\draw [fill=xdxdff] (1,1) circle (2.5pt);
\draw [fill=xdxdff] (0,1) circle (2.5pt);
\draw [fill=xdxdff] (2,1) circle (2.5pt);
\draw [fill=xdxdff] (1,1) circle (2.5pt);
\end{scriptsize}
\end{tikzpicture}	
	\end{minipage}
\caption{Visualization of how the degrees of freedom are interpolated on elements 
of higher order with a coarser grid for $Q_1$ finite elements. \label{fig: Interpol Q1}}
\end{figure}
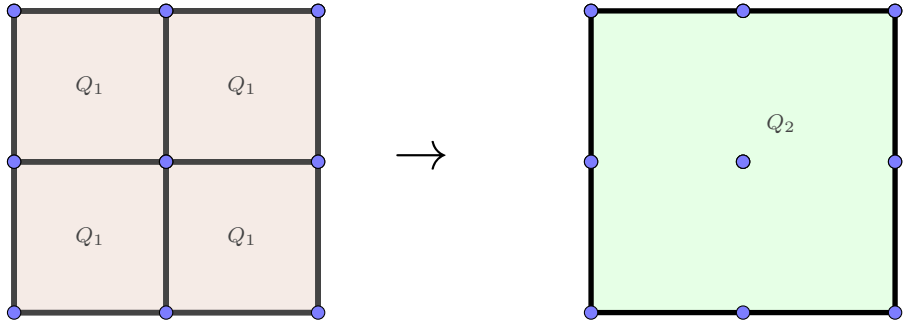

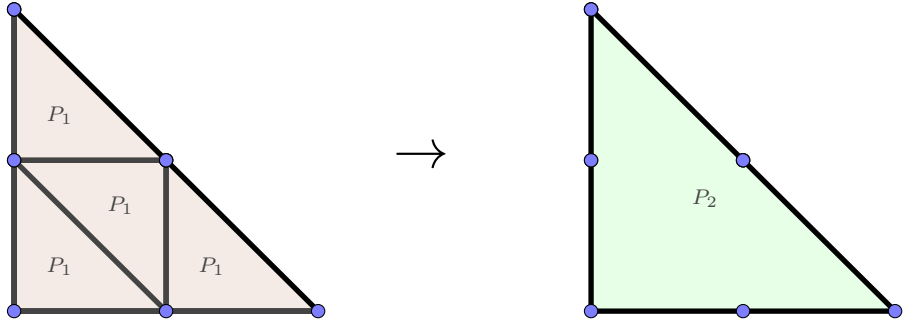
\begin{figure} [H]
	\centering
	\begin{minipage}{0.4\linewidth}
		\definecolor{zzttqq}{rgb}{0.6,0.2,0}
\definecolor{xdxdff}{rgb}{0.49019607843137253,0.49019607843137253,1}
\definecolor{ududff}{rgb}{0.30196078431372547,0.30196078431372547,1}
\definecolor{uuuuuu}{rgb}{0.26666666666666666,0.26666666666666666,0.26666666666666666}
\begin{tikzpicture}[line cap=round,line join=round,>=triangle 45,x=2cm,y=2cm]
\clip(-0.1,-0.1) rectangle (2.1,2.1);
\fill[line width=2pt,color=zzttqq,fill=zzttqq,fill opacity=0.1] (0,0) -- (2,0)  -- (0,2) -- cycle;
\draw [line width=2pt] (2,0)-- (0,0);
\draw [line width=2pt] (0,0)-- (0,2);
\draw [line width=2pt] (2,0)-- (0,2);
\fill[line width=2pt,color=zzttqq,fill=zzttqq,fill opacity=0.0] (0,0) -- (1,0) -- (1,1) -- (0,1) -- cycle;
\draw [line width=2pt,color=uuuuuu] (0,0)-- (2,0);
\draw [line width=2pt,color=uuuuuu] (0,2)-- (0,0);
\draw [line width=2pt,color=uuuuuu] (0,1)-- (1,1);
\draw [line width=2pt,color=uuuuuu] (0,2)-- (0,1);
\draw [line width=2pt,color=uuuuuu] (0,0)-- (1,0);
\draw [line width=2pt,color=uuuuuu] (1,0)-- (1,1);
\draw [line width=2pt,color=uuuuuu] (1,1)-- (0,1);
\draw [line width=2pt,color=uuuuuu] (0,1)-- (0,0);
\draw [line width=2pt,color=uuuuuu] (1,0)-- (2,0);
\draw [line width=2pt,color=uuuuuu] (1,0)-- (0,1);
\draw [line width=2pt,color=uuuuuu] (1,1)-- (1,0);
\begin{scriptsize}
\draw [fill=xdxdff] (0,0) circle (2.5pt);
\draw [fill=xdxdff] (0,2) circle (2.5pt);
\draw [fill=xdxdff] (2,0) circle (2.5pt);
\draw [fill=xdxdff] (1,1) circle (2.5pt);
\draw [fill=xdxdff] (0,1) circle (2.5pt);
\draw [fill=xdxdff] (1,0) circle (2.5pt);
\draw [fill=xdxdff] (0,2) circle (2.5pt);
\draw[color=uuuuuu] (0.7,0.7) node {$P_1$};
\draw[color=uuuuuu] (0.3,0.3) node {$P_1$};
\draw[color=uuuuuu] (0.3,1.3) node {$P_1$};
\draw[color=uuuuuu] (1.3,0.3) node {$P_1$};
\draw [fill=xdxdff] (0,2) circle (2.5pt);
\draw [fill=xdxdff] (1,1) circle (2.5pt);
\draw [fill=xdxdff] (0,1) circle (2.5pt);
\draw [fill=xdxdff] (1,1) circle (2.5pt);
\end{scriptsize}
\end{tikzpicture} 
	\end{minipage} 	{\centering\huge $\rightarrow$} \hfill
	\begin{minipage}{0.4\linewidth}
		\definecolor{zzttqq}{rgb}{0.6,0.2,0}
\definecolor{xdxdff}{rgb}{0.49019607843137253,0.49019607843137253,1}
\definecolor{ududff}{rgb}{0.30196078431372547,0.30196078431372547,1}
\definecolor{uuuuuu}{rgb}{0.26666666666666666,0.26666666666666666,0.26666666666666666}
\begin{tikzpicture}[line cap=round,line join=round,>=triangle 45,x=2cm,y=2cm]
\clip(-0.1,-0.1) rectangle (2.1,2.1);
\fill[line width=2pt,color=green,fill=green,fill opacity=0.1] (0,0) -- (2,0)  -- (0,2) -- cycle;
\draw [line width=2pt] (2,0)-- (0,0);
\draw [line width=2pt] (0,0)-- (0,2);
\draw [line width=2pt] (2,0)-- (0,2);
\begin{scriptsize}
\draw [fill=xdxdff] (0,0) circle (2.5pt);
\draw [fill=xdxdff] (0,2) circle (2.5pt);
\draw [fill=xdxdff] (2,0) circle (2.5pt);
\draw [fill=xdxdff] (1,1) circle (2.5pt);
\draw [fill=xdxdff] (0,1) circle (2.5pt);
\draw [fill=xdxdff] (1,0) circle (2.5pt);
\draw [fill=xdxdff] (0,2) circle (2.5pt);
\draw[color=uuuuuu] (0.75,0.75) node {$P_2$};
\draw [fill=xdxdff] (0,2) circle (2.5pt);
\draw [fill=xdxdff] (1,1) circle (2.5pt);
\draw [fill=xdxdff] (0,1) circle (2.5pt);
\draw [fill=xdxdff] (1,1) circle (2.5pt);
\end{scriptsize}
\end{tikzpicture}	
	\end{minipage}
	\caption{Visualization of how the degrees of freedom are interpolated on elements of higher order with a coarser grid for $P_1$ finite elements. \label{fig: Interpol P1}}
\end{figure}
Now, let $I_{h,u}\tilde{u} \in U_h^{(2)}$ be an arbitrary interpolation, approximating ${u}_h^{(2)} \in U_h^{(2)}$ which solves \eqref{eq: enriched primal}. Furthermore let, $I_{h,z}\tilde{z} \in V_h^{(2)}$ be an interpolation, approximating  ${z}_h^{(2)} \in V_h^{(2)}$, which solves \eqref{eq: enriched Adjointproblem}.
\begin{Theorem}\label{thm: discrete Error Representation Interpolation}
	Let us assume the assumptions of Theorem~\ref{thm: Error Representation} and let   $\tilde{u} \in U_h$ and $ \tilde{z} \in V_h$ be arbitrary but fixed. 
	Then for $I_{h,u}\tilde{u} \in U_h^{(2)}$ and $I_{h,z}\tilde{z} \in V_h^{(2)}$ 
	holds
	\begin{align} \label{eq: discrete error representation interpolation }
	\begin{split}
	J(I_{h,u}\tilde{u})-J(\tilde{u})&= \frac{1}{2}\left(\rho(\tilde{u})(I_{h,z}\tilde{z}-\tilde{z})+\rho^*(\tilde{u},\tilde{z})(I_{h,u}\tilde{u}-\tilde{u})\right) 
	-\rho (\tilde{u})(\tilde{z}) \\ &-\rho(I_{h,u}\tilde{u})(\frac{I_{h,z}\tilde{z}+ \tilde{z}}{2})+\frac{1}{2}\rho^*(I_{h,u}\tilde{u},I_{h,z}\tilde{z})(I_{h,u}\tilde{u}-\tilde{u})+ \tilde{\mathcal{R}}^{(3)(2)}.
	\end{split}
	\end{align}
	The term $\tilde{\mathcal{R}}^{(3)(2)}$ is given by
	\begin{equation*}
	\begin{split}	\label{Error Estimator: Remainderterm1}
	\tilde{\mathcal{R}}^{(3)(2)}:=\frac{1}{2}\int_{0}^{1}[&J_h'''(\tilde{u}+se_{I_{h,u}})(e_{I_{h,u}},e_{I_{h,u}},e_{I_{h,u}}) \\
	-&(\mathcal{A}^{(2)})'''(\tilde{u}+se_{I_{h,u}})(e_{I_{h,u}},e_{I_{h,u}},e_{I_{h,u}},\tilde{z}+se_{I_{h,z}}^*)\\
	-&3(\mathcal{A}^{(2)})''(\tilde{u}+se_{I_{h,u}})(e_{I_{h,u}},e_{I_{h,u}},e_{I_{h,z}}^*)]s(s-1)\,ds,
	\end{split} 
	\end{equation*}
	with 
	\begin{equation*}
		e_{I_{h,u}}=I_{h,u}\tilde{u}-\tilde{u} \qquad \text{and} \qquad e_{I_{h,z}}^* =I_{h,z}\tilde{z}-\tilde{z}.
	\end{equation*}

	\begin{proof}
		The proof is similar to the proof of \cite{RanVi2013,endtmayer2023goal} and Theorem~\ref{thm: Error Representation} and Theorem~\ref{thm: error identity}.
		First, we define $x_I := (I_{h,u}\tilde{u},I_{h,z}\tilde{z}) \in  X_h^{(2)}:=U_h^{(2)} \times V_h^{(2)}$ and $\tilde{x}:=(\tilde{u},\tilde{z}) \in X_h^{(2)}$.
		Since  $\mathcal{A}^{(2)} \in \mathcal{C}^3(U_h^{(2)},V_h^{(2)})$ and the $J \in
		\mathcal{C}^3(U_h^{(2)},\mathbb{R})$, we can define a discrete Lagrange functional
		\begin{equation*}
		\mathcal{L}_h(\hat{x}):= J_h(\hat{u})-\mathcal{A}^{(2)}(\hat{u})(\hat{z}) \quad \forall (\hat{u},\hat{z})=:\hat{x} \in X_h^{(2)},
		\end{equation*}
		which
		belongs to
		$\mathcal{C}^3(X_h^{(2)},\mathbb{R})$. Following the same steps as in Theorem~\ref{thm: error identity}, we get
		\begin{equation*}
		\mathcal{L}_h(x_I)-\mathcal{L}(\tilde{x})=\int_{0}^{1} \mathcal{L}'(\tilde{x}+s(x_I-\tilde{x}))(x_I-\tilde{x})\,ds.
		\end{equation*}
		Using  the trapezoidal rule 
		\begin{equation*}
		\int_{0}^{1}f(s)\,ds =\frac{1}{2}(f(0)+f(1))+ \frac{1}{2} \int_{0}^{1}f''(s)s(s-1)\,ds,
		\end{equation*}
		with $f(s):= \mathcal{L}_h'(\tilde{x}+s(x_I-\tilde{x}))(x_I-\tilde{x})$, cf. \cite{RanVi2013}, we obtain
		\begin{align*}
		\mathcal{L}_h(x_I)-\mathcal{L}(\tilde{x_I}) =& \frac{1}{2}(\mathcal{L}_h'(x_I)(x_I-\tilde{x}) +\mathcal{L}'(\tilde{x})(x_I-\tilde{x})) + \mathcal{R}^{(3)}.
		\end{align*}
		Furthermore, it follows that
		\begin{align*}
		J_h(I_{h,u}\tilde{u})-J_h(\tilde{u})=&\mathcal{L}_h(x_I)-\mathcal{L}(\tilde{x}) +{\mathcal{A}^{(2)}(I_{h,u}\tilde{u})(I_{h,z}\tilde{z}) }- \mathcal{A}^{(2)}(\tilde{u})(\tilde{z}) \\
		=& \frac{1}{2}\left(\mathcal{L}_h'(x_I)(x_I-\tilde{x}) +\mathcal{L}_h'(\tilde{x})(x_I-\tilde{x})\right) \\ 
		+&\underbrace{\mathcal{A}^{(2)}(I_{h,u}\tilde{u})(I_{h,z}\tilde{z})}_{=-\rho_h(I_{h,u}\tilde{u})(I_{h,z}\tilde{z})} \underbrace{-\mathcal{A}^{(2)}(\tilde{u})(\tilde{z})}_{=\rho_h(\tilde{u})(\tilde{z})}+ \tilde{\mathcal{R}}^{(3)(2)}.
		\end{align*}
		Investigating the part $\left(\mathcal{L}_h'(x_I)(x_I-\tilde{x}) +\mathcal{L}_h'(\tilde{x})(x_I-\tilde{x})\right)$, we observe
		\begin{align*}
		\mathcal{L}_h'(x_I)(x_I-\tilde{x}) +\mathcal{L}_h'(\tilde{x})(x_I-\tilde{x}) = & \underbrace{J_h'(I_{h,u}\tilde{u})(I_{h,u}\tilde{u}-\tilde{u})-(\mathcal{A}^{(2)})'(I_{h,u}\tilde{u})(I_{h,u}\tilde{u}-\tilde{u},I_{h,z}\tilde{z})}_{=\rho_h^*(I_{h,u}\tilde{u},I_{h,z}\tilde{z})(I_{h,u}\tilde{u}-\tilde{u})} \\
		&\underbrace{-{\mathcal{A}^{(2)}(I_{h,u}\tilde{u})(I_{h,z}\tilde{z}-\tilde{z})}}_{=\rho_h(I_{h,u}\tilde{u})(I_{h,z}\tilde{z}-\tilde{z})}\\ &+\underbrace{J_h'(\tilde{u})(I_{h,u}\tilde{u}-\tilde{u})-(\mathcal{A}^{(2)})'(\tilde{u})(I_{h,u}\tilde{u}-\tilde{u},\tilde{z})}_{=\rho_h^*(\tilde{u},\tilde{z})(I_{h,u}\tilde{u}-\tilde{u})} \\ &\underbrace{-\mathcal{A}^{(2)}(\tilde{u})(I_{h,z}\tilde{z}-\tilde{z})}_{=\rho_h(\tilde{u})(I_{h,z}\tilde{z}-\tilde{z})}.
		\end{align*}
		To sum up, we get
		\begin{align*}
		J_h(I_{h,u}\tilde{u})-J_h(\tilde{u}) &=\frac{1}{2}\left(\rho_h(\tilde{u})(I_{h,z}\tilde{z}-\tilde{z})+\rho_h^*(\tilde{u},\tilde{z})(I_{h,u}\tilde{u}-\tilde{u})\right)\\
		&+\frac{1}{2}\left(\rho_h(I_{h,u}\tilde{u})(I_{h,z}\tilde{z}-\tilde{z})+\rho_h^*((I_{h,u}\tilde{u},I_{h,z}\tilde{z})(I_{h,u}\tilde{u}-\tilde{u})\right)\\
		&-\rho_h(I_{h,u}\tilde{u})(I_{h,z}\tilde{z})+\rho_h(\tilde{u})(\tilde{z})+ \tilde{\mathcal{R}}^{(3)(2)}.
		\end{align*}
		After simplifications, we get 
		\begin{align*}
		J_h(I_{h,u}\tilde{u})-J_h(\tilde{u}) &=\frac{1}{2}\left(\rho_h(\tilde{u})(I_{h,z}\tilde{z}-\tilde{z})+\rho_h^*(\tilde{u},\tilde{z})(I_{h,u}\tilde{u}-\tilde{u})\right)\\
		&-\frac{1}{2}\rho_h(I_{h,u}\tilde{u})(I_{h,z}\tilde{z}+\tilde{z})+\frac{1}{2}\rho_h^*((I_{h,u}\tilde{u},I_{h,z}\tilde{z})(I_{h,u}\tilde{u}-\tilde{u})\\
		&+\rho_h(\tilde{u})(\tilde{z})+ \tilde{\mathcal{R}}^{(3)(2)}.
		\end{align*}
		On the discrete spaces $U_h^{(2)}$ and $V_h^{(2)}$, we have $\rho_h=\rho$ and $\rho_h^*=\rho^*$, hence we get
		\begin{align*}
		\begin{split}
		J(I_{h,u}\tilde{u})-J(\tilde{u})&= \frac{1}{2}\left(\rho(\tilde{u})(I_{h,z}\tilde{z}-\tilde{z})+\rho^*(\tilde{u},\tilde{z})(I_{h,u}\tilde{u}-\tilde{u})\right) 
		-\rho (\tilde{u})(\tilde{z}) \\ &-\rho(I_{h,u}\tilde{u})(\frac{I_{h,z}\tilde{z}+ \tilde{z}}{2})+\frac{1}{2}\rho^*(I_{h,u}\tilde{u},I_{h,z}\tilde{z})(I_{h,u}\tilde{u}-\tilde{u})+ \tilde{\mathcal{R}}^{(3)(2)}.
		\end{split}
		\end{align*}
		This concludes the proof.
	\end{proof}
\end{Theorem}
With this identity, we can define the error estimator for interpolation $\eta_{I}$ as
\begin{equation}\label{eq: def etaI}
	\begin{split}
		\eta_{I}&:= \frac{1}{2}\left(\rho(\tilde{u})(I_{h,z}\tilde{z}-\tilde{z})+\rho^*(\tilde{u},\tilde{z})(I_{h,u}\tilde{u}-\tilde{u})\right) 
		-\rho (\tilde{u})(\tilde{z}) \\ &-\rho(I_{h,u}\tilde{u})(\frac{I_{h,z}\tilde{z}+ \tilde{z}}{2})+\frac{1}{2}\rho^*(I_{h,u}\tilde{u},I_{h,z}\tilde{z})(I_{h,u}\tilde{u}-\tilde{u})+ \tilde{\mathcal{R}}^{(3)(2)}.
	\end{split}
\end{equation}
For further results, we require again a saturation assumption for the interpolation.
\begin{Assumption}[Saturation assumption for interpolations] \label{Ass: Interpolation}
	Let $I_{h,u}$ be an interpolation. Then the saturation assumption is fulfilled if 
	there exists a $b_0 \in (0,1)$ and a $b_h^I\in (0,b_0)$ with 
	\begin{equation} \label{eq: saturation interpol.}
		|J(u)-J(I_{h,u}\tilde{u})| \leq b_h^I |J(\tilde{u})-J(u)|.
	\end{equation} 
\end{Assumption}

\begin{Theorem} \label{thm: Interpolation Efficiency and Reliability}
	Let Assumption~\ref{Ass: Interpolation} be satisfied. Additionally let all assumptions of Theorem~\ref{thm: discrete Error Representation Interpolation} be fulfilled.
	Then the error estimator $\eta_{I}$ defined in \eqref{eq: def etaI}, i.e.
\begin{equation*}
	\begin{split}
		\eta_{I}&:= \frac{1}{2}\left(\rho(\tilde{u})(I_{h,z}\tilde{z}-\tilde{z})+\rho^*(\tilde{u},\tilde{z})(I_{h,u}\tilde{u}-\tilde{u})\right) 
		-\rho (\tilde{u})(\tilde{z}) \\ &-\rho(I_{h,u}\tilde{u})(\frac{I_{h,z}\tilde{z}+ \tilde{z}}{2})+\frac{1}{2}\rho^*(I_{h,u}\tilde{u},I_{h,z}\tilde{z})(I_{h,u}\tilde{u}-\tilde{u})+ \tilde{\mathcal{R}}^{(3)(2)}.
	\end{split}
\end{equation*}
	is efficient and reliable with the constants $\underline{c}=\frac{1}{1+b_h^I}$ and $\overline{c}=\frac{1}{1-b_h^I}$.
	\begin{proof}
		The proof follows the same steps as in Theorem~\ref{thm: Efficiency and Reliability}.
	\end{proof}
\end{Theorem}

\subsection{The Parts of the Error Estimator for Interpolations}
\label{subsec: Parts of EE for Interpolation}

If Assumption~\ref{Ass: Interpolation} is fulfilled, Theorem \ref{thm: Interpolation Efficiency and Reliability} shows us that the error estimator $\eta_I$ defined in \eqref{eq: def_eta(2)} is efficient and reliable as in the case of enriched spaces. 
We split $\eta_I$ into  the following parts
\begin{equation} \label{def: etaI Parts}
	\begin{split}
		\eta_{I}&:= \underbrace{\frac{1}{2}\left(\rho(\tilde{u})(I_{h,z}\tilde{z}-\tilde{z})+\rho^*(\tilde{u},\tilde{z})(I_{h,u}\tilde{u}-\tilde{u})\right) }_{:=\eta_{I,h}}
		\underbrace{-\rho (\tilde{u})(\tilde{z})}_{:=\eta_k} \\ &\underbrace{-\rho(I_{h,u}\tilde{u})(\frac{I_{h,z}\tilde{z}+ \tilde{z}}{2})}_{:=\eta_{I_u}}+\underbrace{\frac{1}{2}\rho^*(I_{h,u}\tilde{u},I_{h,z}\tilde{z})(I_{h,u}\tilde{u}-\tilde{u})}_{:=\eta_{I_z}}+ \underbrace{\tilde{\mathcal{R}}^{(3)(2)}}_{:=\eta_{\mathcal{R}_I}}.
	\end{split}
\end{equation}
Additionally, we mention that if $I_{h,u}\tilde{u}=u_h^{(2)}$ solving \eqref{eq: enriched primal} and if $I_{h,z}\tilde{z}=z_h^{(2)}$ solving \eqref{eq: enriched Adjointproblem}, then $\eta_I=\eta^{(2)}$.

\subsubsection{The Remainder Part $\eta_{\mathcal{R}_I}$ or linearization error estimator for interpolation.}
As for the enriched chase, $\eta_{\mathcal{R}_I}$ is of higher order, and usually is neglected in literature. However, we would like to mention, that here, higher-order means with respect to $I_{h,z}\tilde{z}-\tilde{z}$ and  $I_{h,u}\tilde{u}-\tilde{u}$, respectively.
\subsubsection{The Iteration Error Estimator $\eta_k$ for Interpolation}
The iteration error estimator for interpolation and enriched solutions is identical. Therefore we again use the same symbol. As already described in Section~\ref{SubSubSec: etak}, this error estimator can be used to stop the linear or non-linear solver.
\subsubsection{The Primal Interpolation Error Estimator $\eta_{I_u}$}
This part resembles the error which 
is introduced by
the interpolation of the primal variable $u$ in the enriched space. If $I_{h,u}\tilde{u}=u_h^{(2)}$, where $u_h^{(2)}$ solves \eqref{eq: enriched primal}, then $\eta_{I_u}=0$. This part can be used to decide whether $I_{h,u}\tilde{u}$ or $u_h^{(2)}$ should be used to construct the error estimator.
\subsubsection{The Adjoint Interpolation Error Estimator $\eta_{I_z}$}
The adjoint interpolation error estimator $\eta_{I_z}$ represents the error variable $z$ in the enriched. Again, if $I_{h,z}\tilde{z}=z_h^{(2)}$, where $z_h^{(2)}$ solves \eqref{eq: enriched Adjointproblem}, then $\eta_{I_z}=0$. As for the primal interpolation error estimator, we can decide whether $I_{h,z}\tilde{z}$ or $z_h^{(2)}$ is used during the computation.
\subsubsection{The Iteration Error Estimator $\eta_{I,h}$ for Interpolation}
The part 
\begin{equation} \label{eq: def primal and adjoint error estimator Interpolation}
	\eta_{I,h}:=\frac{1}{2}\left(\rho(\tilde{u})(I_{h,z}\tilde{z}-\tilde{z})+\rho^*(\tilde{u},\tilde{z})(I_{h,u}\tilde{u}-\tilde{u})\right),
\end{equation}
represents the discretization error of the error estimator as discussed in \cite{RanVi2013,endtmayer2021reliability}.
 Here,
 the part $\rho(\tilde{u})(I_{h,z}\tilde{z}-\tilde{z})$ and $\rho^*(\tilde{u},\tilde{z})(I_{h,u}\tilde{u}-\tilde{u})$ do not necessarily coincide, even if the given problem is affine linear. As in the enriched case, this part of the error estimator can be localized, which is shown in Section~\ref{SubSec:Localization}. Furthermore, this error estimator can be used to adapt the finite dimensional subspace $U_h$ or, in case of finite elements, the mesh $\mathcal{T}_h$.

\begin{Theorem}[Efficiency and Reliability Result for $\eta_{I,h}$]\label{thm: etah Efficiency and Reliability Interpolation}
	For the discretization error estimator it holds
	\begin{equation*}
		\frac{1}{1+b_h^I}\left(|\eta_{I,h}|-|\eta_k +\eta_{\mathcal{R}_I}+\eta_{I_u}+\eta_{I_z}|\right)\leq |J(u)-J(\tilde{u})| \leq\frac{1}{1-b_h^I}\left(|\eta_{I,h}|+|\eta_k +\eta_{\mathcal{R}_I}+\eta_{I_u}+\eta_{I_z}|\right).
	\end{equation*}
	Furthermore, if there exists a $\alpha_{\eta_{I,h}} \in (0,1)$ such that
	\begin{equation}
		|\eta_k +\eta_{\mathcal{R}_I}+\eta_{I_u}+\eta_{I_z}| \leq \alpha_{\eta_{I,h}} |\eta_{I,h}|,
	\end{equation}
	then the discretization error estimator is efficient and reliable, i.e.
	\begin{equation}
		\underline{c}_h|\eta_{I,h}|\leq |J(u)-J(\tilde{u})| \leq\overline{c}_h|\eta_{I,h}|,
	\end{equation}
	with $\underline{c}_h:=(1-\alpha_{\eta_{I,h}})\frac{1}{1+b_h^I}$ and $\overline{c}_h:=(1+\alpha_{\eta_{I,h}})\frac{1}{1-b_h^I}$.
	\begin{proof}
		The proof follows from the same arguments as the proof of Theorem~\ref{thm: etah Efficiency and Reliability}.
	\end{proof}
\end{Theorem}

\subsection{Effectivity Index for Interpolations}
{
In this subsection, we define and investigate effectivity indices for interpolations. For background and motivation of effectivity indices, we refer 
the reader to Section \ref{sec_Ieff_enriched_spaces}.
}
\begin{Theorem}[Bounds on the effectivity index]
	We assume that $|J(u)-J(\tilde{u})|\not=0$.
	Additionally, let the assumptions of Theorem~\ref{thm: Efficiency and Reliability} be fulfilled.
	Then, for the effectivity index $I_{\mathrm{eff},+}^I$ defined by
	\begin{equation} \label{def: Ieff+ Int}
		I_{\mathrm{eff},+}^I:= \frac{\eta_I}{J(u)-J(\tilde{u})},
	\end{equation}
	we have the bounds
	\begin{equation} \label{eq: Ieff+ bounds Int}
		1-b_h^I\leq |I_{\mathrm{eff},+}^I| \leq 1+b_h^I.
	\end{equation}
	If the assumptions of Theorem~\ref{thm: etah Efficiency and Reliability} are fulfilled then for the effectivity index $I^I_{\mathrm{eff}}$ defined by
	\begin{equation} \label{def: Ieff Int}
		I^I_{\mathrm{eff}}:= \frac{\eta_{I,h}}{J(u)-J(\tilde{u})},
	\end{equation}
	we have the bounds
	\begin{equation} \label{eq: Ieff bounds Int}
		\frac{1-b_h^I}{1+\alpha_{\eta_{I,h}}}\leq |I^I_{\mathrm{eff}}| \leq \frac{1+b_h^I}{1-\alpha_{\eta_{I,h}}}.
	\end{equation}
\begin{proof}
	The proof follows the same steps as the proof of Theorem~\ref{thm: Bounds: Ieff}.
\end{proof}
\end{Theorem}
%

\subsection{Error Localization with a Partition-of-Unity Approach}
\label{SubSec:Localization}
In this subsection, we address error localization such that globally defined error
estimators $\eta$
can be applied locally to mesh elements
in order to steer adaptive algorithms.
To this end, the previous a posteriori 
error estimators $\eta$ need to be split into element-wise or DoF-wise 
contributions $\eta_i, i=1,\ldots,N$, where $N$ is the number of degrees of freedom. 
Three known approaches are
the classical integration by parts \cite{BeRa96}, 
a variational filtering operator over patches of elements \cite{BraackErn02} 
and a variational partition-of-unity localization \cite{RiWi15_dwr}. 
For stationary problems, the
effectivity of these localizations 
was established and numerically substantiated in \cite{RiWi15_dwr}.
The quality measure of the localization process is the so-called 
indicator index that was as well introduced in \cite{RiWi15_dwr}: 
\begin{Definition}
Let $u\in U$ be the solution of \eqref{eq: General Modelproblem} and $\tilde{u}\in U_h$. For the definition of the discretization error $\eta_h^{(2)}$ in 
\eqref{eq: def primal and adjoint error estimator}, and localizing to all 
degrees of freedom $i=1,\ldots,N$ of the governing triangulation, the indicator 
index is defined as
\begin{equation}
I_{ind} := \frac{\sum_{i=1}^N |\eta_h^{(2)}(i)|}{|J(u)-J(\tilde{u})|}.
\end{equation}
Here, $\eta_h^{(2)}(i)$ is the localization of $\eta_h^{(2)}$ to the degree of freedom 
$i$ for $i=1,\ldots,N$.
\end{Definition}
For good effectivity and indicator indices, an important point is the influence 
of neighboring mesh elements; see e.g., \cite{CaVer99}. Consequently, in the traditional 
method, integration by parts ensured 
gathering
information from neighbor faces and edges 
\cite{BeRa01}.
For coupled problems and multiphysics applications, this procedure is error-prone 
and might be computationally expensive as higher-order operators need to be evaluated. 
Consequently, the objective is to stay in the weak formulation as done with the filtering approach in \cite{BraackErn02}.
A further simplification was then the so-called PU-DWR approach \cite{RiWi15_dwr},
which we shall explain in the following. 
Using the PU, we touch different mesh elements
per PU-node, and consequently, we gather the required information from neighboring elements.
In our adaptive procedures in the remainder of this work, we always employ the 
PU-DWR method, namely for stationary problems with single goals, multiple goals, 
and space-time situations with both single and multiple goals.

\subsubsection{Abstract Realization}
In this section, we derive abstract results and show that the PU-DWR localization enters simply 
as a modified test function into the semi-linear/bilinear forms of the operator equation
and the right hand side. To start, we first introduce the PU space and 
its fundamental property.
Let us assume $\{\chi_1,\ldots,\chi_M\}$ is a basis of the PU space 
$V_{PU}$ such that
\begin{align}
\label{eq_PU_one}
\sum_{i=1}^{M} \chi_{i} \equiv1.
\end{align}          
holds. Common choices for the PU spaces are low-order finite element 
spaces such as $V_{PU} = Q_1$ \cite{RiWi15_dwr}, $V_{PU} = P_1$,
$V_{PU} = \widetilde{X}_{k,h}^{0,1}$, i.e., a $cG(1)dG(0)$ space-time discretization \cite{ThiWi24} or $V_{PU} = X_{kh}^{1,1}$, i.e. a $cG(1)cG(1)$ discretization, 
for example in $d+1$-dimensional space-time discretizations \cite{endtmayer2023goal}.
In general, this ensures a coupling between neighboring temporal elements to address the 
problem shown in \cite{CaVer99}. However, for discontinuous Galerkin discretizations, the dominating edge residuals, i.e. jump terms, are explicitly 
included in the estimator.

\begin{Proposition}[Localized error estimator]
Let the previous PU be given.
The localized form of the error estimator $\eta$ \eqref{eq: theory_Error_Estimator}
is
\[
J(u)-J(\tilde{u}) = \eta = \sum_{i=1}^M
\frac{1}{2}\rho(\tilde{u})((z-\tilde{z})\chi_i)
+\frac{1}{2}\rho^*(\tilde{u},\tilde{z})((u-\tilde{u})\chi_i)
+\rho(\tilde{u})(\tilde{z}) + \mathcal{R}^{(3)}(\chi_i).
\]
As before, the part $\rho(\tilde{u})(\tilde{z})$ 
determines the deviation of the approximate solution $\tilde u$ 
in comparison to the `exact' discrete solution $u_h$. This can be iteration errors 
due iterative linear or non-linear solutions. Since they act globally on the entire 
solution, no localization to $i$ with the PU function is required here. 
\end{Proposition}
\begin{proof}
Inserting the PU function $\chi_i$ into \eqref{eq: theory_Error_Estimator}
and using property \eqref{eq_PU_one}, we immediately establish the result.
\end{proof}
\begin{Corollary}[Localized error estimator]
\label{coro_localized_eta}
Neglecting the remainder term yields a computable form and
it holds
\[
J(u)-J(\tilde{u}) \approx \eta = \sum_{i=1}^M
\frac{1}{2}\rho(\tilde{u})((z-\tilde{z})\chi_i)
+\frac{1}{2}\rho^*(\tilde{u},\tilde{z})((u-\tilde{u})\chi_i)
+\rho(\tilde{u})(\tilde{z}).
\]
\end{Corollary}

\begin{Definition}
The error estimator in Corollary \ref{coro_localized_eta}
is composed by the following parts:
\[
\eta = \eta_h + \eta_k := \sum_{i=1}^M (\eta_p + \eta_a) + \eta_k,
\]
where $\eta_h$ denotes the discretization error and $\eta_k$ the 
non-linear iteration error. Specifically, we have 
\begin{align*}
\eta_p &:= \eta_p(i) := \frac{1}{2}\rho(\tilde{u})((z-\tilde{z})\chi_i), \\
\eta_a&:= \eta_a(i) := \frac{1}{2}\rho^*(\tilde{u},\tilde{z})((u-\tilde{u})\chi_i),\\
\eta_k &:= \rho(\tilde{u})(\tilde{z}).
\end{align*}
\end{Definition}
In the computational realization it is immediately clear from Section \ref{Sec:single-goal} 
that 
$z$ and $\tilde{z}$ as well as $u$ and $\tilde u$ are approximated 
through discrete unknowns from spaces such as $U_h$. 
Here, it is important that $z_h \approx z$ and $\tilde{z}_h \approx \tilde{z}$
come from different discrete spaces since otherwise 
$z_h-\tilde{z}_h \equiv 0$ and $u_h-\tilde{u}_h \equiv 0$.
A practical version of the previous localized form reads:
\begin{Proposition}[Practical error estimator]
\label{prop_DWR_practical}
Let $\tilde{u}\in U_h$ be a low-order approximation to \eqref{eq: discrete primal}, $u_h^{(2)}\in U_h^{(2)}$ the higher-order solution to \eqref{eq: enriched primal}, and
$\tilde{z} \in U_h$ be a low-order approximation to \eqref{eq: discrete adjoint} $z_h^{(2)}\in V_h^{(2)}$ the higher-order adjoint solutions \eqref{eq: enriched Adjointproblem},
respectively.
The practical localized PU error estimator reads
\[
J(u)-J(\tilde{u}) \approx \eta = \sum_{i=1}^M
\frac{1}{2}\rho(\tilde{u})((z_h^{(2)} - \tilde{z})\chi_i)
+\frac{1}{2}\rho^*(\tilde{u},\tilde{z})((u_h^{(2)} - \tilde{u})\chi_i)
+\rho(\tilde{u})(\tilde{z}),
\]
where we now re-define the previous notation and obtain as error parts
\[
\eta = \eta_h + \eta_k := \sum_{i=1}^M (\eta_p + \eta_a) + \eta_k
\]
with
\begin{align*}
\eta_p &:= \eta_p(i) := \frac{1}{2}\rho(\tilde{u})((z_h^{(2)} - \tilde{z})\chi_i),\\
\eta_a &:= \eta_a(i) := \frac{1}{2}\rho^*(\tilde{u},\tilde{z})((u_h^{(2)} - \tilde{u})\chi_i),\\
\eta_k &:= \rho(\tilde{u})(\tilde{z}).
\end{align*}
\end{Proposition}

\subsubsection{Details on Using $\tilde{u}_h^{(2)}$ for the Localization}
In this part, we use $\tilde{u}_h^{(2)}=u_h^{(2)}$ in the case of enriched spaces and $\tilde{u}_h^{(2)}=I_{h,u}\tilde{u}$ in the case of interpolation.
For unified notation we use $\tilde{u}_h^{(2)}$ for the interpolation and enriched solutions.
Furthermore, we use $\tilde{z}_h^{(2)}=z_h^{(2)}$ in the case of enriched spaces and $\tilde{z}_h^{(2)}=I_{h,z}\tilde{z}$ in the case of interpolation. 
As before, let $\chi_i \in S_1$, where $S_1 \in \{P_1,Q_1\}$. Then, we know that $\sum_{i=1}^{M}\chi_i=1$, where $M:=\operatorname{dim}(S_1)$.
It holds
	\begin{align}
	&\frac{1}{2}\left(\eta_p(\tilde{u})(\tilde{z}_h^{(2)}-\tilde{z})+\eta_a(\tilde{u},\tilde{z})(\tilde{u}_h^{(2)}-\tilde{u})\right), \nonumber\\
	=&\frac{1}{2}\left(\eta_p(\tilde{u})((\tilde{z}_h^{(2)}-\tilde{z})\sum_{i=1}^{M}\chi_i)+\eta_a(\tilde{u},\tilde{z})((\tilde{u}_h^{(2)}-\tilde{u})\sum_{i=1}^{M}\chi_i)\right) \nonumber\\
	=&\sum_{i=1}^{M}\underbrace{\frac{1}{2}\left(\eta_p(\tilde{u})((\tilde{z}_h^{(2)}-\tilde{z})\chi_i)+\eta_a(\tilde{u},\tilde{z})((\tilde{u}_h^{(2)}-\tilde{u})\chi_i)\right)}_{:=\eta_i^{PU}} \label{eq: eta_iPU}\\
	=&\sum_{i=1}^{M}\eta_i^{PU}. \nonumber
	\end{align}	 
These indicators $\eta_i^{PU}$ represent an error distribution of the PU 
or the nodal error contribution.
 This nodal error estimator is equally distributed to all elements sharing that node. If $S_1=Q_1$, then  adaptive refinement will introduce hanging nodes. The nodal error estimator on these hanging nodes is distributed to the corresponding neighboring nodes as visualized in Figure~\ref{fig: NiPU hanging nodes}.
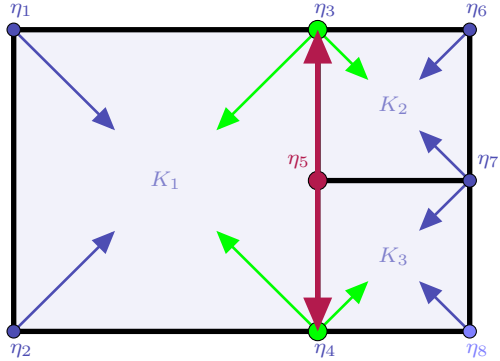
\begin{figure}[H]
	\centering
	\definecolor{farbe2}{rgb}{0.7,0.1,0.3}
	\definecolor{farbe1}{rgb}{0.5,0.5,1}
	\definecolor{farbe4}{rgb}{0.5,0.5,0.8}
	\definecolor{farbe5}{rgb}{0,1,0}
	\definecolor{farbe6}{rgb}{0.3,0.3,0.7}
	\scalebox{1.0}{
		\begin{tikzpicture}[line cap=round,line join=round,>=triangle 45,x=1cm,y=1cm]
			\clip(-5,-1) rectangle (3,5);
			\fill[line width=2pt,color=farbe4,fill=farbe4,fill opacity=0.10000000149011612] (-4,4) -- (-4,0) -- (0,0) -- (0,4) -- cycle;
			\fill[line width=2pt,color=farbe4,fill=farbe4,fill opacity=0.10000000149011612] (0,0) -- (2,0) -- (2,2) -- (0,2) -- cycle;
			\fill[line width=2pt,color=farbe4,fill=farbe4,fill opacity=0.10000000149011612] (2,2) -- (2,4) -- (0,4) -- (0,2) -- cycle;
			\draw [line width=2pt,color=black] (-4,4)-- (-4,0);
			\draw [line width=2pt,color=black] (-4,0)-- (0,0);
			\draw [line width=2pt,color=black] (0,0)-- (0,4);
			\draw [line width=2pt,color=black] (0,4)-- (-4,4);
			\draw [line width=2pt,color=black] (0,0)-- (2,0);
			\draw [line width=2pt,color=black] (2,0)-- (2,2);
			\draw [line width=2pt,color=black] (2,2)-- (0,2);
			\draw [line width=2pt,color=black] (2,2)-- (2,4);
			\draw [line width=2pt,color=black] (2,4)-- (0,4);
			\draw [line width=2pt,color=black] (0,4)-- (0,2);
			
			\draw [color=farbe5][->,line width=1pt] (0,4) -- (0.6667,3.3333);
			\draw [color=farbe6] [->,line width=1pt] (2,4) -- (1.3333,3.3333);
			\draw [color=farbe6] [->,line width=1pt] (2,2) -- (1.3333,2.6667);
			\draw [color=farbe5][->,line width=1pt] (0.0,-0.0) -- (0.6667,0.6667);
			\draw [color=farbe6] [->,line width=1pt] (2,0) -- (1.3333,0.6667);
			\draw [color=farbe6]  [->,line width=1pt] (2,2) -- (1.3333,1.3333);
			\draw [color=farbe6]  [->,line width=1pt] (-4,4) -- (-2.6667,2.6667);
			\draw [color=farbe5][->,line width=1pt] (0,4) -- (-1.3333,2.6667);
			\draw [color=farbe6] [->,line width=1pt] (-4,0) -- (-2.6667,1.3333);
			\draw [color=farbe5][->,line width=1pt] (0.0,-0.0) -- (-1.3333,1.3333);
			\begin{scriptsize}
				\draw [fill=farbe6] (-4,4) circle (2.5pt);
				\draw[color=farbe6] (-3.9,4.25) node {$\eta_1$};
				\draw [fill=farbe6] (-4,0) circle (2.5pt);
				\draw[color=farbe6] (-3.9,-0.25) node {$\eta_2$};
				\draw [fill=farbe5] (0,4) circle (3.5pt);
				\draw[color=farbe6] (0.1,4.25) node {$\eta_3$};
				\draw [fill=farbe1] (2,0) circle (2.5pt);
				\draw[color=farbe1] (2.1,-0.25) node {$\eta_8$};
				\draw [fill=farbe6] (2,2) circle (2.5pt);
				\draw[color=farbe6] (2.25,2.25) node {$\eta_7$};
				\draw [fill=farbe2] (0,2) circle (3.5pt);
				\draw[color=farbe2] (-0.25,2.25) node {$\eta_5$};
				\draw [fill=farbe6] (2,4) circle (2.5pt);
				\draw[color=farbe6] (2.1,4.25) node {$\eta_6$};
				\draw [fill=farbe5] (0,0) circle (3.5pt);
				\draw[color=farbe6] (0.1,-0.25) node {$\eta_4$};
				\draw[color=farbe4] (-2,2) node {$K_1$};
				\draw[color=farbe4] (1,3) node {$K_2$};
				\draw[color=farbe4] (1,1) node {$K_3$};
				\draw [line width=2pt,color=farbe2] [->,line width=2pt] (0,2) -- (0,4);
				\draw [line width=2pt,color=farbe2] [->,line width=2pt] (0,2) -- (0,0);			
			\end{scriptsize}
		\end{tikzpicture}							
	}
\caption{Nodal error estimator on the hanging $\eta_5$ is distributed equally to $\eta_3$ and $\eta_5$ the nodal contribution is distributed to the elements.\label{fig: NiPU hanging nodes}}
\end{figure}

\begin{Remark}
Of course the partition of unity technique is not restricted to the finite element method and can be applied to other discretization techniques like isogeometric analysis\cite{IGA} as well. 
\end{Remark}

\subsection{Adaptive Algorithms}
\label{sec_adaptive_algorithms_single_goal}
In this section, we briefly describe the fundamental 
algorithms for enriched approximation and interpolation.

\subsubsection{Adaptive Algorithm Using Enriched Approximations}
Using the result of the previous subsections of Section~\ref{Sec:single-goal}, we can construct the following algorithm if we use enriched spaces.
\begin{algorithm}[H]
	\caption{The adaptive algorithm for enriched approximations} \label{Alg: adapt. Alg}
	\begin{algorithmic}[1]
		\Procedure{Goal Adaptive}{$J$, $\mathcal{A}$, $\mathcal{T}_0$, TOL, maxNDoFs}      
		\State $k\leftarrow 0$,$\mathcal{T}_k\leftarrow \mathcal{T}_0$, $\eta_h\leftarrow \infty$
		\While{$\eta_h> 10^{-2}$ TOL \& $|\mathcal{T}_k|$ $\geq$ maxNDoFs }
		\State Solve \eqref{eq: discrete primal} to obtain $\tilde{u}$ with some non-linear solver like Newton's method \label{Alg: Single: Step uh}
		\State Solve \eqref{eq: enriched primal}to obtain $u_h^{(2)}$  with some non-linear solver like Newton's method \label{Alg: Single: Step uh2}
		\State Solve \eqref{eq: enriched Adjointproblem} and \eqref{eq: discrete adjoint} to obtain $z_h^{(2)}$ and $\tilde{z}$ using some linear solver.
		\State Compute $\eta_h$ and  the node-wise error contribution $\eta_i^{PU}$ as in \eqref{eq: eta_iPU}.
		\State Distribute $\eta_i^{PU}$ equally to all elements that share the node $i$.
		\State Mark the elements with some marking strategy, e.g.\ D\"orfler marking \cite{Doerfler:1996a}.
		\State Refine the mesh according to the marked cells
		\State $k\leftarrow k+1$
		\EndWhile
		\Return $J(u_k)$
		\EndProcedure
	\end{algorithmic}	
\end{algorithm}
\begin{Remark}
	In Algorithm~\ref{Alg: adapt. Alg},the Line~\ref{Alg: Single: Step uh2} and \ref{Alg: Single: Step uh} can be swapped.
\end{Remark}
\subsubsection{Adaptive Algorithm Using Interpolations}
From the previous subsections of Section~\ref{Sec:single-goal}, we derive the following algorithm featuring interpolations.
\begin{algorithm}[H]
	\caption{The adaptive algorithm using interpolations} \label{Alg: adapt. AlgInterpolation}
	\begin{algorithmic}[1]
		\Procedure{Goal Adaptive}{$J$, $\mathcal{A}$, $\mathcal{T}_0$, TOL, maxNDoFs}      
		\State $k\leftarrow 0$,$\mathcal{T}_k\leftarrow \mathcal{T}_0$, $\eta_h\leftarrow \infty$
		\While{$\eta_h> 10^{-2}$ TOL \& $|\mathcal{T}_k|$ $\geq$ maxNDoFs }
		\State Solve \eqref{eq: discrete primal} to obtain $\tilde{u}$ with some non-linear solver like Newton's method \label{Alg: Single: Step uh I}
		\State Construct the Interpolation $\tilde{u}_h^{(2)}=I_{h,u}\tilde{u}$ \label{Alg: Single: Step Iuh2}
		\State Solve \eqref{eq: discrete adjoint} to obtain $\tilde{z}$ using some linear solver. \label{Alg: Single: Step zh I}
		\State Construct the Interpolation $\tilde{z}_h^{(2)}=I_{h,z}\tilde{z}$ \label{Alg: Single: Step Izh2}
		\State Compute $\eta_h$ and  the node-wise error contribution $\eta_i^{PU}$ as in \eqref{eq: eta_iPU}.
		\State Distribute $\eta_i^{PU}$ equally to all elements that share the node $i$.
		\State Mark the elements with some marking strategy, e.g.\ D\"orfler marking \cite{Doerfler:1996a}.
		\State Refine the mesh according to the marked cells
		\State $k\leftarrow k+1$
		\EndWhile
		\Return $J(u_k)$
		\EndProcedure
\end{algorithmic}
\end{algorithm}
\begin{Remark}	
In Algorithm~\ref{Alg: adapt. AlgInterpolation}, we cannot swap the solution steps with their corresponding interpolation steps, e.g.\ Line~\ref{Alg: Single: Step Iuh2} cannot be done before Line~\ref{Alg: Single: Step uh I}.
\end{Remark}

%
%
\section{Error Estimation for non-standard Discretizations}
\label{Sec:Nonstandard}

\subsection{Motivation and Examples of Non-Consistencies}
The basic error identity sketched in
Section~\ref{SubSec:Error-Identity} poses very little assumptions on
the underlying problem and in particular on $u,\tilde u\in U$ and the
adjoints $z,\tilde z\in V$. It can directly be applied to any
$U$-conforming discretization, i.e. $u_h\in U_h\subset U$ and $z_h\in
V_h\subset V$ yielding the identity
\begin{equation}\label{ns:id}
  J(u)-J(u_h) = \frac{1}{2}\big(\rho(u_h)(z-z_h) +
  \rho^*(u_h,z_h)(u-u_h)\big) - \rho(u_h)(z_h) + \mathcal{R}^{(3)}.
\end{equation}
If we further consider consistent discretizations, i.e. discrete
solutions satisfying
\begin{equation}\label{ns:cons}
  \begin{aligned}
    \mathcal{A}(u_h)(\phi_h) &= 0\quad \forall \phi_h\in V_h,\\
    \mathcal{A}'(u_h)(\psi_h,z_h) &= J'(u_h)(\psi_h)\quad \forall \psi_h\in U_h,
  \end{aligned}
\end{equation}
and assume for now that the discrete solution $u_h\in U_h$ is indeed a
solution and no iteration error remains,~\eqref{ns:id} simplifies to
\begin{equation}\label{ns:id1}
  J(u)-J(u_h) = \frac{1}{2}\big(\rho(u_h)(z-z_h) +
  \rho^*(u_h,z_h)(u-u_h)\big) + \mathcal{R}^{(3)}.
\end{equation}
This gives rise to the classical formulation of the adjoint error
identity~\cite{BeRa01} that, using Galerkin orthogonality, allows to
localize the error by replacing the approximation errors $u-u_h$ and
$z-z_h$ by any interpolation $I_h u\in U_h$ and $I_h z\in V_h$
\begin{equation}\label{ns:id2}
  J(u)-J(u_h) = \frac{1}{2}\big(\rho(u_h)(z-I_h z) +
  \rho^*(u_h,z_h)(u-I_h u)\big) + \mathcal{R}^{(3)}.
\end{equation}
Besides conformity and consistency of the discretization, the
fundamental assumption is the variational principle defining the
solution $u\in U$ and its adjoint $z\in V$.

In the following paragraphs, we examine various problems in which one
or more of these assumptions are violated. One straightforward example is
the realization of Dirichlet conditions. Assume that the proper
solution space would be
\[
u\in U = u^D+H^1_0(\Omega),
\]
where $u^D\in H^1(\Omega)$ is the extension of some boundary data
$g\in H^\frac{1}{2}(\partial\Omega)$. The discrete finite element
realization would find
\[
u_h\in U_h = I_h u^D + U_h^0,
\]
where $U_h^0\subset H^1_0(\Omega)$ has homogeneous Dirichlet data, but
where $I_h u^D\neq u^D$ such that $U_h\not\subset U$. An
even simpler example is the weak  enforcement of boundary conditions
by Nitsche's method~\cite{Ni71}, where, in general, $u_h\neq 0$ on the
discrete boundary and hence $u_h\not\in H^1_0(\Omega)$. Further
examples leading to non-conforming discretizations are approximations
on curved domains, where $\Omega\neq \Omega_h$~\cite{MinakowskiRichter2019}, non-conforming finite
elements such as the Crouzeix-Raviart element \cite{CrouRa73} or discontinuous
Galerkin methods \cite{Ri08,DiErn2012} for elliptic problems.

Another potential source of problems lies in the non-consistency of
the discrete formulation. While consistency is required to use
Galerkin orthogonality for localization in the spirit
of~\eqref{ns:id2}, it is not essential for the path outlined in
Section~\ref{Sec:single-goal} as the error identity contains the term
$\rho(\tilde u)(\tilde z)$, which, on the one hand, stands for
iteration errors, but which also includes all non-conformity
errors. Classical sources of non-consistency are stabilized finite
element methods, where the discrete variational formulation must be
enriched. Examples are transport stabilizations such as streamline diffusion for
a simple diffusion transport problem
\begin{equation}\label{ns:stab}
  \mathcal{A}_h(\cdot)(\cdot) = \mathcal{A}(\cdot)(\cdot)+\mathcal{S}(\cdot)(\cdot),
\end{equation}
where
\begin{equation}\label{ns:stab:sd}
  \mathcal{A}_h(u)(\phi) = \underbrace{(\nabla u,\nabla \phi) + (\beta\cdot\nabla
    u,\phi)}_{={\mathcal A}(u)(\phi)} +
  \underbrace{(\delta_h\beta\cdot\nabla u,\beta\cdot\nabla
    \phi)}_{={\mathcal S}(u)(\phi)}
\end{equation}
and where the solution $u\in U$ to $\mathcal{A}(u)(\phi)=0$ for all
$\phi\in V$ does not satisfy the discrete problem
$\mathcal{A}_h(u)(\phi)=0$. Other examples are stabilizations of
saddle-point problems such as local
projections~\cite{BeckerBraack2001}. Non-consistency is also relevant
for time-dependent problems. While the DWR
estimator can directly be applied to space-time Galerkin methods as
discussed in Section~\ref{Sec:single-goal}, Galerkin time
discretizations might not be favorable in terms of computational
complexity. Instead, efficient time-stepping methods are used for
simulation and their similarity to certain Galerkin space-time
discretization is utilized for error estimation 
only~\cite{MeiRi14,MeidnerRichter2015}.

The idea of different methods in simulation and theoretical error
analysis can be taken even further. For example, the error estimator
can in principle be applied to discrete solutions that are not based
on variational principles at all, but are obtained using finite
volume methods~\cite{CHEN201469} or neural networks~\cite{minakowski2021error}. 


\subsection{Estimating the Consistency Error}
\label{SubSec:consistency}

As before, let $\mathcal{A}(\cdot)(\cdot)$ be the variational semi-linear form
describing the problem at hand. Then, we have: Find 
{%
\begin{equation}
u\in U: \quad \mathcal{A}(u)(v) = 0\quad \forall v\in V.
\end{equation}
}%
Then, $u_h\in U_h\subset U$ shall be the discrete approximation that
is given in the modified variational form 
\begin{equation}
u_h\in U_h\subset U\quad \mathcal{A}_h(u_h)(v_h) = 0\quad \forall v_h\in V_h,
\end{equation}
where we assume that consistency does not hold, i.e. in the general case we have
\begin{equation}
\mathcal{A}_h(u)(v_h) \neq 0\quad\forall v_h\in V_h.
\end{equation}
For the following we assume that the discrete variational form can be written as
$\mathcal{A}_h(\cdot)(\cdot)=\mathcal{A}(\cdot)(\cdot)+\mathcal{S}_h(\cdot)(\cdot)$, and that the discrete adjoint problem is given by:
\begin{equation}
  \text{Find  $z_h\in V_h$:} \quad \mathcal{A}_h'(u_h)(v_h,z_h) = J'(u_h)(v_h)\quad\forall v_h\in U_h,
\end{equation}
with $\mathcal{A}_h'(\cdot)(\cdot,\cdot) = \mathcal{A}'(\cdot)(\cdot,\cdot) +
\mathcal{S}^*_h(\cdot)(\cdot,\cdot)$, where $\mathcal{S}^*_h$ is not necessarily the derivative of $\mathcal{S}_h(\cdot,\cdot)$. 

\begin{Theorem}\label{thm: Error Representation Consistency}
  Let $u\in U$ and $z\in V$ be primal and adjoint solutions such that
  \[
  \mathcal{A}(u)(v) = 0\quad\forall v\in V,\quad
  \mathcal{A}'(u)(v,z) = J'(u)(v)\quad\forall v\in U.
  \]
  Further, let  $u_h\in U_h\subset U$ and $z_h\in  V_h\subset V$ be primal and adjoint discrete approximations such that
  \[
  \mathcal{A}_h(u_h)(v_h) \approx 0\quad\forall v_h\in V_h,\quad
  \mathcal{A}_h'(u_h)(v_h,z_h) \approx J'(u_h)(v_h)\quad\forall v_h\in U_h,
  \]
  where
  \begin{equation}\label{ns:stabform}
    \mathcal{A}_h(u)(v) = \mathcal{A}(u)(v) + \mathcal{S}_h(u)(v),\quad
    \mathcal{A}'_h(u)(v,z) = \mathcal{A}'(u)(v,z) + \mathcal{S}_h^*(u)(v,z).
  \end{equation}
  Then, the following error representation holds
  \begin{equation} \label{eq:ConsistencyError}
    J(u)-J(u_h)= \frac{1}{2}\left(\rho(u_h)(z-z_h)+\rho^*(u_h,z_h)(u-u_h)\right) - \rho_h(u_h)(z_h)-\mathcal{S}_h(u_h)(z_h) + \mathcal{R}^{(3)}.
  \end{equation}
  Here, $\rho(\cdot)(\cdot),\rho^*(\cdot,\cdot)(\cdot)$ and $\mathcal{R}^{(3)}$ are defined as in Theorem~\ref{thm: error identity}, whereas
  \[
  \rho_h(u)(v) = -A_h(u)(v),\quad \rho_h^*(u,z)(v) = J'(u)(v)-A'(u)(v,z).
  \]
  Given interpolations $I_h:U\to U_h$ and $I_h^*:V\to V_h$, the localization of this error identity reads
  \begin{multline} \label{eq:ConsistencyErrorLocal}
    J(u)-J(u_h)= \frac{1}{2}\left(\rho(u_h)(z-I_h^*z)+\rho^*(u_h,z_h)(u-I_hu) \right)\\
    -\rho_h(u_h)(z_h) 
    +\frac{1}{2}\left(\rho_h(u_h)(I_h^*z - z_h)+\rho_h^*(u_h,z_h)(I_h^*u - u_h)\right)\\
    -\frac{1}{2}\left(\mathcal{S}_h(u_h)(I_h^*z - z_h)+\mathcal{S}_h^*(u_h)(I_h^*u - u_h,z_h)\right) - \mathcal{S}_h(u_h)(z_h)
    + \mathcal{R}^{(3)}.
  \end{multline}
  \begin{proof}
    The error identity~\eqref{eq:Error Representation} of Theorem~\ref{thm: error identity} does not require consistency and is directly applicable
    \[
    J(u)-J(u_h)= \frac{1}{2}\left(\rho(u_h)(z-z_h)+\rho^*(u_h,z_h)(u-u_h)\right) - \rho(u_h)(z_h) + \mathcal{R}^{(3)}.
    \]
  Using~\eqref{ns:stabform},~\eqref{eq:ConsistencyError} directly follows. 

  If we want to exploit Galerkin orthogonality to localize by means of replacing the approximation errors $u-u_h$ and $z-z_h$ by interpolation weights, we introduce $\pm \rho(u_h)(z_h-I_h^*z)$ and $\pm \rho^*(u_h)(u_h-I_h u,z_h)$, and we obtain
  \begin{multline}\label{ns:7}
    J(u)-J(u_h)= \frac{1}{2}\left(\rho(u_h)(z-I_h^*z)+\rho^*(u_h,z_h)(u-I_hu) \right)
    -\rho_h(u_h)(z_h)  - \mathcal{S}_h(u_h)(z_h)\\
    +\frac{1}{2}\left(\rho(u_h)(I_h^*z - z_h)+\rho^*(u_h,z_h)(I_h^*u - u_h)\right)
    + \mathcal{R}^{(3)}.
  \end{multline}
  We further use~\eqref{ns:stabform} to split $\rho(\cdot)(\cdot) = \rho_h(\cdot)(\cdot)-\mathcal{S}_h(\cdot)(\cdot)$, and likewise $\rho^*(\cdot,\cdot)(\cdot)$ to separate the error identity into weighted residuals, iteration errors and consistency errors to reach~\eqref{eq:ConsistencyErrorLocal}.
  \end{proof}
\end{Theorem}
The error identities~\eqref{eq:ConsistencyError} and~\eqref{eq:ConsistencyErrorLocal} now consist of five parts: primal and dual weighted residuals, the discrete iteration error, the non-consistency error and, finally, the remainder. Apart from the remainder, all other terms can be evaluated numerically, where for the residuals one of the approximations discussed in Section~\ref{SubSec:EEEnriched} or Section~\ref{SubSec:Interpolation} has to be used. The iteration errors vanish if the algebraic problems are solved to sufficient precision.

\begin{Remark}[Consistency remainders]
  The two additional non-consistency terms appearing in~\eqref{ns:7} can not be directly evaluated, as they depend on the unknown exact solutions $u\in U$ and $z\in V$. However, they are usually negligible, as they carry an additional order compared to the consistency term $S_h(u_h)(z_h)$. As an example we consider the streamline diffusion stabilization~\eqref{ns:stab:sd}, where
  \begin{equation}
  \mathcal{S}_h(u_h,v_h) = (\delta_h \beta\cdot\nabla u_h,\beta\cdot\nabla v_h)
  \end{equation}
  and $\mathcal{S}^*_h(u_h)(v_h,z_h) = \mathcal{S}_h(v_h,z_h)$. While $\mathcal{S}_h(u_h)(z_h)$ can simply be estimated as
  \begin{equation} \label{ns:x}
  \big|\mathcal{S}_h(u_h)(z_h)\big| \le c |\delta_h|\,\|\nabla u_h\|\cdot \|\nabla z_h\|,
  \end{equation}
  primal and adjoint    consistency remainders give rise to
  \begin{equation} 
  \big|\mathcal{S}_h(u_h)(I_h^*-z_h)\big| \le c |\delta_h|\,\|\nabla u_h\|\cdot
  \big(\|\nabla (z-z_h)\|+\|\nabla (z-I_h^*z)\|\big),
  \end{equation}
  which is of higher order as compared to the main consistency term~\eqref{ns:x}.
\end{Remark}

\subsubsection{Consistency Error in Galerkin Time-Stepping Methods}

Besides stabilization techniques, another typical application of tracking consistency 
errors in error estimation is found in time-stepping methods. To illustrate this, let $u'(t) = f(t,u(t))$ be a given scalar initial value problem on $I=(0,T)$ with $u(0)=u_0$. By $0=t_0<t_1<\dots<t_N=T$ we introduce discrete points in time, where we uniformly assume $\Delta t=t_n-t_{n-1}$ for simplicity of notation only. Simple time-stepping schemes like the backward Euler method
\begin{equation}\label{ns:ie}
u^n + \Delta t f(t_n,u^n) = u^{n-1}
\end{equation}
find their counterpart in Galerkin time-stepping methods such as introduced by Eriksson,  Estep, Hansbo, and Johnson~\cite{ErikssonEstepHansboJohnson1995} as well as Thom{\'e}e~\cite{Thomee1997}. Time-stepping methods and Galerkin methods must still be seen as two distinct methods. Introducing the discontinuous space
\begin{equation}
  U_h = \{ v\in L^2(I) | \; v(t_0)\in\mathbb{R},\; v\big|_{(t_{n-1},t_n]}\in\mathbb{R},\; n=1,\dots,N\},
\end{equation}
the backward Euler method~\eqref{ns:ie} can be formulated as to find $u_h\in U_h$ such that
\begin{equation}
A_h(u_h)(v_h) = 0\quad\forall v_h\in V_h=U_h
\end{equation}
with the fully discrete variational formulation
\begin{equation}\label{var:be}
  A_h(u_h)(v_h) =\left(u_h(t_0)-u_0\right)\cdot v_h(t_0) +  \sum_{n=1}^N \left(u_h(t_n)-u_h(t_{n-1}) + \Delta t f\left(t_n,u(t_n)\right)\right)\cdot v_h(t_n). 
\end{equation}
In contrast, the corresponding discontinuous Galerkin approach also defines $u_h\in U_h$ given by
\begin{equation}
A(u_h)(v_h) = 0\quad\forall v_h\in U_h,
\end{equation}
where the continuous variational form is 
\begin{multline}\label{var:dg}
  A(u_h)(v_h) =\left(u_h(t_0)-u_0\right)\cdot v_h(t_0)\\
  +  \sum_{n=1}^N\left\{ \left(u_h(t_n)-u_h(t_{n-1}\right)\cdot v_h(t_n)
  +\int_{t_{n-1}}^{t_n} f(t,u_h(t))\cdot v_h(t)\right\}.
\end{multline}
Both forms,~\eqref{var:be} and~\eqref{var:dg}, only differ by numerical quadrature. However, as~\eqref{var:be} is based on the box rule which has the same order of convergence as the backward Euler method itself, both discrete approaches must be considered substantially different.

The consistency error can be introduced as $\mathcal{S}_h(u)(z) = A_h(u)(z)-A(u)(z)$ and for the Euler method we obtain
\begin{equation}\label{var:dg_consistency}
  \begin{aligned}
    |S(u_h)(z_h)| &= \sum_{n=1}^N \int_{t_{n-1}}^{t_n} \Delta t\cdot |z_h(t_n)|\cdot \left| f(t_n,u_h(t_n)) - \int_{t_{n-1}}^{t_n} f(t,u_h(t))\,\text{d}t\right|\\
    &= \sum_{n=1}^N \int_{t_{n-1}}^{t_n} \Delta t^2\cdot |z_h(t_n)|\cdot \left| f'(\xi_n) + O(\Delta t)\right|,
  \end{aligned}
\end{equation}
as $u_h(t)=u_h(t_n)$ for $t\in (t_{n-1},t_n]$ and assuming that
$f(\cdot)$ is differentiable. This yields the same order as the
truncation error of the backward Euler method itself. The
consistency term cannot be neglected, but its evaluation must be
included as part of the error identity.
  
If efficient time-stepping methods are used in simulations and a posteriori error estimation based on the DWR method is applied, the consistency error hence has to be tracked. For details and applications to flow problems discretized by the Crank-Nicolson scheme 
and Fractional-Step-$\theta$ method, 
we refer to~\cite{MeiRi14,MeidnerRichter2015,Ri17_fsi}.

\subsection{Estimating the Non-Conformity Error}
\label{SubSec:conformity}

The application to non-conforming discretizations is more cumbersome and there is no one standard approach. This is already due to the fact that even the definition of the error $u-u_h$ can be difficult. This is the case, for example, if $u$ is given on a domain $\Omega$ and $u_h$ on a discretized domain $\Omega_h$. In this case, one remedy could be to limit all considerations to the intersection $\Omega\cap\Omega_h$, see for instance~\cite{MinakowskiRichter2019}.

Since the sources of non-conformity are manifold, we discuss two examples in the following, first classical non-conforming finite element approaches and then discretizations that are not based on variational principles at all.

\subsubsection{Non-Conforming Finite Elements}

Let $u\in U$ be the solution to
\begin{equation}
  \mathcal{A}(u)(v) = 0\quad\forall v\in V.
\end{equation}
By $U_h\not\subset U$ and $V_h\not\subset V$ we denote a
non-conforming finite element discretization. As specific problem, 
consider a $H^1$-elliptic problem in mind with $U=V=H^1_0(\Omega)$
and where $U_h=V_h$ is the non-conforming Crouzeix-Raviart element. We
must assume that the variational form $\mathcal{A}(\cdot)(\cdot)$ is
also defined in $U_h\times U_h$ and elliptic. Hence, we introduce
$X=U\cup U_h$ and consider $\mathcal{A}:X\times X\to\mathbb{R}$ and,
likewise, we assume that the functional is defined on
$J:X\to\mathbb{R}$, which is a limitation as, for instance, line
integrals along mesh edges are not well defined on $U_h$.

The direct application of Theorem~\ref{thm: error identity} is not possible, since the fact that the residual of the continuous solution disappears in a conformal discretization was exploited here, namely $\mathcal{L}'(u,z)(u-u_h,z-z_h)=0$. 
This is not necessarily the case for a non-conforming discretization. Instead, what remains is the term
\begin{equation}
  \mathcal{L}'(u,z)(u-u_h,z-z_h)=
  \frac{1}{2}\left(
  -\mathcal{A}(u)(z-z_h) + J'(u)(u-u_h) - \mathcal{A}'(u)(u-u_h,z)\right). 
\end{equation}
Its meaning must be considered on a case-by-case basis. For
Poisson's problem $\mathcal{A}(u)(v) = (\nabla u,\nabla v)-(f,v)$, the error identity from Theorem~\ref{thm: error identity} reads as
\begin{multline}
  J(u-u_h) = \frac{1}{2}\left(\rho(u_h)(u-u_h)+\rho^*(u_h,z_h)(z-z_h)\right) - \rho(u_h)(z_h)\\
  \frac{1}{2}\left(
  (f,z-z_h) - \left(\nabla u,\nabla (z-z_h)\right)+ J'(u)(u-u_h) - \left(\nabla (u-u_h),\nabla z\right)\right)
  +\mathcal{R}^{(3)}. 
\end{multline}
Using integration by parts, these terms get
\begin{equation}
  \begin{aligned}
    (f,z-z_h) - \left(\nabla u,\nabla (z-z_h)\right) &= \sum_{{K\in \mathcal{T}_h}} \langle \partial_n u,[z_h]\rangle_{{\partial K}}\\
    J'(u)(u-u_h) - \left(\nabla (u-u_h),\nabla z\right)&= \sum_{{K\in \mathcal{T}_h}} \langle \partial_n z,[u_h]\rangle_{{\partial K}},
  \end{aligned}
\end{equation}
where $[\cdot]$ denotes the jump over the element's edge $\partial K$ 
and $\mathcal{T}_h$ is the decomposition of $\Omega_h$. 
This term exactly measures the non-conformity of the discrete approximations $u_h$ and $z_h$. 
Instead of a rigorous bound, it can be estimated based on replacing $\partial_n u$ and $\partial_n z$ by discrete reconstructions, see~\cite{Grajewski} for a similar procedure. 
We refer to~\cite{MaVohYou20} for a detailed presentation of a goal-oriented error estimator for different conforming and non-conforming discretizations.

\subsubsection{Non-Variational Discretization Methods}

Finally, we explore 
to which extend the error estimator can be
applied to discretizations that cannot be written in the form of a
variational formulation at all. These could be, for example, finite
difference methods, or the approximation of differential equations
with neural networks like \emph{Physics Informed Neural Networks
(PINNs)} such as the \emph{Deep Ritz} method~\cite{EYu2017}. In both cases, the discrete approximation $u_h$ cannot be represented by a variational formulation.
Once again, we cannot rely on a uniform theoretical principle, but must argue on a case-by-case basis.

What we require in any case is the embedding of the discrete solution
$u_h$ into a function space $\mathcal{E}:u_h\mapsto U_h$ that is in some sense compatible, meaning that it is either conforming $U_h\subset U$ or that the variational form 
can be extended onto $X:=U\cup U_h$ such as in the case of the non-conforming Crouzeix-Raviart element. Considering finite difference approximations $\mathcal{E}(u_h) = I_h u_h$, the interpolation into a usual finite element space is one conforming option. In this setting, the original error identity from Theorem~\ref{thm: error identity} can directly be applied:
\begin{Theorem}[Error identity for conforming embeddings]
Let $u\in U$ and $z\in V$ be the primal and dual solutions. Let $u_h$ and $z_h$ be any discrete approximations. Further, let
\begin{equation}
\mathcal{E}:u_h\mapsto U_h\subset U,\quad
\mathcal{E}^*:z_h\mapsto V_h\subset V,
\end{equation}
be embeddings of the discrete approximations into subspaces. Then, it holds
\begin{multline}
J(u)-J(\mathcal{E}(u_h)) = \frac{1}{2}
\left(
\rho\left(\mathcal{E}(u_h)\right)\left(z-\mathcal{E}^*(z_h)\right)+
\rho^*\left(\mathcal{E}(u_h),\mathcal{E}^*(z_h)\right)\left(u-\mathcal{E}(u_h)\right)\right)\\
-\rho\left(\mathcal{E}(u_h)\right)\left(\mathcal{E}^*(z_h)\right) + \mathcal{R}^{(3)}.
\end{multline}
\end{Theorem}
Like all previous error identities, this one suffers from the need 
to approximate numerically the primal and adjoint solutions $u\in U$ and $z\in V$. Furthermore, the embeddings $\mathcal{E}(u_h)$ and $\mathcal{E}^*(z_h)$ must be available for numerical evaluation. But most important, the error estimator is still based on variational principles although the discrete approximations $u_h$ and $z_h$ can be obtained in completely different ways. To evaluate the error, integrals over the computational domains containing the embeddings $\mathcal{E}(u_h)$ and $\mathcal{E}^*(z_h)$ must be computed. This will either require a mesh, which is all but trivial, if the choice of discretization method - e.g. smoothed particle hydrodynamics - was motivated as being mesh-free, or, if, for instance, neural networks are chosen for high-dimensional problems. Resorting to simple Monte Carlo quadrature avoids this problem, but it brings along a substantial quadrature error. We refer to~\cite{minakowski2021error} for a first discussion on this mostly open topic.

%
%
\section{Multi-Goal Oriented Error Control}
\label{Sec:Multigoal}

In the past two decades, numerous efforts regarding the efficient and robust numerical solution of time-dependent, non-linear, coupled partial differential 
equations (PDE) have been made. Often, such PDE systems arise in continuum mechanics from conservation laws such as conservation of 
mass, momentum, angular momentum, and energy.
Examples are porous media applications \cite{chavent1986mathematical,LeSchref99,Coussy2004}, multi-phase flow and heat transfer processes \cite{FaZha20}, and
fluid-structure interaction \cite{GaRa10,BaTaTe13,Ri17_fsi}. In addition, such multiphysics problems 
may be subject to inequality constraints, which result into 
coupled variational inequality systems (CVIS), for which multiphysics 
phase-field fracture is an example \cite{Wi20}. 
In such problems, often more than one single 
goal functional shall be controlled 
and measured up to a certain accuracy. This is the 
motivation
for multigoal-oriented error control and adaptivity that we will 
explain in this section.

\subsection{Combined Goal Functional}
Let us assume that we are interested in the evaluation of $N$ functionals, which we denote by 
$J_1, J_2, \ldots,J_N$. This can also be seen as a vector valued quantity of interest 
\begin{equation} \label{eq: def vecJ}
\vec{J}(v):=(J_1(v),J_2(v),\cdots, J_{N}(v) ).
\end{equation}
There are a lot of works on multi-goal oriented adaptivity \cite{ahuja2022multigoal,HaHou03,Ha08,KerPruChaLaf2017,KerPruChaLaf2017,PARDO20101953,EnWi17,EnLaWiPAMM18,EnLaWi18,BruZhuZwie16,endtmayerPAMM2022,BeEnLaWi2021a,BeEnWi2021P}. The aim of this section is to combine the functionals 
$J_1, J_2, \ldots,J_N$ into a single functional and apply the results from Section~\ref{Sec:single-goal}.
A simple idea to combine this functionals is to add them up.
We define $J_A$ as 
\begin{equation}\label{eq: def Addfunctionals}
	J_A := \sum_{i=1}^N J_i.
\end{equation}
This functional allows us to combine the functionals. 
Let us  assume that $u$ solves the primal problem \eqref{eq: General Modelproblem}, and let $\tilde{u} \in U_h^{(2)}$ be approximation to $u_h$ solving \eqref{eq: discrete primal}.
For instance, let us assume $N=3$ and 
the error $J_1(u)-J_1(\tilde{u})=1$, $J_2(u)-J_2(\tilde{u})=1$ and $J_3(u)-J_3(\tilde{u})=-2$. Then $J_A(u)-J_A(\tilde{u})=0$. This shows that $J_A$ is not suitable, since we get that the error vanishes even though the error does not vanish in any of the original functionals.

\subsection{Error-Weighting Function}
In this part, we introduce a weighting function to balance the sum of the different 
goal functionals. This function includes a sign evaluation in order to avoid 
cancellation of goal functionals with similar values, but different signs.

\begin{Definition}[Error-Weighting function]
	Let $ M \subseteq \mathbb{R}^N$. 
	The function $\mathfrak{E}: (\mathbb{R}^+_0)^N \times  M \mapsto \mathbb{R}^+_0$ is an \textit{error-weighting
		function} if  $\mathfrak{E}(\cdot,m) \in
	\mathcal{C}^3((\mathbb{R}^+_0)^N,\mathbb{R}^+_0)$ is strictly monotonically
	increasing in each component and $\mathfrak{E} (0,m)=0$ for all
	$m \in M$.
\end{Definition}
Examples for such error weighting functions are
\begin{equation}\label{eq:error weighting function:example 1}
\mathfrak{E}(x,m):=\sum_{i=1}^N \frac{x_i}{|m_i|},
\end{equation}
\begin{equation}\label{eq:error weighting function:example 2}
\mathfrak{E}(x,m):=\sum_{i=1}^N x_i,
\end{equation}
\begin{equation}\label{eq:error weighting function:example 3}
	\mathfrak{E}(x,m):=\sum_{i=1}^N \frac{x_i^p}{|m_i|^p} \qquad p \in (1,\infty),
\end{equation}
and
\begin{equation}\label{eq:error weighting function:example 4}
	\mathfrak{E}(x,m):=\sum_{i=1}^N \sqrt{x_i},
\end{equation}
for $x,m \in (\mathbb{R}^+_0)^N \times  M$. 
These functions should mimic some kind of norm or metric. Here, the elements $m\in M$ are used to weight the contributions. These are user chosen weights, or weights balancing the relative errors instead of the absolute ones.
Finally, we define the {error functional} as follows
\begin{align}\label{eq: TrueErrorrepresentationFunctional}
	\tilde{J}_{\mathfrak{E}}(v):=\mathfrak{E}(|\vec{J}(u)-\vec{J}(v)|_N, m) \qquad \forall v \in U,
\end{align}
where, $|\cdot|_N$ describes the component wise absolute value. 
It follows from the definition of $\mathfrak{E}$ that $J_\mathfrak{E}(v) \in \mathbb{R}^+_0$ for all $v \in V$. 

\begin{Remark}
	The error functional $\tilde{J}_{\mathfrak{E}}(v)$ mimics a semi-metric (as in \cite{SierpinskiTopo52,KhaKirMetric2001}) for the errors in the functionals $J_i$ Hence, $\tilde{J}_{\mathfrak{E}}(v)$ represents a
	semi-metric, which ensures that $\tilde{J}_{\mathfrak{E}}$  is
	monotonically increasing if $|J_i(u)-J_i(\tilde{u})|$ is monotonically
	increasing.	
\end{Remark}
We notice that $\tilde{J}_{\mathfrak{E}}(v)$ as defined in \eqref{eq: TrueErrorrepresentationFunctional} is not computable, since we do not know the exact solution $u$. 
And if we were to know $u$, we also would know all $J_i(u)$, for which we would not need any 
finite element simulations. In Section~\ref{Sec:single-goal}, we had a similar problem. There, we introduced the enriched solutions $u_h^{(2)}$ solving \eqref{eq: enriched primal}, which was used to replace all $u$ in the error identity of Theorem~\ref{thm: Error Representation}. To this end, 
we apply the same approach and we define the computable error functional
\begin{align}\label{eq: ErrorrepresentationFunctional}
	{J}_{\mathfrak{E}}(v):=\mathfrak{E}(|\vec{J}(u_h^{(2)})-\vec{J}(v)|_N,m) \qquad \forall v \in U,
\end{align}
where $u_h^{(2)}$ solves the enriched primal problem \eqref{eq: enriched primal}.
This functional allows us to use the ideas from 
Section~\ref{Sec:single-goal}
with $J=J_\mathfrak{E}$.
From \eqref{eq: ErrorrepresentationFunctional} we conclude that
\begin{equation} \label{eq: Errorweighting-function def}
{J}_{\mathfrak{E}}'(v)(\delta v)=\sum_{i=1}^{N}\operatorname{sign}\left(J_i(u_h^{(2)})-J_i(v)\right)\frac{\partial \mathfrak{E}}{\partial x_i}(|\vec{J}(u_h^{(2)})-\vec{J}(v)|_N,m) \qquad \forall v,\partial v \in U,
\end{equation}
The value $\operatorname{sign}\left(J_i(u)-J_i(v)\right)$ which is approximated by the 
$\operatorname{sign}\left(J_i(u_h^{(2)})-J_i(v)\right)$ in \eqref{eq: Errorweighting-function def} is important to avoid error cancellations. In \cite{Ha08,HaHou03}, the sign was approximated using an adjoint to adjoint (dual to dual) problem, whereas in a prior work \cite{EnLaWi18} the sign was approximated using the difference $J_i(u_h^{(2)})-J_i(\tilde{u})$. For a linear partial differential equation and linear goal functionals those two methods coincide. 

Common choices are $m=|\vec{J}(\tilde{u})|_N$ or $m=|\vec{J}(u_h^{(2)})|_N$. Together with the error function \eqref{eq:error weighting function:example 1}, this choice aims for a similar relative error in all the functionals. In contrast the error function \eqref{eq:error weighting function:example 2} aims for a similar absolute error. Furthermore, the error function \eqref{eq:error weighting function:example 3} penalizes larger relative errors, and the error function \eqref{eq:error weighting function:example 4} aims for a similar decrease of the error in all the functionals.
The choice $m_i=\omega_i |J_i(\tilde{u})| $, {where $\omega_i \in\mathbb{R}^+_0$} are user chosen weights,
in combination with the error functional \eqref{eq:error weighting function:example 1}, leads to an error functional with almost the same properties as $J_c$ in \cite{HaHou03,Ha08,EnLaWi18,EnLaNeiWoWiPAMM2019,EnLaNeiWoWi2020,EnLaWiPAMM18,EnWi17,endtmayerPAMM2022} defined as
\begin{equation} \label{eq: def J_c}
	J_c(v):=\sum_{i=1}^{N}w_i J_i(v),
\end{equation}
where $w_i:=\omega_i\operatorname{sign}\left(J_i(u_h^{(2)})-J_i(\tilde{u})\right)/|J_i(\tilde{u})|$.

\subsection{Error Localization}
Since we work with the combined goal functional, the localization procedure 
can be carried out in the same fashion as previously described in 
Section~\ref{SubSec:Localization}. An immediate consequence 
of Proposition \ref{prop_DWR_practical} is
\begin{Proposition}[Practical error estimator for  the functional $J_\mathfrak{E}$]
\label{prop_DWR_practical_multigoal}
Let $\tilde{u}\in U_h$ be a low-order approximation to \eqref{eq: discrete primal}, $u_h^{(2)}\in U_h^{(2)}$ the higher-order solution to \eqref{eq: enriched primal}, and
$\tilde{z} \in U_h$ be a low-order approximation to \eqref{eq: discrete adjoint} $z_h^{(2)}\in V_h^{(2)}$ the higher-order adjoint solutions \eqref{eq: enriched Adjointproblem},
respectively.
The practical localized PU error estimator reads for the error functional reads
\[
J_\mathfrak{E}(u)-J_\mathfrak{E}(\tilde{u}) \approx \eta^{(2)} = \sum_{i=1}^M
\frac{1}{2}\rho(\tilde{u})((z_h^{(2)} - \tilde{z})\chi_i)
+\frac{1}{2}\rho^*(\tilde{u},\tilde{z})((u_h^{(2)} - \tilde{u})\chi_i)
+\rho(\tilde{u})(\tilde{z})
\]
where we now re-define the previous notation and obtain as error parts
\[
\eta^{(2)} = \eta_h^{(2)} + \eta_k := \sum_{i=1}^M (\eta_p + \eta_a) + \eta_k 
\]
with
\begin{align*}
\eta_p &:= \eta_p(i):= \frac{1}{2}\rho(\tilde{u})((z_h^{(2)} - \tilde{z})\chi_i),\\
\eta_a &:= \eta_a(i):= \frac{1}{2}\rho^*(\tilde{u},\tilde{z})((u_h^{(2)} - \tilde{u})\chi_i),\\
\eta_k &:= \rho(\tilde{u})(\tilde{z}).
\end{align*}
\end{Proposition}

\subsection{Adaptive Algorithm}
Summarizing the previous ingredients allows us to formulate multigoal algorithms 
for adaptivity in which non-linear iteration errors are balanced with discretization 
errors.

\begin{algorithm}[H]
	\caption{The adaptive multigoal algorithm} \label{Alg: adapt. multigoal Algorithm}
	\begin{algorithmic}[1]
		\Procedure{Multigoal Adaptive}{$\vec{J}$, $\mathcal{A}$, $\mathcal{T}_0$, TOL, maxNDoFs}      
		\State $k\leftarrow 0$,$\mathcal{T}_k\leftarrow \mathcal{T}_0$, $\eta_h\leftarrow \infty$
		\While{$\eta_h> 10^{-2}$ TOL \& $|\mathcal{T}_k|$ $\geq$ maxNDoFs }
		\State Solve \eqref{eq: enriched primal} to obtain $u_h^{(2)}$ with some non-linear solver like Newton's method \label{Alg: Multi: Step uh2}
		\State Solve \eqref{eq: discrete primal} to obtain $\tilde{u}$ with some non-linear solver like Newton's method \label{Alg: Multi: Step uh}
		\State Construct $J_\mathfrak{E}$.
		\State Solve \eqref{eq: enriched Adjointproblem} and \eqref{eq: discrete adjoint} to obtain $z_h^{(2)}$ and $\tilde{z}$ using some linear solver.
		\State Compute $\eta_h$ and  the node-wise error contribution $\eta_i^{PU}$ as in \eqref{eq: eta_iPU}.
		\State Distribute $\eta_i^{PU}$ equally to all elements that share the node $i$.
		\State Mark the elements with some marking strategy like D\"orfler marking \cite{Doerfler:1996a}.
		\State Refine the mesh according to the marked cells
		\State $k\leftarrow k+1$
		\EndWhile
		\Return $\vec{J}(u_k)$
		\EndProcedure
	\end{algorithmic}	
\end{algorithm}
\begin{Remark}
Since we have to approximate $\operatorname{sign}\left(J_i(u_h^{(2)})-J_i(v)\right)$,
	we only provide an algorithm using enriched spaces.
\end{Remark}
\begin{Remark}
	If the non-linear solver is stopped using the iteration error estimator $\eta_k$, we require $u_h^{(2)}$ before $\tilde{u}$. 
	Therefore, in contrast to Algorithm~\ref{Alg: adapt. Alg}, in Algorithm~\ref{Alg: adapt. multigoal Algorithm} Line~\ref{Alg: Multi: Step uh2} and \ref{Alg: Multi: Step uh} must not be swapped.
\end{Remark}

\section{Applications}
\label{sec_applications}
In this section, we substantiate 
our theoretical results and numerical algorithms
from the previous sections with the help of 
four numerical examples. These include 
linear and non-linear stationary settings as well as a non-linear space-time 
test. All examples are evaluated with the help of state-of-the-art 
measures by observing error reductions, reductions in the estimators, 
effectivity indices and indicator indices. These demonstrate the performance 
of goal-oriented adaptivity for different types of discretizations.
The numerical tests are computed with the open-source finite element 
libraries deal.II \cite{BangerthHartmannKanschat2007,dealII90,deal2020} 
and MFEM \cite{mfem:2021,mfem-web}. Parts of our code developments 
are open-source on GitHub\footnote{\url{https://github.com/tommeswick/}}.

\subsection{Poisson's Problem}
\label{Subsec: Example1}
The purpose of this first numerical example is to illustrate some of the 
previous algorithmic and theoretical developments with the help
of a numerical experiment including open-source code developments. We span 
from the problem statement, over discretization, numerical solution to 
goal functional evaluations on adaptively refined meshes using 
algorithms from Section \ref{sec_adaptive_algorithms_single_goal}.
To this end, we employ a simple example, namely Poisson's problem in two dimensions. In the numerical simulations, we also demonstrate trivial 
effects such as that the dual space must be richer than the primal function space for the error interpolations, as otherwise the error estimator is identically zero due to Galerkin orthogonality; cf. \eqref{ns:id2}.

\subsubsection{Open-Source Programming Code}
The code is based on 
deal.II \cite{BangerthHartmannKanschat2007,deal2020} and the current release version 9.5.1 
\cite{dealII95}. Furthermore, this code is 
published open-source on GitHub\footnote{\url{https://github.com/tommeswick/PU-DWR-Poisson}}.
Conceptually, this code builds directly upon \cite{RiWi15_dwr}.

\subsubsection{Problem Statement: PDE and Boundary Conditions}
Let $\Omega:=(0,1)^2$ be the domain with 
the Dirichlet boundary $\partial\Omega$.
Let $f:\Omega\to\mathbb{R}$ be some volume force. Find $u:\bar{\Omega}\to\mathbb{R}$ such that
\begin{align*}
- \Delta u &= f \quad\text{in } \Omega, \\
u &= 0 \quad\text{on } \partial \Omega, 
\end{align*}
where $f=-1$. 

\subsubsection{Weak Form, Discretization, Numerical Solution}
We define the function spaces, here $U = V:= H^1_0(\Omega)$. Then, the weak 
form reads: Find $u\in U$ such that
\[
\mathcal{A}(u)(\psi) = l(\psi) \quad\forall \psi\in V,
\]
with
\[
\mathcal{A}(u)(\psi):= \int_{\Omega} \nabla u\cdot \nabla \psi\, dx, 
\quad \mbox{and} \quad 
l(\psi) := \int_{\Omega} f\psi\, dx.
\]
The discrete problem is formulated on quadrilateral elements (Section \ref{sec_FEM_Qk})  
as decomposition of the 
domain $\Omega$ and reads: Find $u_{h}\in U_h$ such that
\[
\mathcal{A}(u_{h})(\psi_{h}) = l(\psi_{h}) \quad\forall \psi_{h}\in V_h,
\]
with
\[
\mathcal{A}(u_{h})(\psi_{h}):= \int_{\Omega} \nabla u_{h}\cdot \nabla \psi_{h}\, dx, 
\quad \mbox{and} \quad 
l(\psi_{h}) := \int_{\Omega} f\psi_{h}\, dx.
\]
For implementation reasons, the numerical solution is obtained 
within a Newton scheme
in which the linear equations are solved with a direct solver (UMFPACK \cite{UMFPACK}).
Clearly, the problem is linear and for this reason Newton's method converges within one 
single iteration. The main reason using a 
Newton method is that the code can easily be extended 
to non-linear problems. 

\subsubsection{Goal Functional (Quantity of Interest) and Adjoint Problem}
\index{Goal functional}
\index{Quantity of interest}
Furthermore, the goal functional 
is given as a point evaluation in the middle point:
\[
J(u) = u(0.5,0.5).
\]
For the later comparison, the reference value is determined as
\begin{equation}
\label{eq_ref_value_736}
u_{ref}(0.5,0.5) := -7.3671353258859554e-02,
\end{equation}
and was obtained on a globally refined super-mesh.

The adjoint problem is derived as in Equation \eqref{eq: General Adjointproblem}
\[
\text{Find } z\in V: \quad  \mathcal{A}(\varphi,z) = J(\varphi) \quad\forall \varphi\in U,
\]
where we notice that both the left hand side and right sides are linear. Specifically, 
these are given as:
\begin{equation*}
\mathcal{A}(\varphi,z) = \int_{\Omega} \nabla \varphi\cdot \nabla z\, dx, \quad
J(\varphi) = \varphi(0.5,0.5).
\end{equation*}
The adjoint is solved similar to the primal problem, namely with
Newton's method (converging
in one single iteration) and UMFPACK are employed again. Here, Newton's method 
is not at all necessary since 
the adjoint is always linear as previously discussed. However, for implementation convenience, we employ 
the same Newton solver for the primal and adjoint problem, which constitutes solely our reason.

\subsubsection{Objectives}
The objectives of our studies are:
\begin{itemize}
\item Evaluating goal-oriented exact error $J(u_{ref}) - J(u_{h})$;
\item Evaluating error estimator $\eta$;
\item Computing $I_\mathrm{eff}$ and $I_{ind}$;
\item Employing PU-DWR for local mesh adaptivity;
\item Showing necessity of higher-order information of the adjoint 
$z$ in $z_{h}^{high} - z_{h}$ in the dual-weighted residual estimator;
\item Employing higher-order shape functions;
\item Employing higher-order PU function.
\end{itemize}

\subsubsection{Discussion and Interpretation of Our Findings}
In this section, we conduct in total ten numerical experiments to 
investigate our previous objectives.

\paragraph{Computation 1: $u \in Q_1$ and $z \in Q_1$}
\index{Galerkin orthogonality}
In this first numerical test (Table \ref{tab_1456_1}), the adjoint has the same order as the primal solution,
and due to Galerkin orthogonality (see e.g., \eqref{ns:id2}), the estimator $\eta$ is identical to zero
and consequently, the adaptive algorithm stops after the first iteration, 
which is certainly not what we are interested in.

\begin{table}[h!]
\begin{center}
{
\begin{verbatim}
================================================================================
Dofs    Exact err       Est err         Est ind         Eff             Ind
--------------------------------------------------------------------------------
18      2.01e-02        0.00e+00        0.00e+00        0.00e+00        0.00e+00
================================================================================
\end{verbatim}
}
\caption{Example 1: Computation 1.}
\label{tab_1456_1}
\end{center}
\end{table}

\paragraph{Computation 2: $u \in Q_1$ and $z \in Q_2$}
In this second experiment (Table \ref{tab_1456_2}), 
we choose now a higher-order adjoint solution and obtain 
a typical adaptive loop. In the true (exact) error and the estimator $\eta$, we observe 
quadratic convergence as to be expected. The $I_\mathrm{eff}$ and $I_{ind}$ are around the 
optimal value $1$. These findings are in excellent agreement with similar results 
published in the literature.

\begin{table}[h!]
\begin{center}
{
\begin{verbatim}
================================================================================
Dofs    Exact err       Est err         Est ind         Eff             Ind
--------------------------------------------------------------------------------
18      2.01e-02        2.00e-02        2.00e-02        9.98e-01        9.98e-01
50      4.01e-03        4.03e-03        4.75e-03        1.00e+00        1.18e+00
162     9.27e-04        9.28e-04        1.17e-03        1.00e+00        1.26e+00
482     2.37e-04        2.44e-04        3.08e-04        1.03e+00        1.30e+00
1794    5.82e-05        5.91e-05        7.52e-05        1.02e+00        1.29e+00
4978    1.73e-05        1.92e-05        2.34e-05        1.11e+00        1.35e+00
14530   5.13e-06        6.02e-06        7.74e-06        1.18e+00        1.51e+00
================================================================================
\end{verbatim}
}
\caption{Example 1: Computation 2.}
\label{tab_1456_2}
\end{center}
\end{table}

\paragraph{Computation 3: $u \in Q_1$ and $z \in Q_3$}

In this third example (Table \ref{tab_1456_3}), 
we increase the adjoint polynomial order and observe that 
we obtain roughly the same error tolerance of about $10^{-6}$ 
with less degrees of freedom. The $I_\mathrm{eff}$ is still optimal, while $I_{ind}$ shows 
a slight overestimation with $I_{ind} \approx 2$.

\begin{table}[h!]
\begin{center}
{
\begin{verbatim}
================================================================================
Dofs    Exact err       Est err         Est ind         Eff             Ind
--------------------------------------------------------------------------------
18      2.01e-02        2.01e-02        3.43e-02        9.99e-01        1.71e+00
50      4.01e-03        4.01e-03        8.57e-03        1.00e+00        2.14e+00
162     9.27e-04        9.27e-04        1.99e-03        1.00e+00        2.14e+00
274     3.33e-04        3.98e-04        7.02e-04        1.20e+00        2.11e+00
994     8.45e-05        9.16e-05        1.67e-04        1.08e+00        1.98e+00
2578    2.39e-05        2.65e-05        4.64e-05        1.11e+00        1.94e+00
8482    6.00e-06        6.35e-06        1.21e-05        1.06e+00        2.01e+00
================================================================================
\end{verbatim}
}
\caption{Example 1: Computation 3.}
\label{tab_1456_3}
\end{center}
\end{table}

\paragraph{Computation 4: $u \in Q_1$ and $z \in Q_4$}
In this fourth numerical test (Table \ref{tab_1456_4}), the results are close to the previous setting, which 
basically shows that such high-order adjoint solutions, do not increase necessarily 
anymore the error estimator and adaptivity.

\begin{table}[h!]
\begin{center}
{
\begin{verbatim}
================================================================================
Dofs    Exact err       Est err         Est ind         Eff             Ind
--------------------------------------------------------------------------------
18      2.01e-02        2.01e-02        4.64e-02        9.99e-01        2.31e+00
50      4.01e-03        4.01e-03        1.02e-02        1.00e+00        2.54e+00
162     9.27e-04        9.27e-04        2.45e-03        1.00e+00        2.65e+00
482     2.37e-04        2.44e-04        6.25e-04        1.03e+00        2.63e+00
802     9.72e-05        1.12e-04        2.10e-04        1.16e+00        2.16e+00
2578    2.39e-05        2.65e-05        5.21e-05        1.11e+00        2.18e+00
8290    6.12e-06        6.57e-06        1.37e-05        1.07e+00        2.24e+00
================================================================================
\end{verbatim}
}
\caption{Example 1: Computation 4.}
\label{tab_1456_4}
\end{center}
\end{table}

\paragraph{Computation 5: $u \in Q_2$ and $z \in Q_1$}
The fifth numerical experiment (Table \ref{tab_1456_5}) has the inverted polynomial order in which 
Galerkin orthogonality is even more violated than in the $Q_1/Q_1$ case. Our
numerical results confirm the theory.

\begin{table}[h!]
\begin{center}
{
\begin{verbatim}
================================================================================
Dofs    Exact err       Est err         Est ind         Eff             Ind
--------------------------------------------------------------------------------
50      4.66e-05        0.00e+00        0.00e+00        0.00e+00        0.00e+00
================================================================================
\end{verbatim}
}
\caption{Example 1: Computation 5.}
\label{tab_1456_5}
\end{center}
\end{table}

\paragraph{Computation 6: $u \in Q_2$ and $z \in Q_2$}
In the sixth test (Table \ref{tab_1456_6}), we are in the same situation with equal-order polynomials 
as in the $Q_1/Q_1$ case, and consequently, again due to Galerkin orthogonality 
the estimator $\eta$ is zero.

\begin{table}[h!]
\begin{center}
{
\begin{verbatim}
================================================================================
Dofs    Exact err       Est err         Est ind         Eff             Ind
--------------------------------------------------------------------------------
50      4.66e-05        0.00e+00        0.00e+00        0.00e+00        0.00e+00
================================================================================
\end{verbatim}
}
\caption{Example 1: Computation 6.}
\label{tab_1456_6}
\end{center}
\end{table}

\paragraph{Computation 7: $u \in Q_2$ and $z \in Q_3$}

We perform now a seventh experiment (Table \ref{tab_1456_7}), which again works. In the last row,
the $I_\mathrm{eff}$ is off, likely to the reason that the reference value \eqref{eq_ref_value_736} is not 
accurate enough anymore. This shows that numerically obtained reference values 
must be computed with care.

\begin{table}[h!]
\begin{center}
{
\begin{verbatim}
================================================================================
Dofs    Exact err       Est err         Est ind         Eff             Ind
--------------------------------------------------------------------------------
50      4.66e-05        2.31e-05        1.27e-03        4.96e-01        2.72e+01
162     1.98e-05        1.96e-05        1.09e-04        9.88e-01        5.47e+00
578     1.45e-06        1.44e-06        7.62e-06        9.96e-01        5.27e+00
1826    2.54e-08        3.31e-08        7.26e-07        1.31e+00        2.86e+01
6978    1.96e-09        7.72e-11        5.27e-08        3.94e-02        2.69e+01
================================================================================
\end{verbatim}
}
\caption{Example 1: Computation 7.}
\label{tab_1456_7}
\end{center}
\end{table}

\paragraph{Computation 8: $u \in Q_2$ and $z \in Q_4$}
In this eighth example (Table \ref{tab_1456_8}), the results are similar to the previous test case, which 
basically confirms the fourth example, that higher-order adjoint do not contribute 
to better findings anymore for this specific configuration.

\begin{table}[h!]
\begin{center}
{
\begin{verbatim}
================================================================================
Dofs    Exact err       Est err         Est ind         Eff             Ind
--------------------------------------------------------------------------------
50      4.66e-05        3.31e-05        1.33e-03        7.09e-01        2.85e+01
162     1.98e-05        1.98e-05        9.63e-05        1.00e+00        4.86e+00
578     1.45e-06        1.45e-06        6.88e-06        1.00e+00        4.75e+00
1826    2.54e-08        3.38e-08        6.88e-07        1.33e+00        2.71e+01
6978    1.96e-09        4.32e-11        5.00e-08        2.20e-02        2.55e+01
================================================================================
\end{verbatim}
}
\caption{Example 1: Computation 8.}
\label{tab_1456_8}
\end{center}
\end{table}

\paragraph{Computation 9: $u \in Q_1$ and $z \in Q_2$ and PU $Q_2$}
We finally conduct two tests with higher-order PU polynomial degrees. 
In Table \ref{tab_1456_9}, the $I_\mathrm{eff}$ performs very well, while the indicator index shows over estimation
of a factor $4$.

\begin{table}[h!]
\begin{center}
{
\begin{verbatim}
================================================================================
Dofs    Exact err       Est err         Est ind         Eff             Ind
--------------------------------------------------------------------------------
18      2.01e-02        2.00e-02        3.61e-02        9.98e-01        1.80e+00
50      4.01e-03        4.03e-03        1.30e-02        1.00e+00        3.24e+00
162     9.27e-04        9.28e-04        3.70e-03        1.00e+00        4.00e+00
482     2.37e-04        2.40e-04        1.05e-03        1.01e+00        4.40e+00
1106    7.73e-05        8.52e-05        3.27e-04        1.10e+00        4.23e+00
3810    2.01e-05        2.09e-05        8.62e-05        1.04e+00        4.29e+00
13250   5.39e-06        6.01e-06        2.39e-05        1.11e+00        4.43e+00
================================================================================
\end{verbatim}
}
\caption{Example 1: Computation 9.}
\label{tab_1456_9}
\end{center}
\end{table}

\paragraph{Computation 10: $u \in Q_1$ and $z \in Q_2$ and PU $Q_3$}
In this final test (Table \ref{tab_1456_10}), the indicator index shows 
an overestimation of a factor about 
$8$. In terms of the true error and estimated error as well as the $I_\mathrm{eff}$, the results 
are close to being optimal, while the indicator index shows an over estimation. 
With respect to a higher computational cost in computing the PU, these findings 
suggest that a low-order PU for this configuration is sufficient.

\begin{table}[h!]
\begin{center}
{
\begin{verbatim}
================================================================================
Dofs    Exact err       Est err         Est ind         Eff             Ind
--------------------------------------------------------------------------------
18      2.01e-02        2.00e-02        6.22e-02        9.98e-01        3.10e+00
50      4.01e-03        4.03e-03        2.18e-02        1.00e+00        5.43e+00
162     9.27e-04        9.28e-04        6.14e-03        1.00e+00        6.62e+00
482     2.37e-04        2.32e-04        1.69e-03        9.78e-01        7.11e+00
1602    6.21e-05        6.95e-05        4.61e-04        1.12e+00        7.42e+00
5618    1.59e-05        1.74e-05        1.22e-04        1.10e+00        7.71e+00
19602   4.32e-06        5.16e-06        3.31e-05        1.19e+00        7.67e+00
================================================================================
\end{verbatim}
}
\caption{Example 1: Computation 10.}
\label{tab_1456_10}
\end{center}
\end{table}

\paragraph{Graphical output}
Finally, from the \texttt{vtk} raw data, using \texttt{visit} \cite{visit}, 
we show some graphical solutions in Figure \ref{fig_DWR} that illustrate 
the performance of our algorithms.

\begin{figure}[H]
\begin{center}
  \includegraphics[width=5.2cm]{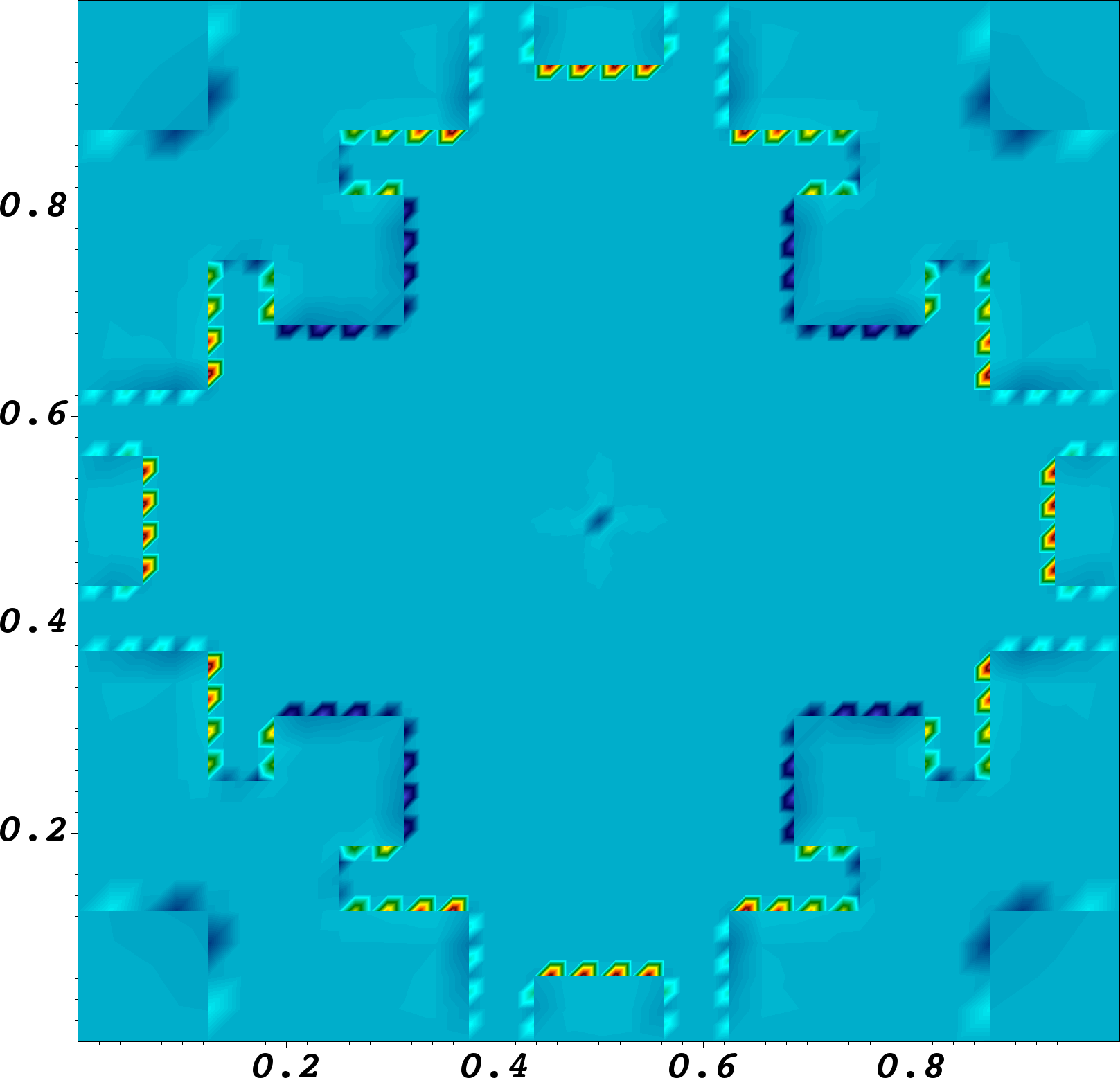}
  \includegraphics[width=5.2cm]{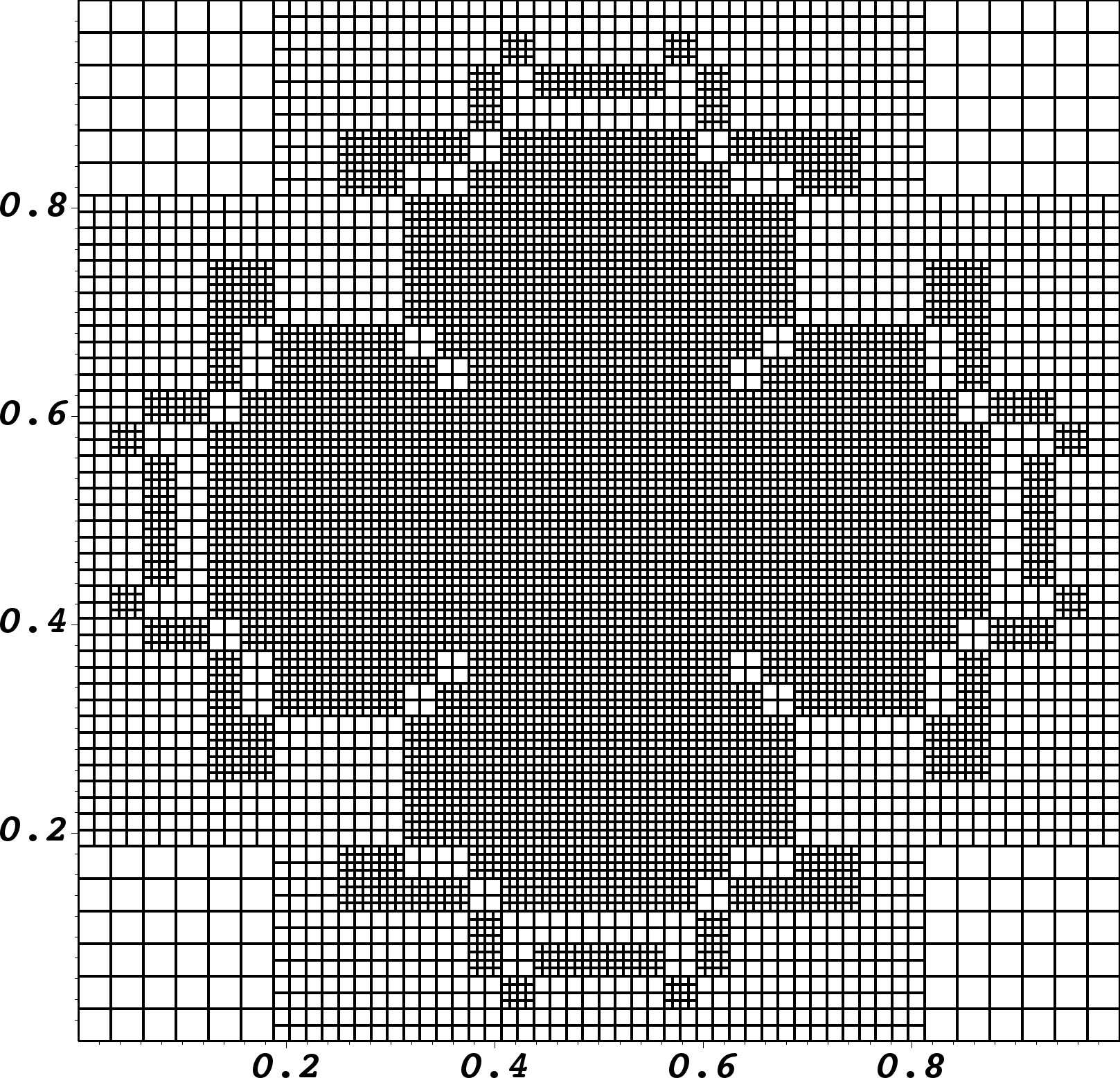}\\[0.5em]
  \includegraphics[width=5.2cm]{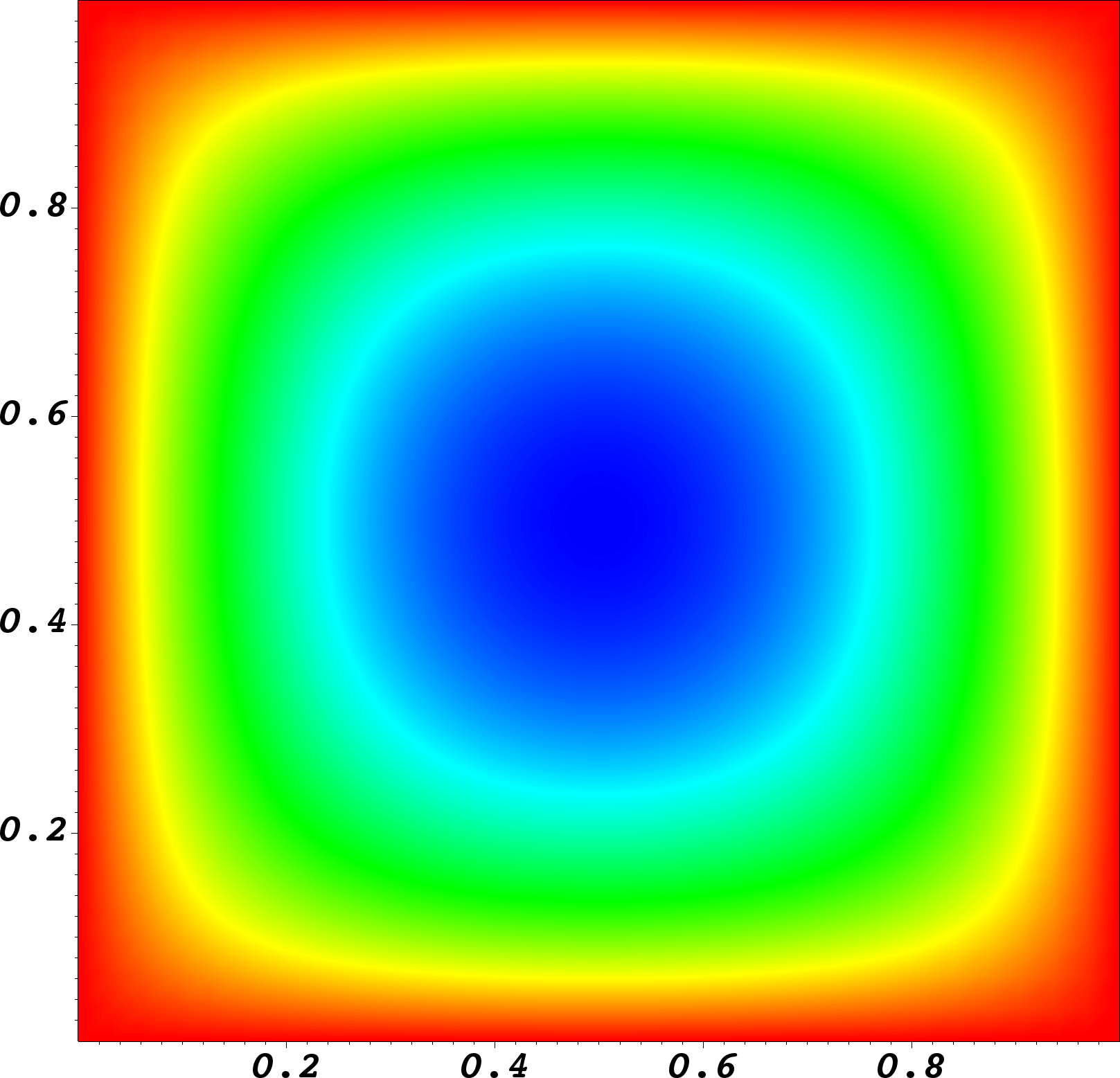}
  \includegraphics[width=5.2cm]{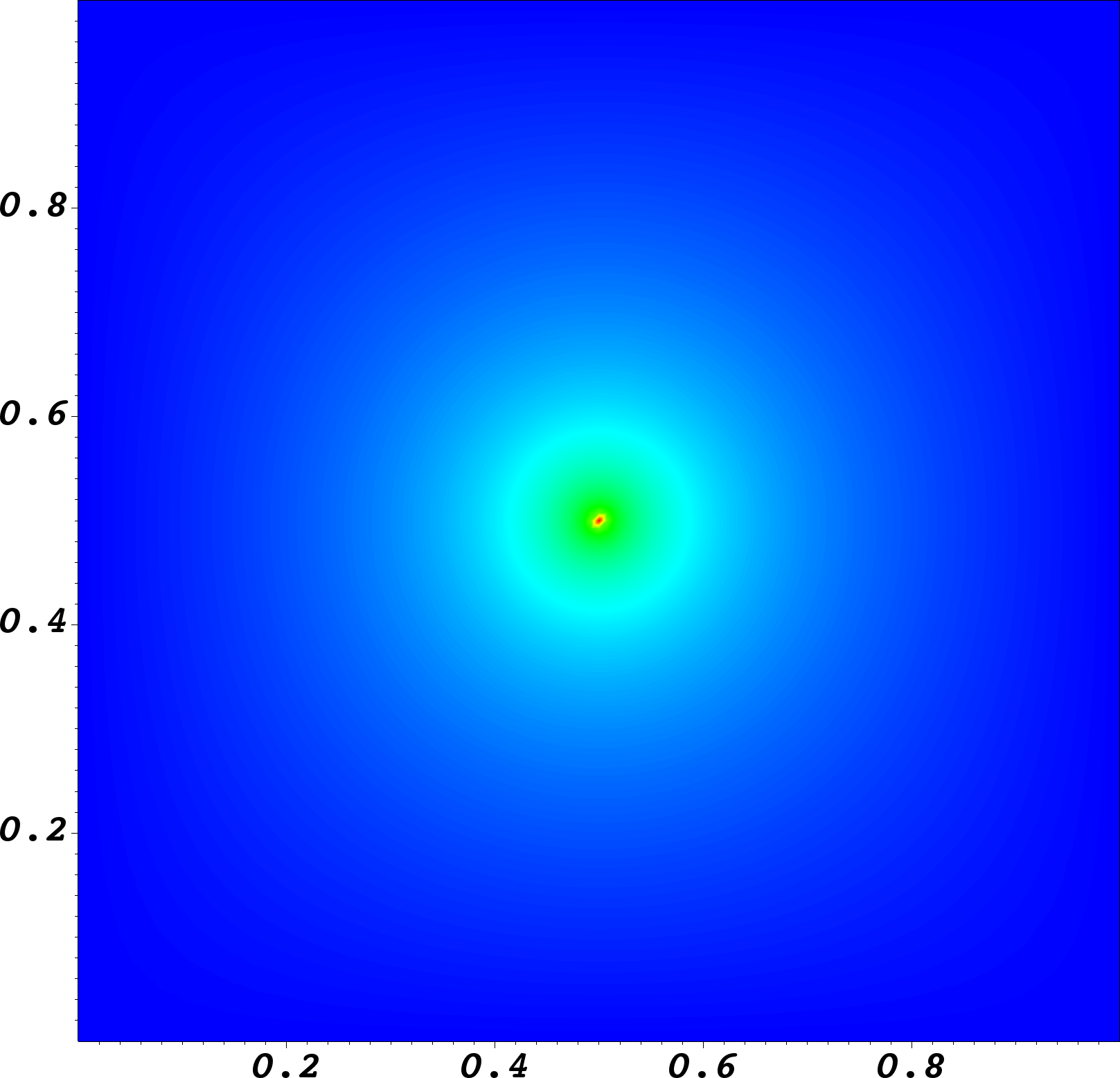}
   \caption{Example 1: Going from left to right and top to bottom: error indicators, adaptive mesh, primal solution and adjoint solution.}
\label{fig_DWR}
\end{center}
\end{figure}

\subsection{Non-linear Elliptic Boundary Value Problems}
\label{Subsec: Example2}
In the second example, we consider the homogeneous Dirichlet boundary value problem 
for a non-linear elliptic partial differential equation in the plane domain $\Omega=(0,5)\times (0,3) \setminus \left((1,2)^2 \cup (3,4)\times (1,2) \right) \subset \mathbb{R}^2$ visualized 
together with the initial mesh and the quantities of interest in Figure~\ref{fig: Ex2: Init Mesh+ Mesh25}.

\subsubsection{Strong and Weak Problem Formulations}
We look for some function $u$ such that
\begin{equation}
\label{eq:Ex2:pde}
-\text{div}(\nu(|\nabla u|)\nabla u)=f\; \text{in}\;\Omega
 \quad \text{and} \quad u=0 \; \text{on}\; \partial\Omega,
\end{equation}
with 
$\nu(s)=2+\text{arctan}(s^2)$ 
and given right-hand side $f=10$.
Such kind of non-linear partial differential equations arise, for instance, 
in magneto-static where $\nu$ is the reluctivity that non-linearly depends on the gradient of $u$;
see, e.g., \cite{ELRWS:Heise:1994SINUM}.
The weak form of \eqref{eq:Ex2:pde} reads as follows: Find 
$u \in U = V = H_0^1(\Omega) =\mathring{W}_2^1(\Omega)$ 
such that
\begin{equation}
\label{eq:Ex2:WeakForm}
\mathcal{A}(u)(v) = 0 \quad \forall \,v \in V
\end{equation}
that can equivalently be written as non-linear operator equation
\begin{equation}
\label{eq:Ex2:OperatorEquation}
\mathcal{A}(u) = 0 \quad \text{in} \; V^* = U^* = H^{-1}(\Omega),
\end{equation}
where the non-linear operator $\mathcal{A}(u): H_0^1(\Omega) \rightarrow H^{-1}(\Omega)$
is well defined by the variational identity
\begin{equation}
\label{eq:Ex2:OperatorDefinition}
\mathcal{A}(u)(v) := \langle \nu(|\nabla u|)\nabla u,\nabla v \rangle-\langle f,v \rangle 
\quad \forall\, v, u \in H_0^1(\Omega).
\end{equation}
\subsubsection{Mathematical Properties and Fr\'echet Derivative}

Since the non-linear operator is strongly monotone and Lipschitz continuous, 
the operator equation \eqref{eq:Ex2:OperatorEquation} and the equivalent 
weak or variational formulation \eqref{eq:Ex2:WeakForm} have a unique solution 
$u \in H_0^1(\Omega)$;
see, e.g., \cite{Zeidler:1990a,ELRWS:Heise:1994SINUM}.
The differentiability properties of $\nu$ yield the corresponding 
Fr\'{e}chet differentiability properties of the non-linear operator $\mathcal{A}$.
For instance, the first Fr\'{e}chet derivative  of $\mathcal{A}$ at some $u \in U$
is nothing but the bounded, linear operator $\mathcal{A}'(u): U \rightarrow V^*$ 
defined by the variational identity
\begin{equation}
\label{eq:Ex2:FrechetDerivative}
\mathcal{A}'(u)(v,w) := 
\langle (\nu(|\nabla u|) I + \frac{\nu'(|\nabla u|)}{|\nabla u|}\nabla u (\nabla u)^T) \nabla v,\nabla w \rangle
\quad \forall\, v, w \in H_0^1(\Omega).
\end{equation}

\subsubsection{{Goal Functionals and Combined Functional}}
In the following, we are interested in the following quantities of interest:
\begin{itemize}
\item the flux on $\Gamma_{\text{Flux}}:=\{0\}\times (0,3)$: 
	$J_1(u):=\int_{\Gamma_{\text{Flux}}} \nabla u \cdot n \text{ d}s_x,$
\item the point evaluation at $x_0=(0.2,0.2)$: $J_2(u):=u(x_0),$
\item the point evaluation at $x_1=(0.9,0.1)$: $J_3(u):=u(x_1).$
\end{itemize}
We can assume that the functionals are well defined at the solution $u$.
As an error weighting function, we choose the function presented in \eqref{eq:error weighting function:example 1}, i.e.
	\begin{equation*}
	\mathfrak{E}(x,m)=\sum_{i=1}^N \frac{x_i}{|m_i|},
	\end{equation*}
	with $m=(J_1(u_h),J_2(u_h),J_3(u_h))$.
	Then, the combined functional follows from \eqref{eq: TrueErrorrepresentationFunctional}.

\subsubsection{Finite Element Discretization}
The basic finite element discretization is performed by $Q_1$ finite elements, whereas $Q_2$
finite elements are used to construct the enriched spaces. For instance, the finite element 
scheme for solving the primal problem \eqref{eq:Ex2:WeakForm} reads as follows: 
Find $u_h \in U_h \subset U$ such that
\begin{equation}
\label{eq:Ex2:FinitElementScheme}
\mathcal{A}(u_h)(v_h) = 0 \quad \forall \,v_h \in V_h=U_h,
\end{equation}
where $U_h$ is spanned by the $Q_1$ basis functions corresponding to the internal nodes.

\subsubsection{Discussion and Interpretation of Our Findings}
Figure~\ref{fig: Ex2: Init Mesh+ Mesh25} shows the initial mesh with three quantities of interest, 
and the adaptively refined mesh after 25 refinement steps 
when the adaptivity is driven by the combined functional $J_\mathfrak{E}$. 
As expected, we observe heavy mesh refinement where the three functionals are computed.
The 8 interior 
corners  
with a re-entrant angle of {$\frac{3}{2}\pi$} give rise to singularities that pollute the solution 
elsewhere.
However, the mesh refinement around 
these singularities depends on the distance to the places where we compute the functionals.
The strongest refinement is observed in the lower left re-entrant corner, 
whereas the initial mesh is almost not refined in the upper right re-entrant corner that has the largest distance from the place where we evaluate the functionals. 
Figure~\ref{fig: Ex2: Solution+Mesh} shows the finite element solution $u_h$ and the localized error estimator after 25 adaptive refinements of the initial mesh. We observe that $u_h$ is recovered more accurately in regions 
where the accuracy is needed for computing accurate values of the functionals.
This goal-oriented adaptive refinement aiming at the joint accurate computation 
of local functionals is very different from the adaptive mesh refinement driven 
by global values like the 
$H^1$-norm of the discretization error in the residual a posteriori error estimator; 
see, e.g. \cite{Verfuerth:1996a}.
Indeed, in Figure~\ref{fig: Ex2: Solution+MeshL2},
we display the finite element solution $u_h$ and the mesh after 25 adaptive refinements of the initial mesh 
when the adaptive refinement is driven by the $L_2$-norm, i.e., by the functional 
$J(u) = \|u\|^2_{L_2(\Omega)}$ that is a global functional too.
As expected, the refinement is located at the 8 singularities in the 8 re-entrant corners.
In addition to this, Figures~\ref{fig: Ex2: IeffL2} and \ref{fig: Ex2: errorL2} 
illustrates the numerical behavior of the effectivity indices $I_{\mathrm{eff}}$, $I_{\mathrm{eff,a}}$, and $I_{\mathrm{eff,p}}$ for 
$J(u)$
and the decay of the absolute error for the functional $J$ in
the cases of adaptive and uniform refinements, respectively. 
This error decay behaves like 
$O(h^{3/2}) = O(\text{DoFs}^{-3/4})$ for uniform mesh refinement due to the singularities,
whereas $O(\text{DoFs}^{-1})$, 
that is equivalent to $O(h^{2})$ on a uniform mesh,
is observed for the adaptive finite element procedure driven by the functional $J$.

Furthermore, Figure~\ref{fig: Ex2: Ieff} shows the numerical behavior of
the effectivity indices $I_{\mathrm{eff}}$, $I_{\mathrm{eff,a}}$, and $I_{\mathrm{eff,p}}$ 
for the combined functional $J_\mathfrak{E}$.
We see that all three effectivity indices are close to $1$, and $I_{\mathrm{eff}}$ 
is practically almost $1$ for more than $10000$ DoFs.
Figure~\ref{fig: Ex2: errors} provides the error decay of the relative errors 
for $J_1$, $J_2$, $J_3$, and the absolute error for $J_\mathfrak{E}$.
Figures~\ref{fig: Ex2: Error: J1}, \ref{fig: Ex2: Error: J2}, and  \ref{fig: Ex2: Error: J3} 
compare the error decay of the uniform and adaptive refinements 
for the functionals $J_1$, $J_2$, and $J_3$, respectively,
whereas Figure~\ref{fig: Ex2: Error: J?} shows that the error estimator $\eta_h$ is
practically identical with the error for $J_\mathfrak{E}$, and both decay like 
$O((\text{DoFs})^{-1})$ as expected in the adaptive case.

\begin{figure}[H]
	\begin{tikzpicture}
	\definecolor{xdxdff}{rgb}{0.49019607843137253,0.49019607843137253,1}
	\definecolor{xdxdffxkcd}{rgb}{0.8,0.3,0.5}
	\node[anchor=south west,inner sep=0] (image) at (0,0) {\includegraphics[width=0.47\linewidth]{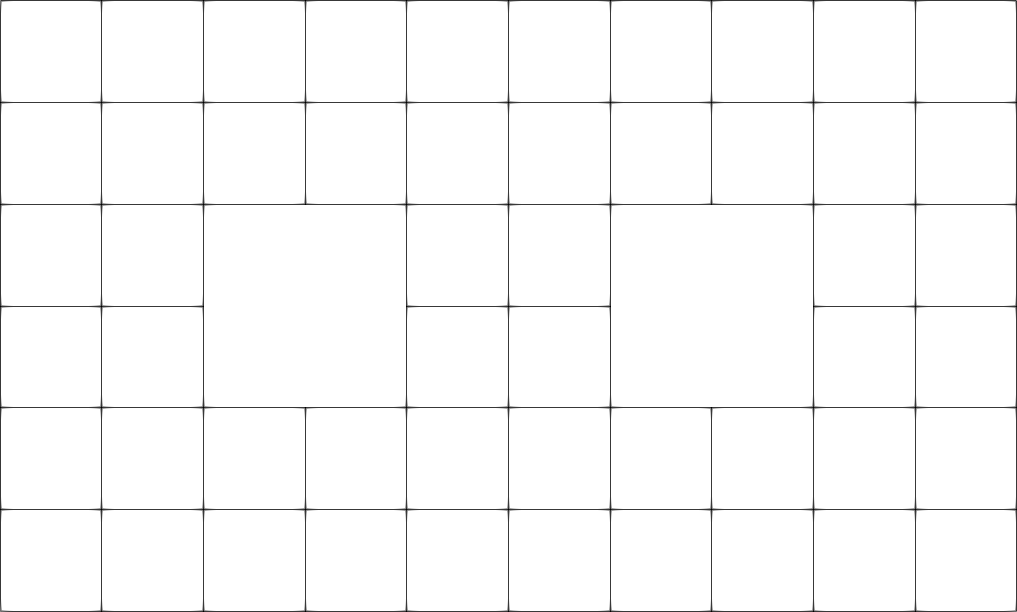}};
	\begin{scope}[x={(image.south east)},y={(image.north west)}]
	\draw [line width=4pt, color=xdxdff,opacity=0.4] (0,0)-- (0,1);
	\draw[color=xdxdff] (0.05,0.55) node {$J_1$};
	\draw [fill=xdxdffxkcd,opacity=0.4] (0.04,0.066666666666666666666666) circle (3.pt);
	\draw[color=xdxdffxkcd,opacity=0.8] (0.078,0.045) node {$J_2$};
	\draw [fill=orange,opacity=0.4] (0.18,0.0333333333333333333333) circle (3.pt);
	\draw[color=orange,opacity=1.0] (0.175,0.11) node {$J_3$};
	\end{scope}
	\end{tikzpicture}\hfill
	\includegraphics[width=0.47\linewidth]{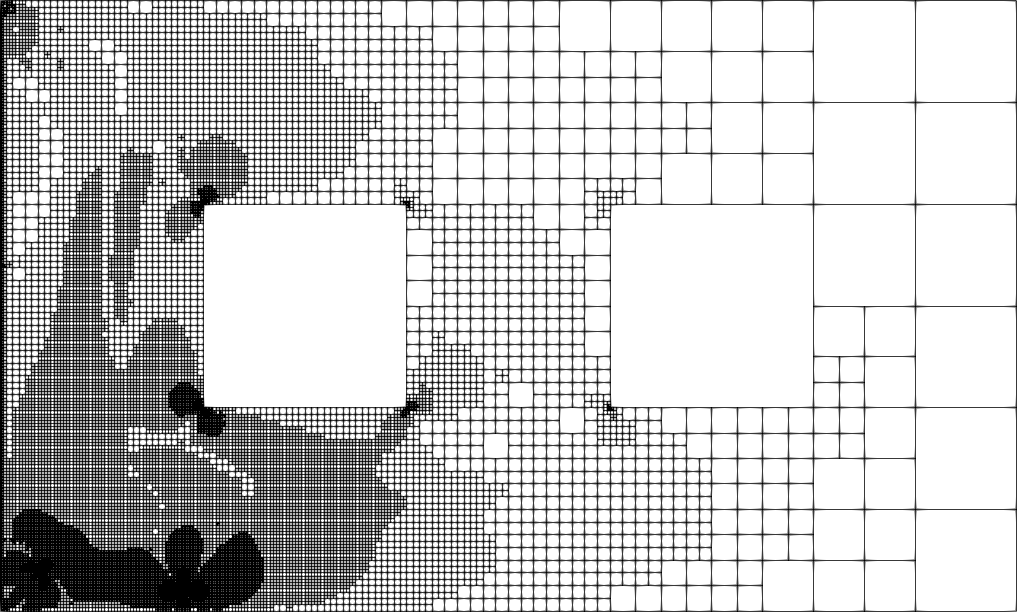}
	\caption{Example 2:  Initial mesh with quantities of interest (left), 
	                 and adaptively refined mesh after 25 refinement steps  {driven by the combined functional $J_\mathfrak{E}$} (right).}  
	\label{fig: Ex2: Init Mesh+ Mesh25}
\end{figure}
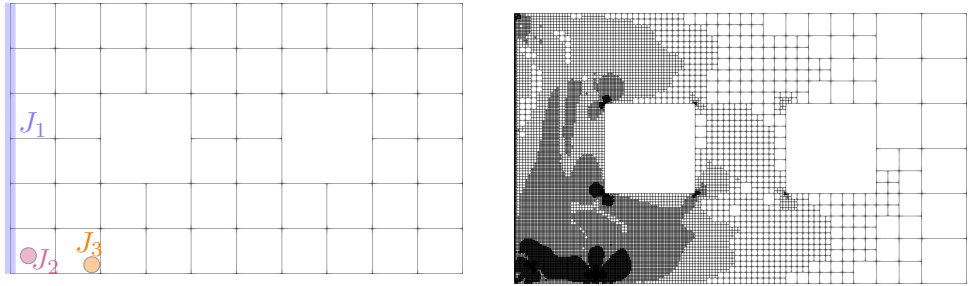

\begin{figure}[H]
	\includegraphics[width=0.47\linewidth]{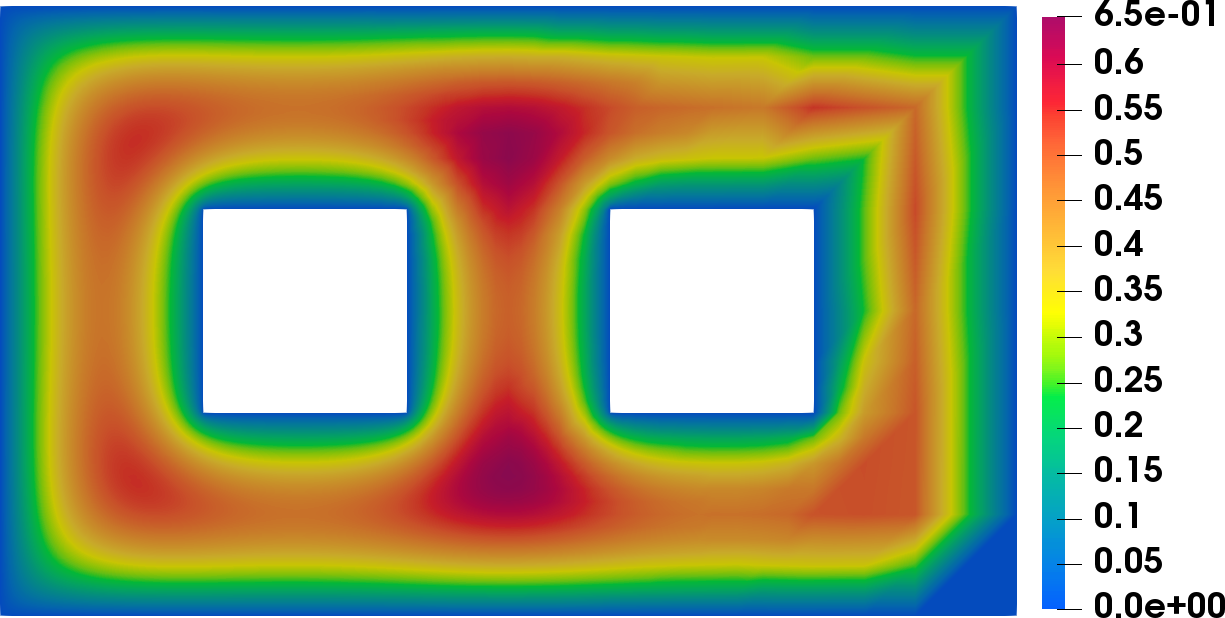} \hfill
	\includegraphics[width=0.47\linewidth]{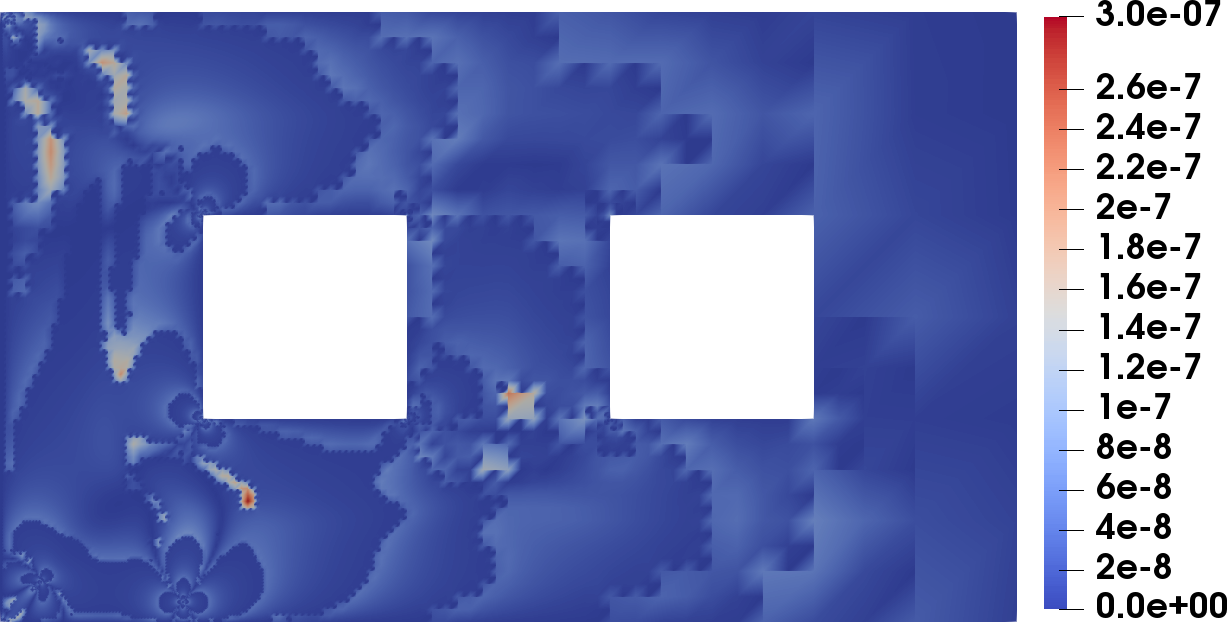}
	\caption{Example 2:  The approximate solution $u_h$ (left), and the localized error estimator after 25 refinement steps (right)  {driven by the combined functional $J_\mathfrak{E}$}.\label{fig: Ex2: Solution+Mesh}}   
	
\end{figure}

\begin{figure}[H]


	{	\begin{minipage}{0.47\linewidth}
			\scalebox{0.60}{\input{Figures/IeffEx12.tex}}
			\caption{Example 2: 
			Effectivity indices $I_{\mathrm{eff}}$, $I_{\mathrm{eff,a}}$, and $I_{\mathrm{eff,p}}$ 
			for $J_\mathfrak{E}$.
			\label{fig: Ex2: Ieff}}
		\end{minipage}\hfill
		\begin{minipage}{0.47\linewidth}
			\scalebox{0.60}{\input{Figures/ErrorEx12.tex}}
			\caption{Example 2: Relative errors for $J_1$, $J_2$, $J_3$, and absolute error for $J_\mathfrak{E}$. \label{fig: Ex2: errors}}
		\end{minipage}	
	}	
\end{figure}

\begin{figure}[H]


	{	\begin{minipage}{0.47\linewidth}
			\scalebox{0.60}{\input{Figures/ErrorEx1J12.tex}}
			\caption{Example 2: Relative errors for adaptive and uniform refinement for $J_1$. \label{fig: Ex2: Error: J1}}
		\end{minipage}\hfill
		\begin{minipage}{0.47\linewidth}
			\scalebox{0.60}{\input{Figures/ErrorEx1J22.tex}}
			\caption{Example 2: Relative errors for adaptive and uniform refinement for $J_2$. \label{fig: Ex2: Error: J2}}
		\end{minipage}	
	}		
\end{figure}

\begin{figure}[H]


	{	\begin{minipage}{0.47\linewidth}
			\scalebox{0.60}{\input{Figures/ErrorEx1J32.tex}}
			\caption{Example 2: Relative errors for adaptive and uniform refinement for $J_3$. \label{fig: Ex2: Error: J3}}
		\end{minipage}\hfill
		\begin{minipage}{0.47\linewidth}
			\scalebox{0.60}{\input{Figures/ErrorEx1Estimator12.tex}}
			\caption{Example 2: Error for $J_\mathfrak{E}$ and error estimator $\eta_h$. \label{fig: Ex2: Error: J?}}
		\end{minipage}	
	}	
	
\end{figure}

	\begin{figure}[H]
		\includegraphics[width=0.52\linewidth]{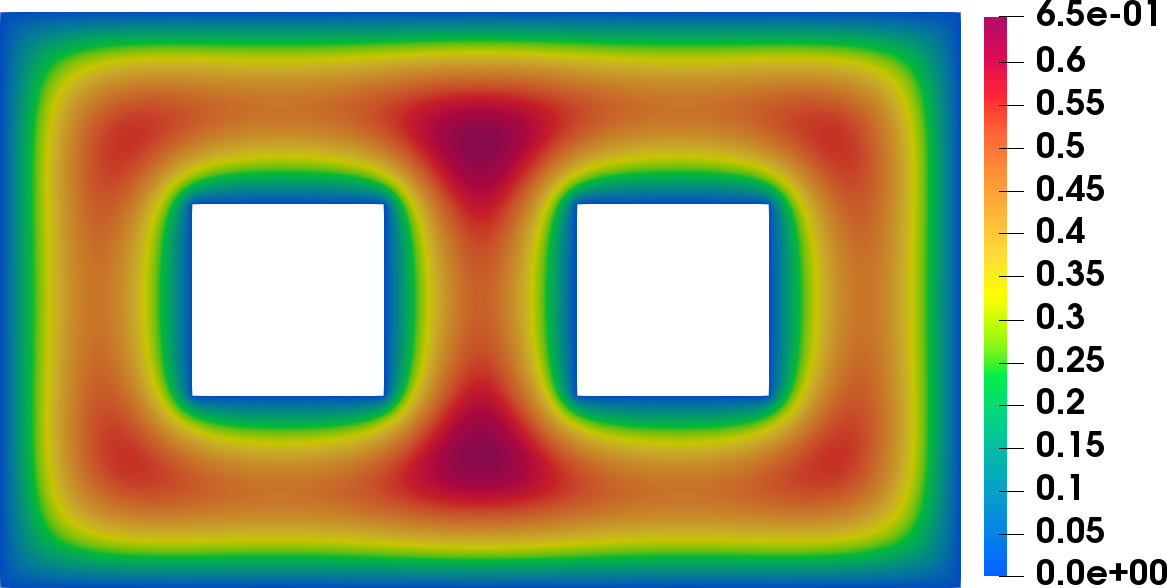} \hfill
		\includegraphics[width=0.43\linewidth]{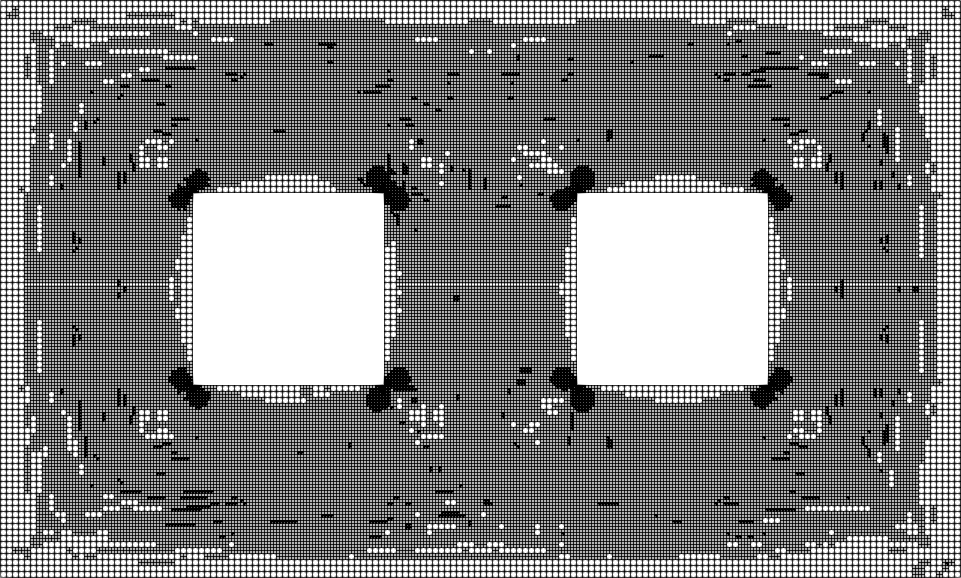}
		\caption{Example 2:  The approximate solution $u_h$ (left), 
		and adaptively refined mesh after {25} refinement steps (right)
		for controlling the $L_2$-norm by means of the functional $J(u) = \|u\|^2_{L_2(\Omega)}$.
		\label{fig: Ex2: Solution+MeshL2}}  		
	\end{figure}

\begin{figure}[H]

					{	\begin{minipage}{0.47\linewidth}
							\scalebox{0.60}{\input{Figures/IeffEx1L2Norm2.tex}}
							\caption{Example 2: 
							Effectivity indices $I_{\mathrm{eff}}$, $I_{\mathrm{eff,a}}$, and $I_{\mathrm{eff,p}}$ for $J$.
							\label{fig: Ex2: IeffL2}}
						\end{minipage}\hfill
						\begin{minipage}{0.47\linewidth}
							\scalebox{0.60}{\input{Figures/ErrorEx1L2Norm2.tex}}
							\caption{Example 2: Absolute  error for $J$
							in the cases of adaptive and uniform refinements.
							\label{fig: Ex2: errorL2}}
						\end{minipage}	
					}	
				\end{figure}

\newpage
\subsection{Stationary Incompressible Navier-Stokes Problem}
\label{Subsec: Example3}
In this third example, we consider the flow around a 
cylinder benchmark from \cite{TurSchabenchmark1996}. 
We notice that this example is an extension of our prior study \cite{EnLaWi20}.

\subsubsection{Domain}
The domain $\Omega:=(0,L)\times (0,H)\setminus \mathcal{B}$ where $L=2.2$, $H=0.41$ and $\mathcal{B}:=\{x \in \mathbb{R}:|x-(0.2,0.2)|<0.05\}$. The domain as well as the boundary conditions are depicted in Figure~\ref{fig: Ex3: Omega+Boundary}.

\begin{figure}[H]
	\begin{tikzpicture}[line cap=round,line join=round,>=triangle 45,x=6cm,y=6cm]
	\clip(-0.4,-0.2) rectangle (2.6,0.6);
	\draw [line width=2pt,color=cyan] (2.205,0.41)-- (2.2,0);
	\draw [line width=2pt,color=black] (0,0.41)-- (2.2,0.41);	
	\draw [line width=2pt,color=black] (2.2,0)-- (0,0);
	\draw [line width=2pt,color=orange] (0,0)-- (0,0.41);
	\draw [line width=2pt,color=blue] (0.2,0.2) circle (0.05);
	\begin{scriptsize}
	\draw[color=orange] (-0.15,0.2) node {\Large $\Gamma_{\mathrm{inflow}}$};
	\draw[color=black] (1.1,0.46) node {\Large$\Gamma_{0}$};
	\draw[color=cyan] (2.35,0.2) node {\Large$\Gamma_{\mathrm{outflow}}$};
	\draw[color=black] (1.1,-0.06) node {\Large$\Gamma_{0}$};
	\draw[color=blue] (0.2,0.3) node {\Large$\partial \mathcal{B}$};
	\draw[color=magenta] (0.3,0.2) node {\Large$x_2$};
	\draw[color=magenta] (0.1,0.2) node {\Large$x_1$};
	\draw[color=magenta] (2.13,0.205) node {\Large$x_3$};
	\draw [fill=magenta] (0.15,0.2) circle (2.5pt);
	\draw [fill=magenta] (0.25,0.2) circle (2.5pt);
	\draw [fill=magenta] (2.2,0.205) circle (2.5pt);
	\end{scriptsize}
	\end{tikzpicture}	
	\caption{Section \ref{Subsec: Example3}: Flow around a cylinder. The domain $\Omega$ with boundary parts and visualization of $x_1$, $x_2$ and $x_3$. \label{fig: Ex3: Omega+Boundary}}
\end{figure}
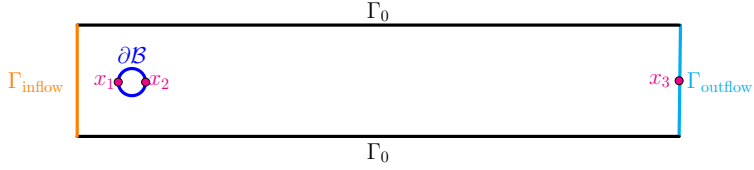

\subsubsection{Equations in Strong Form and Boundary Conditions}
Find $u:=(v,p)$ such that
\begin{align*}
 - \nabla\cdot\left (\nu \nabla v\right)+(v \cdot \nabla) v + \nabla p  =& 0   \qquad\text{in }\Omega,\nonumber\\ 
\nabla \cdot v =&  0   \qquad\text{in }\Omega,\nonumber\\
\end{align*}
with $\nu =10^{-3}$.
Furthermore, we consider the following boundary conditions	
\begin{equation}
\begin{aligned}
v=&0   \qquad\text{on } \Gamma_{\mathrm{no-slip}}:=\partial \mathcal{B} \cup \Gamma_{0}, \\
v=&\hat{v}  \qquad \text{on } \Gamma_{\mathrm{inflow}}, \nonumber\\
\nu \frac{\partial v}{\partial \vec{n}} - p \cdot \vec{n }=& 0  
\qquad\text{on } \Gamma_{\mathrm{outflow}}, \nonumber\\
\end{aligned}
\end{equation}
where $\Gamma_{\mathrm{inflow}}:=\{0\}\times (0,H)$, $\Gamma_{\mathrm{outflow}}:=\{L\}\times (0,H)$, and $\Gamma_{0}:=(0,L)\times\{0,H\}$.
The inflow profile $\hat{v}$ is $\hat{v}(x_1,x_2):=1.2x_2(H-x_2)/H^2$.
Additionally, we mention that the pressure is uniquely determined 
due to the do-nothing condition on $\Gamma_{\mathrm{outflow}}$; see \cite{HeRaTu96}.

\subsubsection{Weak Formulation}
The weak form of the problem is given by: Find $u:=(v,p) \in (\hat{v},0) + U:=H^1_{\Gamma_\mathrm{no-slip}\cup\Gamma_{\mathrm{inflow}}}(\Omega) \times L^2(\Omega)$ such that
\begin{equation}
	\mathcal{A}(u)(\psi):=(\nu \nabla v, \nabla \psi_v)+(\nabla v \cdot v, \psi_v)+(p, \operatorname{div}(\psi_v))+(\operatorname{div}(u),\psi_p)
\end{equation}
for all test functions $(\psi_v,\psi_p)=:\psi \in V=U$,
where $H^1_{\Gamma_\mathrm{no-slip}\cup\Gamma_{\mathrm{inflow}}}(\Omega):=\{v \in H^1(\Omega): u =0 \text{ on }\Gamma_\mathrm{no-slip}\cup\Gamma_{\mathrm{inflow}} \}$.

\subsubsection{Finite Element Discretization}
The domain is decomposed into quadrilateral elements; cf. Figure~\ref{fig: Ex3: Mesh0+9} (left). As discretization, we use $Q_2^d$ finite elements for the velocity $v$ and $Q_1$ finite elements for the pressure $p$. For the enriched space we use $Q_4^d$ finite elements for the velocity $v$ and $Q_2$ finite elements for the pressure $p$. This discretizations are inf-sup stable as shown in \cite{zulehner2022short}.

\subsubsection{Goal Functionals and Combined Functional}
We are interested in the following functionals of interest:
\begin{itemize}
	\item the pressure difference between $x_1=(0.15,0.2)$ and $x_2=(0.25,0.2)$: $$J_1(u):=p(x_1)-p(x_2),$$
	\item the drag at the ball $\mathcal{B}$: 
	$$J_2({u}):= 500\int_{\partial \mathcal{B}} \left[\nu \frac{\partial u}{\partial \vec{n}} - p \vec{n }\right]\cdot \vec{e}_1\,\mathrm{ d}s_{(x)},$$
	\item the lift at the ball $\mathcal{B}$: 
	$$J_3({u}):= 500\int_{\partial \mathcal{B}} \left[\nu \frac{\partial u}{\partial \vec{n}} - p \vec{n }\right]\cdot \vec{e}_2\,\mathrm{ d}s_{(x)},$$
	\item the square of the magnitude of velocity at $x_3=(2.2,0.205)$: $$J_4(u):=|v|^2(x_3),$$
\end{itemize}
where $\vec{e}_1 = (1,0)$ and $\vec{e}_2 = (0,1)$.
For $J_1$, $J_2$, $J_3$, we use the reference values from \cite{nabh1998high}, and for $J_4$ we used uniform $p$-refinement on an adaptively refined grid.
The reference values are :
\begin{itemize}
	\item $J_1(u)=0.11752016697$,
	\item $J_2(u)=5.57953523384$,
	\item $J_3(u)=0.010618948146$,
	\item $J_4(u)=0.088364291405$.
\end{itemize}
As in the previous example, we choose the error weighting function \eqref{eq:error weighting function:example 1}, i.e.
	\begin{equation*}
	\mathfrak{E}(x,m)=\sum_{i=1}^N \frac{x_i}{|m_i|},
	\end{equation*}
	with $m=(J_1(u_h),J_2(u_h),J_3(u_h),J_4(u_h))$.

\subsubsection{Discussion and Interpretation of Our Findings}
In Figure~\ref{fig: Ex3: errors}, we can see that the error functional $J_\mathfrak{E}$ dominates the relative error in the single functionals. At the beginning of the refinement procedure, the error in $J_4$ is the smallest of all the functionals. However, 
it requires different local refinement needs than the other functionals. Therefore, the error in $J_4$ is not reduced in the first $10$ refinement steps. In fact, the elements around the point $x_3$ are not refined a single time during the first $9$ refinement steps, cf. Figure~\ref{fig: Ex3: Mesh0+9}. In Figure~\ref{fig: Ex3: Mesh11+17}, there is already refinement in the area around $x_3$ but not at $x_3$ after $11$ adaptive refinements, but a lot at $x_3$ after $17$ adaptive refinements. We also observe that the error of $J_4$ is reduced in those refinement steps, cf. Figure~\ref{fig: Ex3: errors} and Figure~\ref{fig: Ex3: Error: J4}. The effectivity index  $I_\mathrm{eff}$ in Figure~\ref{fig: Ex3: Ieff} is between 0.5 and 2.5. The adjoint and primal effectivity indices have a similar value. 
In Figures~\ref{fig: Ex3: Error J1}~-~\ref{fig: Ex3: Error: J4}, we observe that adaptive refinement has a smaller error than uniform refinement after $30\, 000$ DoFs. Additionally, the convergence rate is better for adaptive refinement for all $4$ functionals. 
At $100\, 000$ DoFs, all errors are approximately one order of magnitude lower for adaptive refinement than for uniform refinement.
\begin{figure}[H]
\includegraphics[width=0.47\linewidth]{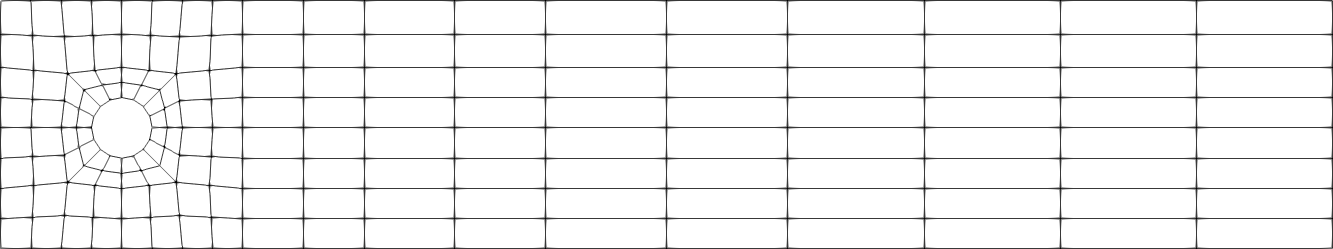} \hfill
\includegraphics[width=0.47\linewidth]{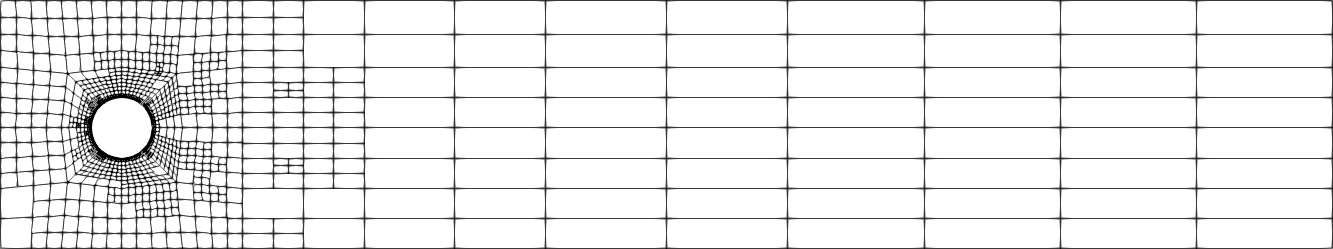}
\caption{Example 3: Initial mesh (left) and adaptively refined mesh after 9 refinements (right). \label{fig: Ex3: Mesh0+9}}
\end{figure}

\begin{figure}[H]
	\includegraphics[width=0.47\linewidth]{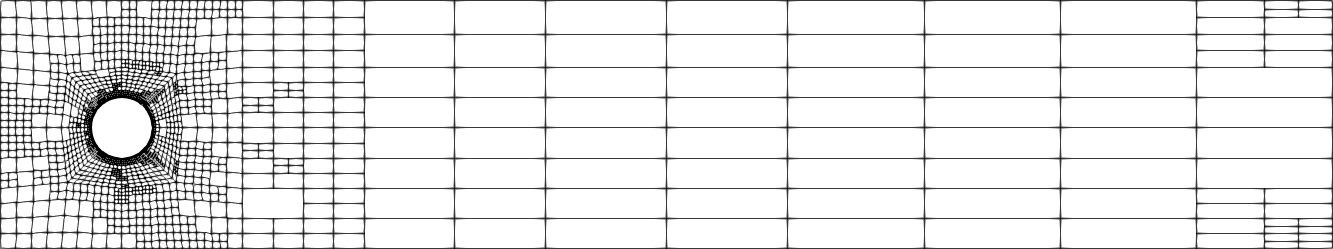} \hfill
	\includegraphics[width=0.47\linewidth]{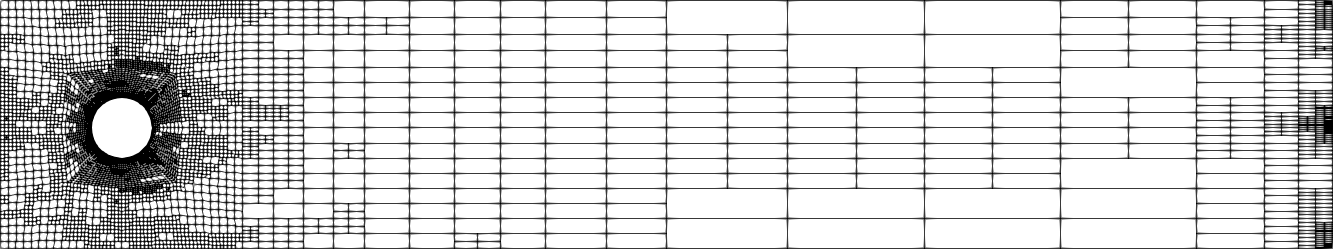}
	\caption{Example 3: Adaptively refined mesh after 11 refinements (left) and 17 refinements (right). \label{fig: Ex3: Mesh11+17}}
\end{figure}

\begin{figure}[H]


	{	\begin{minipage}{0.47\linewidth}
			\scalebox{0.60}{\input{Figures/IeffEx22.tex}} 
			\caption{Example 3: Effectivity index for $J_\mathfrak{E}$. \label{fig: Ex3: Ieff}}
		\end{minipage}\hfill
		\begin{minipage}{0.47\linewidth}
			\scalebox{0.60}{\input{Figures/ErrorEx22.tex}}
			\caption{Example 3: Relative errors for $J_1$, $J_2$, $J_3$, $J_4$ and absolute error for $J_\mathfrak{E}$. \label{fig: Ex3: errors}}
		\end{minipage}	
	}	
\end{figure}

\begin{figure}[H]


	{	\begin{minipage}{0.47\linewidth}
			\scalebox{0.60}{\input{Figures/ErrorEx2J12.tex}}
			\caption{Example 3: Relative errors for adaptive and uniform refinement for $J_1$. \label{fig: Ex3: Error J1}}
		\end{minipage}\hfill
		\begin{minipage}{0.47\linewidth}
			\scalebox{0.60}{\input{Figures/ErrorEx2J22.tex}}
			\caption{Example 3: Relative errors for adaptive and uniform refinement for $J_2$. \label{fig: Ex3: Error: J2}}
		\end{minipage}	
	}	
	
\end{figure}

\begin{figure}[H]


	{	\begin{minipage}{0.47\linewidth}
			\scalebox{0.60}{\input{Figures/ErrorEx2J32.tex}}
			\caption{Example 3: Relative errors for adaptive and uniform refinement for $J_3$. \label{fig: Ex3: Error J3}}
		\end{minipage}\hfill
		\begin{minipage}{0.47\linewidth}
			\scalebox{0.60}{\input{Figures/ErrorEx2J42.tex}}
			\caption{Example 3: Relative errors for adaptive and uniform refinement for $J_4$. \label{fig: Ex3: Error: J4}}
		\end{minipage}	
	}	
	
\end{figure}

\subsection{Parabolic p-Laplace Initial-Boundary Value Problems}\label{Subsec: Example4}
For our final application, we consider a non-linear evolutionary PDE, in particular, 
the regularized parabolic $p$-Laplace PDE.
Let $\Omega \subset \mathbb{R}^d$ again denote our spatial domain, and let $T > 0$ be the final time horizon.
To illustrate the performance of the space-time multi-goal oriented finite element method, we perform numerical experiments for $d=2$ and $d=3$. Here we consider problems with spatial singularities, but that are smooth in time. Depending on the goal functionals, we expect refinement not only in the domains of interest, but also towards the spatial singularities. In particular, we consider the classical L-shaped domain for the 2D+1 ($d=2$) case, 
and the Fichera corner domain for the 3D+1 ($d=3$) case. 

\subsubsection{Strong and Weak Problem Statements}
The problem reads: Find $u$ such that
\begin{gather}\label{eq:ex4:pde}
	\partial_t u - \operatorname{div}_x( (|\nabla_x u|^2 + \epsilon^2)^{(p-2)/2} \nabla_x u) = f\text{ in } Q = \Omega \times (0,T),\\
	{u = 0 \text{ on } \Sigma_0=\Omega\times\{0\} \quad\text{and}\quad u = 0\text{ on } 
	\Sigma=\partial \Omega\times(0,T).}
	\label{eq:ex4:pde:ibc}
\end{gather}
Note that this form of the regularized $p$-Laplacian is motivated by the Carreau model. In the numerical experiment, we choose $p=4$ and $\epsilon = 10^{-2}$. 
We can formulate the above problem again in the abstract framework introduced in Section~\ref{sec_notation_abstract_setting}. Indeed, with the choices $ V = L_p(0,T; \mathring{W}^{1}_{p}(\Omega))$ and 
$U = \{ v \in V: \partial_t v \in V^*,  v = 0\ \text{on}\ \Sigma_0\} $,
we 
arrive at the operator equation: Find $u \in U$ such that
\begin{equation}
\label{eq:ex4:OperatorEquation}
	\mathcal{A}(u) = 0\quad \mbox{in}\; V^*,
\end{equation}
where 
\begin{equation*}
	\langle \mathcal{A}(u), v \rangle = \langle\partial_t u, v\rangle + \langle (|\nabla_x u|^2 + \epsilon^2)^{(p-2)/2} \nabla_x u, \nabla_x v \rangle  -  \langle f, v\rangle \quad \mbox{for all}\, v\in V,
\end{equation*}
$ L_p(0,T; X) $ denotes the Bochner space of $L_p$-functions that map from the interval $(0,T)$ to $X$, and $\mathring{W}^{1}_{p}(\Omega) = \{v \in L_p(\Omega): 
\nabla v  \in (L_p(\Omega))^d, v = 0 \ \text{on}\ \partial \Omega \}$.
The existence and uniqueness of the solution of the
operator equation \eqref{eq:ex4:OperatorEquation} follow from standard monotonicity arguments;
see, e.g.,\ \cite{Zeidler:1990a,Roubicek:2013a}.

\subsubsection{Space-Time Finite Element Discretization and Non-Linear Solver}
Evolutionary problems like \eqref{eq:ex4:pde}--\eqref{eq:ex4:pde:ibc} are usually solved by the method of lines, i.e.\ we perform semi-discretization separately for the spatial domain $\Omega$ and the time interval $(0,T)$. An alternative ansatz is an all-at-once discretization, i.e.\ we directly discretize the space-time cylinder $Q\subset \mathbb{R}^{d+1}$. In this work, we will focus on the latter approach. We will briefly describe the key points, for more details on space-time methods see, e.g.,~ 
\cite{LaStein19}. 
Starting with a space-time Galerkin ansatz with finite dimensional subspaces  
$ U_h \subset U $ and $ V_h = U_h \subset V $,
we arrive at the discretized problem: Find $u_h \in U_h $ such that
\begin{equation*}
	\langle \mathcal{A}(u_h), v_h \rangle = 0 \quad \text{for all}\ v_h \in V_h.
\end{equation*}
The above discrete problem has a unique solution; see \cite{endtmayer2023goal} and the references therein. Finally, in order to solve the non-linear systems of equations, we again apply Newton's method, where in each iteration, we have to solve the linear problem: Find $ w_h \in U_h $ such that
\begin{equation*}
	\langle \mathcal{A}'(u_h)w_h, v_h \rangle = -\langle \mathcal{A}(u_h), v_h \rangle \quad \text{for all}\ v_h \in V_h.
\end{equation*}
The above equation is uniquely solvable. Hence, Newton's method is well defined; see \cite{endtmayer2023goal} and the references therein.

We apply Algorithm~\ref{Alg: adapt. multigoal Algorithm}, i.e.\ the enriched primal and adjoint problems are solved by Newton's method and preconditioned GMRES, respectively.

We implemented the space-time finite element method using the finite element library MFEM \cite{mfem:2021}. We use a $P_1$ conforming ansatz space for the 2D+1 ($d=2$) and 3D+1 ($d=3$) case, i.e.\ we use a simplicial decomposition $ \mathcal{T}_h $ of the $(d+1)$-dimensional space-time cylinder $Q \subset \mathbb{R}^{d+1}$.

\subsubsection{Goal Functionals and Combined Goal Functional}
We consider two goal functionals:
\begin{itemize}
	\item Let $\Omega_I := (-15/16,-11/16)\times (11/16,15/16)$ and 		
		  $\Omega_I := (11/16,15/16)^3$ for $d=2$ and $d=3$, respectively, and
	\begin{equation*}
		J_1(u) := \int_{0.5}^{0.75}\!\int_{\Omega_I}\! |\nabla u(x,t)|^4 \;\mathrm{d}x\mathrm{d}t,
	\end{equation*}
	i.e.\ we are interested in the energy in some subdomain far away from the corner singularity, and
	\item the ``average'' of the solution at final time $T$, i.e.
	\begin{equation*}
		J_2(u) = \int_\Omega\! u(x,T) \;\mathrm{d}x.
	\end{equation*}
\end{itemize}
As an error weighting function, we choose the function presented in \eqref{eq:error weighting function:example 1}. 
Thus, our combined functional has the form
\begin{equation*}
	J_{\mathfrak{E}}(v) = \frac{|J_1(u_h^{(2)}) - J_1(v)|}{J_1(u_h)} + \frac{|J_2(u_h^{(2)}) - J_2(v)|}{J_2(u_h)}.
\end{equation*}
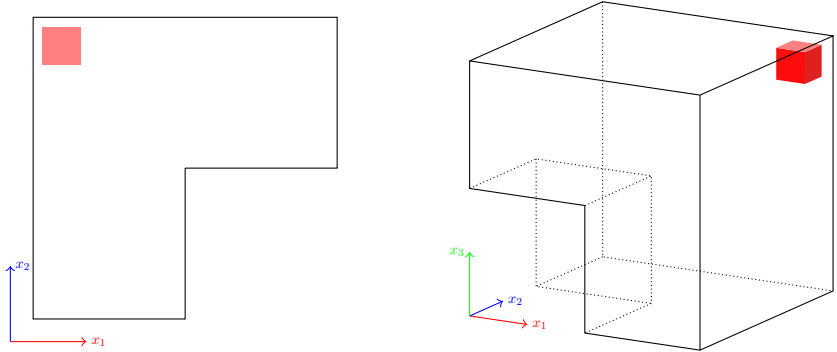
\begin{figure}[htb]
	\centering%
	\begin{tikzpicture}[scale=2]
		\draw (0,0) -- (1,0) -- (1,1) -- (-1,1) -- (-1,-1) -- (0,-1) -- (0,0);
		\fill[red,opacity=.5] (-15/16,15/16) -- (-15/16,11/16) -- (-11/16,11/16) -- (-11/16,15/16);
		\draw[->,thin,red] (-1.15, -1.15) -- (-0.65,-1.15) node[right,scale=0.5] {$x_1$};
		\draw[->,thin,blue] (-1.15, -1.15) -- (-1.15,-0.65) node[right,scale=0.5] {$x_2$};
	\end{tikzpicture}
	\hspace*{2em}
	\tdplotsetmaincoords{75}{30}
	\begin{tikzpicture}[tdplot_main_coords,scale=1.75]
		\begin{scope}[canvas is xz plane at y=11/16, fill=red!95!white]
			\fill (11/16, 11/16) -- (11/16,15/16) -- (15/16,15/16) -- (15/16,11/16) -- cycle;
		\end{scope}
		\begin{scope}[canvas is yz plane at x=15/16, fill=red!75!gray]
			\fill (11/16, 11/16) -- (11/16,15/16) -- (15/16,15/16) -- (15/16,11/16) -- cycle;
		\end{scope}
		\begin{scope}[canvas is xy plane at z=15/16, fill=red!50!white]
			\fill (11/16, 11/16) -- (11/16,15/16) -- (15/16,15/16) -- (15/16,11/16) -- cycle;
		\end{scope}

		\coordinate (0) at (0,-1,-1) [scale=.25] {(0) };
		\coordinate (1) at (1,-1,-1) [scale=.25] {(1) };
		\coordinate (2) at (-1,0,-1) [scale=.25] {(2) };
		\coordinate (3) at (0,0,-1) [scale=.25] {(3) };
		\coordinate (4) at (1,0,-1) [scale=.25] {(4) };
		\coordinate (5) at (-1,1,-1) [scale=.25] {(5) };
		\coordinate (6) at (0,1,-1) [scale=.25] {(6) };
		\coordinate (7) at (1,1,-1) [scale=.25] {(7) };
		\coordinate (8) at (-1,-1,0) [scale=.25] {(8) };
		\coordinate (9) at (0,-1,0) [scale=.25] {(9) };
		\coordinate (10) at (1,-1,0) [scale=.25] {(10)};
		\coordinate (11) at (-1,0,0) [scale=.25] {(11)};
		\coordinate (12) at (0,0,0) [scale=.25] {(12)};
		\coordinate (14) at (-1,1,0) [scale=.25] {(14)};
		\coordinate (16) at (1,1,0) [scale=.25] {(16)};
		\coordinate (17) at (-1,-1,1) [scale=.25] {(17)};
		\coordinate (18) at (0,-1,1) [scale=.25] {(18)};
		\coordinate (19) at (1,-1,1) [scale=.25] {(19)};
		\coordinate (20) at (-1,0,1) [scale=.25] {(20)};
		\coordinate (22) at (1,0,1) [scale=.25] {(22)};
		\coordinate (23) at (-1,1,1) [scale=.25] {(23)};
		\coordinate (24) at (0,1,1) [scale=.25] {(24)};
		\coordinate (25) at (1,1,1) [scale=.25] {(25)};

		\draw (17) -- (18) -- (19) -- (22) -- (25) -- (24) -- (23) -- (20) -- (17);
		\draw (19) -- (10) -- (1) -- (4) -- (7) -- (16) -- (25);
		\draw[densely dotted] (23) -- (14) -- (5) -- (6) -- (7);
		\draw[] (17) -- (8) -- (9) -- (0) -- (1);
		\draw[densely dotted] (9) -- (12);
		\draw[densely dotted] (5) -- (2);
		\draw[densely dotted] (8) -- (11) -- (12) -- (3) -- (0);
		\draw[densely dotted] (11) -- (2) -- (3);

		\draw[->,thin,red] (-1,-1,-1) -- (-0.5,-1,-1) node[right,scale=0.5] {$x_1$};
		\draw[->,thin,blue] (-1,-1,-1) -- (-1,-0.5,-1) node[right,scale=0.5] {$x_2$};
		\draw[->,thin,green] (-1,-1,-1) -- (-1,-1,-0.5) node[left,scale=0.5] {$x_3$};
	\end{tikzpicture}
	\caption{Example 4: Spatial integration domains $ \Omega_I$ (red) for the energy functional $J_1$, for $d=2$ (left) and $d=3$ (right)}\label{fig:ex3:domains}
\end{figure}

\subsubsection{L-shaped: Discussion and Interpretation of Our Findings}
First we consider the L-shaped domain 
$\Omega = (-1,1)^2 \setminus \{ [0,1)\times(-1,0]\}$
as depicted in the left of Figure~\ref{fig:ex3:domains}. Here, we know the exact solution of the corresponding elliptic problem\footnote{See \url{https://math.nist.gov/amr-benchmark/index.html}, 
and select ``L-shaped Domain Homogeneous Boundary Conditions''.} and we manufacture a time-dependent solution for the parabolic 
$p$-Laplace initial-boundary value problem,
which is given by
\begin{equation*}
	u(x_1,x_2,t) = \frac{3}{2} (1-x_1)^2(1-x_2)^2 r^{2/3} \sin(\frac{2}{3}\theta)\,\sin(t), 
\end{equation*}
where $ (r,\theta)$ denote the polar representation of the spatial coordinates $(x_1,x_2)$. The right-hand side is computed accordingly. 

We first consider the efficiency indices as presented in the left of Figure~\ref{fig:ex3:2d+1:efficiency and errors}. Here, we observe that after some initial oscillations, the combined efficiency index $I_{\mathrm{eff}} $ is very close to $1$. These initial oscillations most likely occur as the integration domain $Q_I = \Omega_I \times (0.5,0.75)$ is not exactly captured in the initial mesh; see Figure~\ref{fig:ex3:2d+1:meshes} (top right). Next, we consider the convergence of the error for the functionals. In the right of Figure~\ref{fig:ex3:2d+1:efficiency and errors}, we observe that the error in the combined functional $ J_{\mathfrak{E}} $ and the error in the ``average'' $J_2$ converge with almost $\mathcal{O}(\mathrm{DoFs}^{-2/3})$, i.e.\ quadratic. While the error for the ``energy'' integral $ J_1$ is more oscillatory, we observe an overall almost quadratic convergence rate of around $\mathcal{O}(\mathrm{DoFs}^{-2/3})$. In Figure~\ref{fig:ex3:2d+1:error J1 and J2}, we present the error of the individual functionals, and compare the adaptive refinements with uniform refinements. We again observe the almost quadratic convergence of the adaptive refinements, while uniform refinements on the one hand need more DoFs to reach a specific error, and on the other hand have a worse convergence rate in the pre-asymptotic range.

\begin{figure}[!htb]
	\centering%
	\includegraphics[width=.95\linewidth]{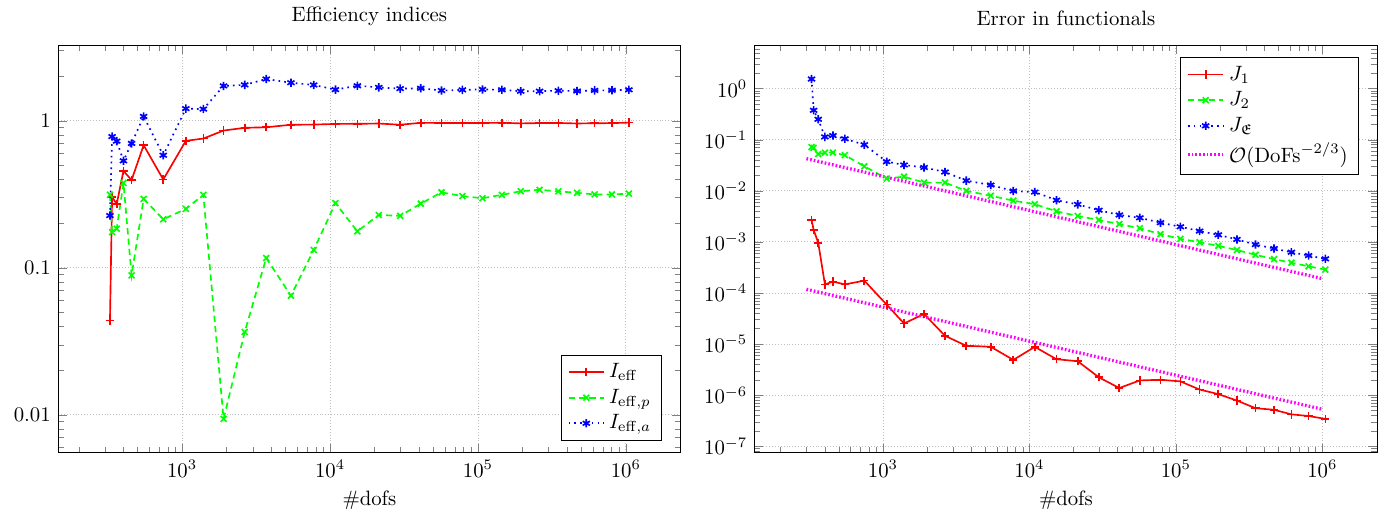}
	\caption{Example~4: Efficiency indices and error convergence.}\label{fig:ex3:2d+1:efficiency and errors}
\end{figure}

\begin{figure}[!htb]
	\centering%
	\includegraphics[width=.95\linewidth]{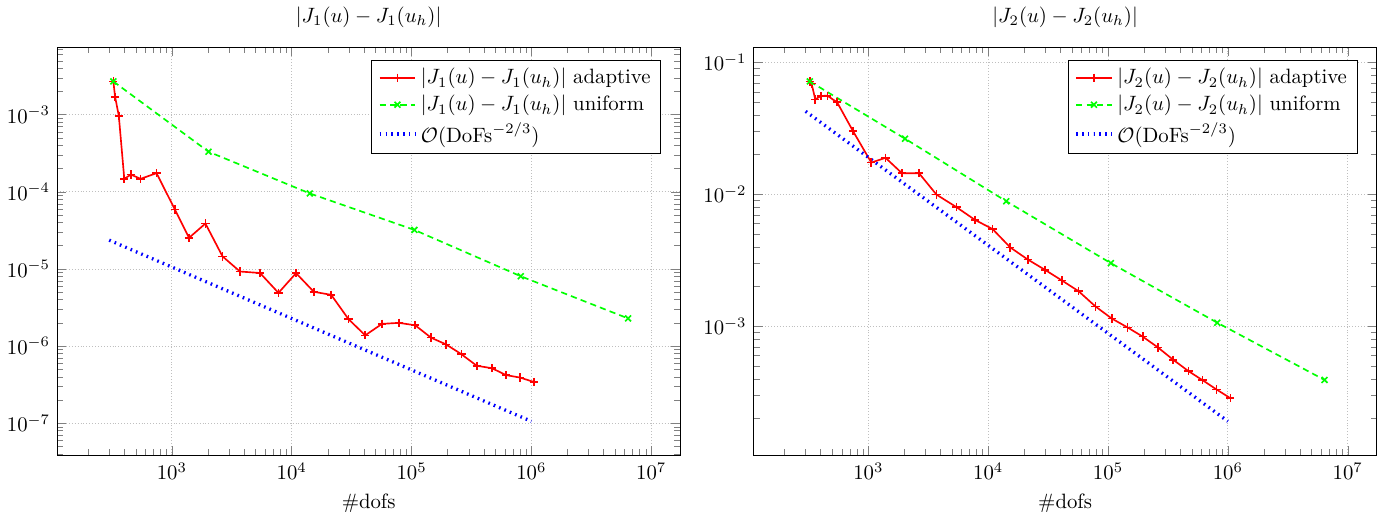}
	\caption{Example~4: Error convergence for the functionals and comparison with uniform refinements.}\label{fig:ex3:2d+1:error J1 and J2}
\end{figure}

Next, we consider the meshes obtained from the adaptive algorithm. In Figure~\ref{fig:ex3:2d+1:meshes}, we present the space-time mesh in the first two rows, and three cuts with the $(x_1,x_2)$-plane in the bottom row. In the upper two rows, we additionally present intersections of the space-time cylinder with the $(x_1,t)$-plane at $x_2 = 13/16$ and with the $(x_2,t)$-plane at $ x_1 = -13/16$ in the right column. We indicate these cuts also in the left columns by thick red lines. In the upper row, we show the initial meshes, and in the middle row the meshes obtained after 18 adaptive refinements using Algorithm~\ref{Alg: adapt. multigoal Algorithm}. As expected, we observe that the top of the space-time cylinder is refined, with special focus on the corner singularity. Moreover, we observe refinements toward the integration domain $Q_I$ of $J_2$, which is indicated with red shading. Since $\Omega_I$ is quite far away from the singularity, the pollution effect is quite mild. This can also be observed in the bottom row of Figure~\ref{fig:ex3:2d+1:meshes}, as the refinements towards the re-entrant corner are far less in the first two columns compared with the third column. From left to right, the cuts are located at $t=0.5$, $ t = 0.625$, and $t = 1$, respectively.

\begin{figure}[!htb]
	\centering%
	\includegraphics[width=.495\linewidth]{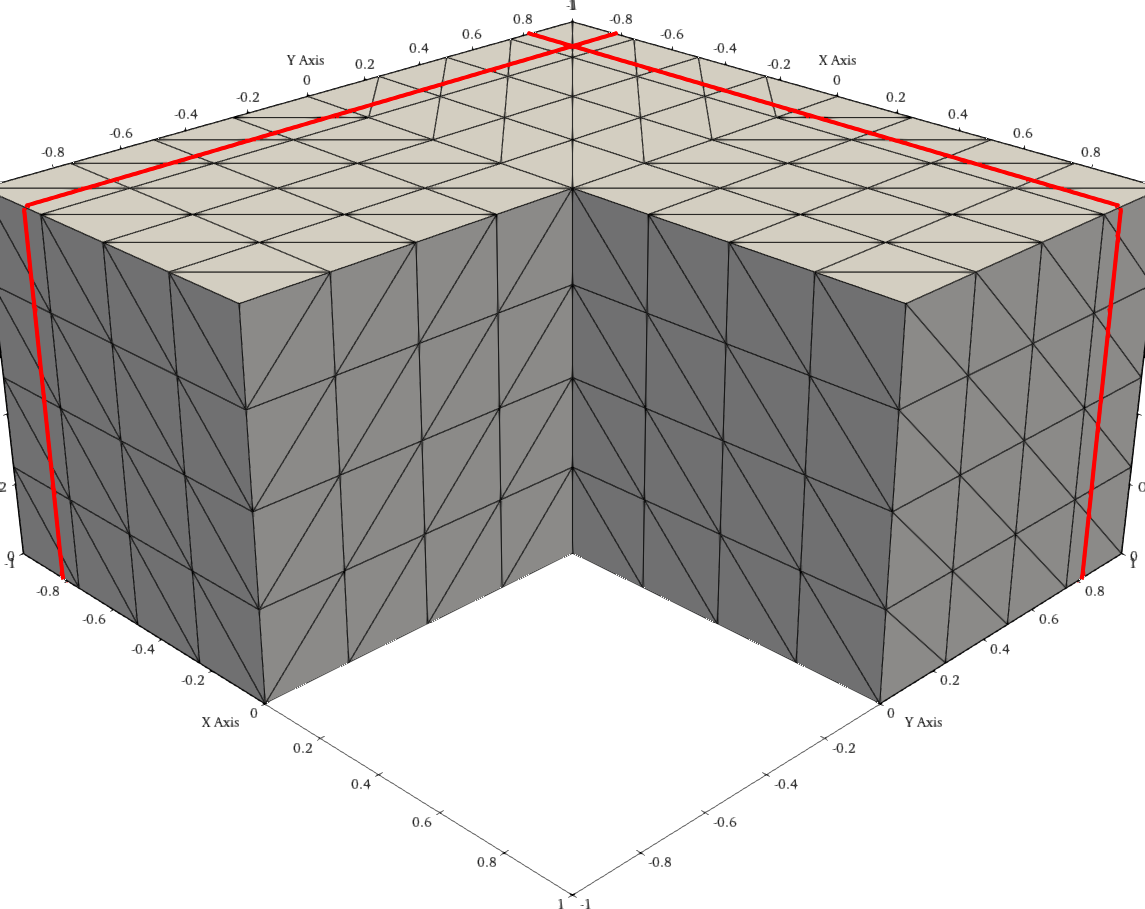}%
	\hfill%
	\includegraphics[width=.495\linewidth]{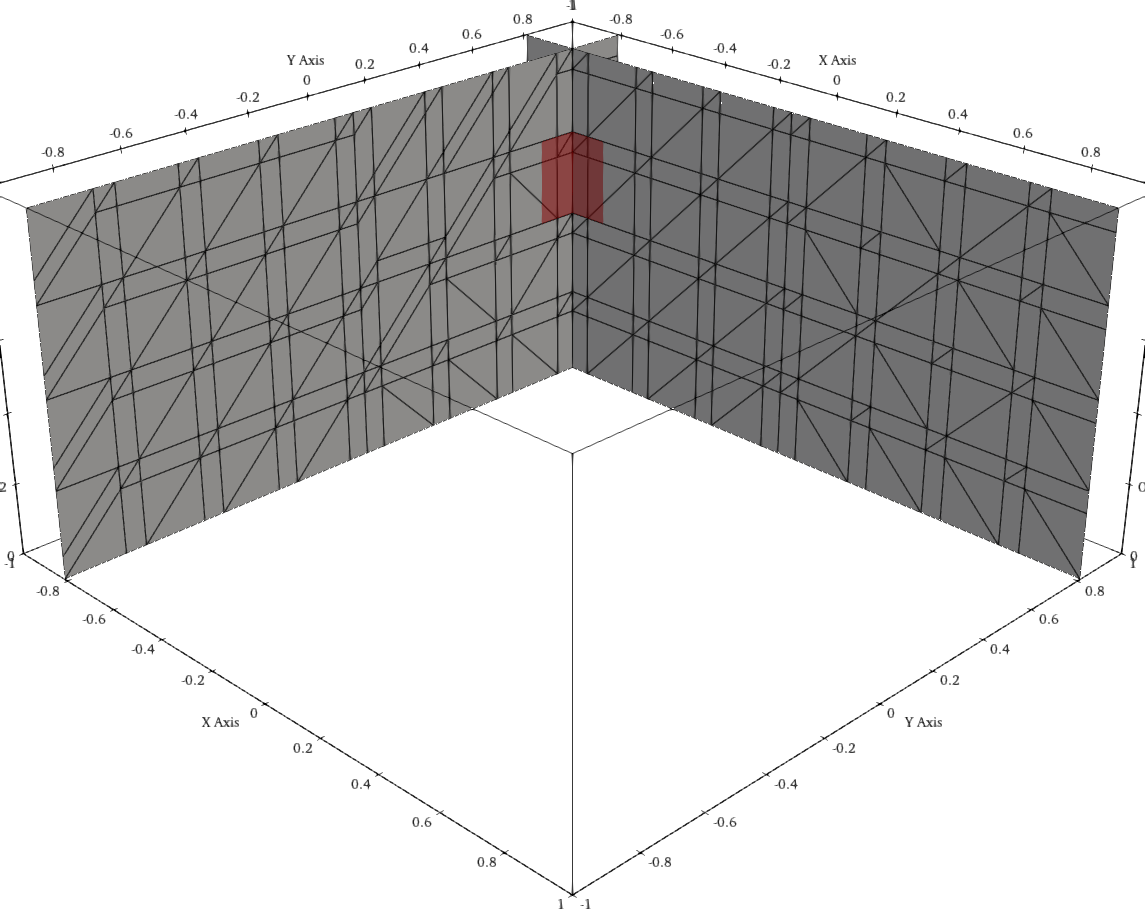}%
	\hfill%
	\includegraphics[width=.495\linewidth]{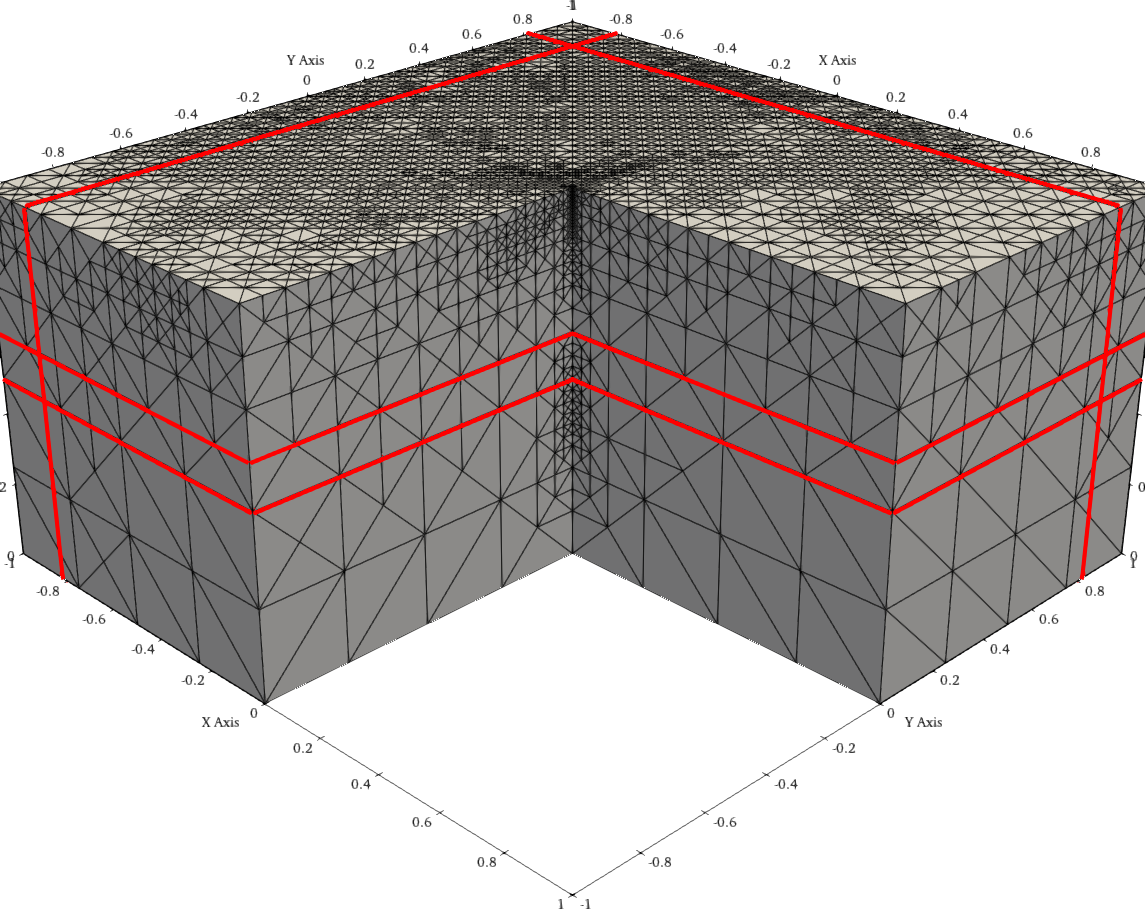}%
	\hfill%
	\includegraphics[width=.495\linewidth]{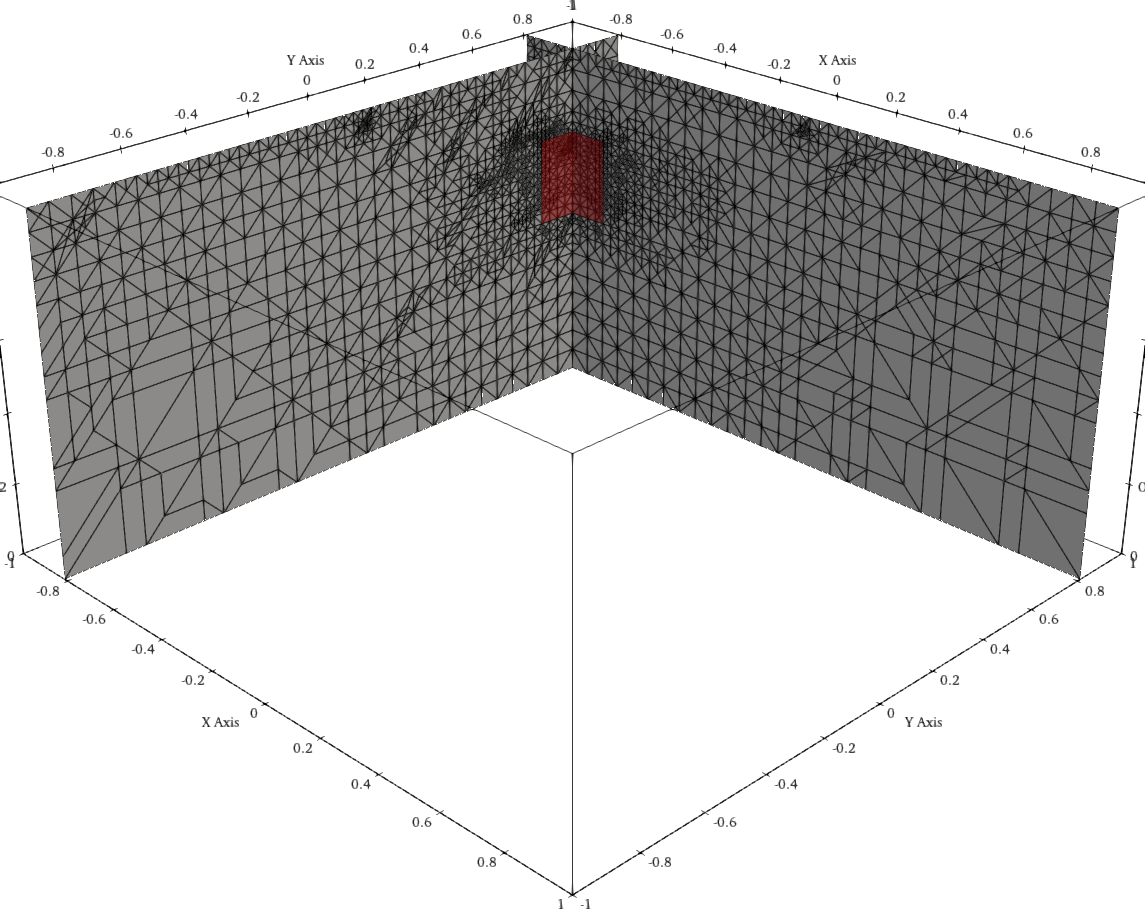}%
	\hfill%
	\includegraphics[width=.3295\linewidth]{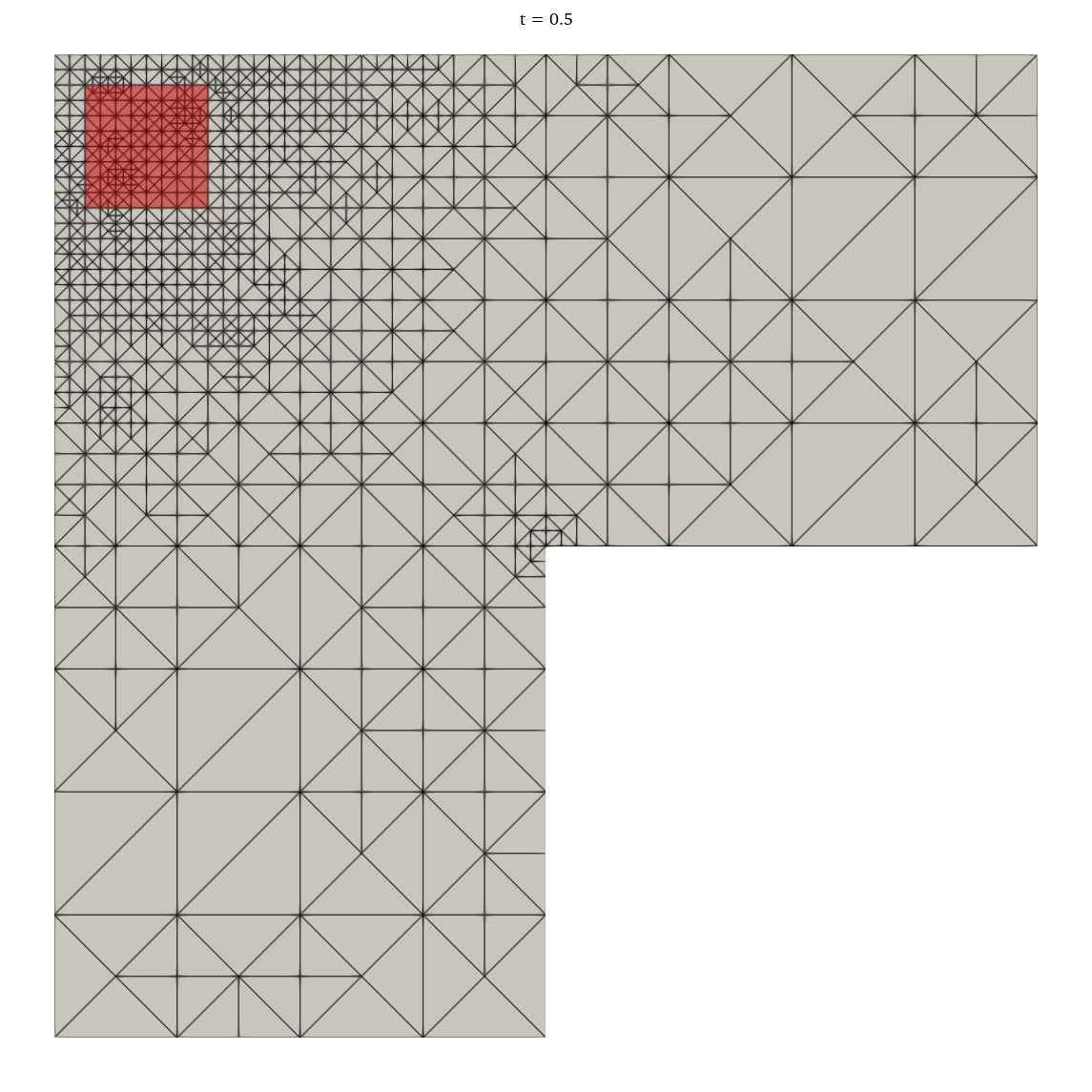}%
	\hfill%
	\includegraphics[width=.3295\linewidth]{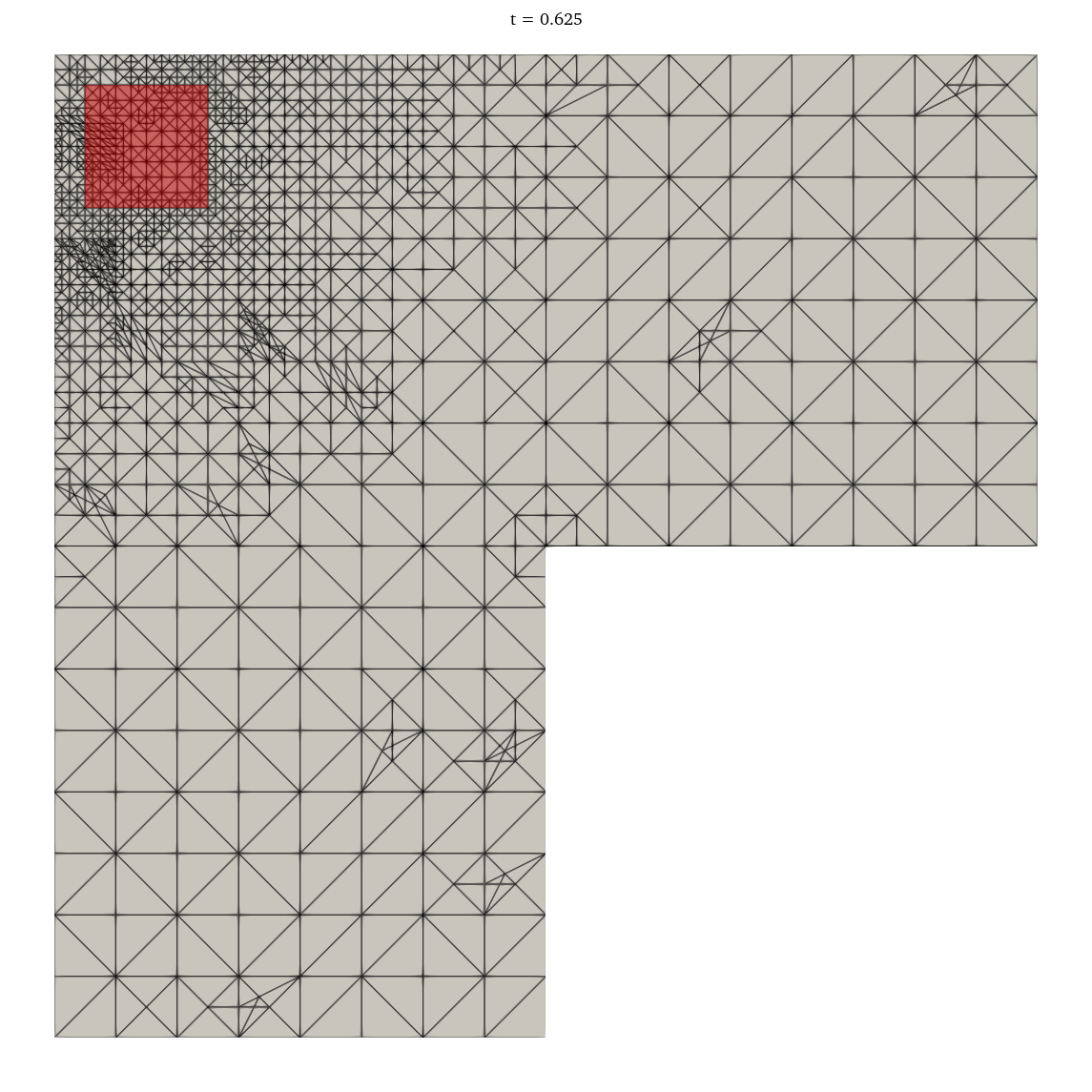}%
	\hfill%
	\includegraphics[width=.3295\linewidth]{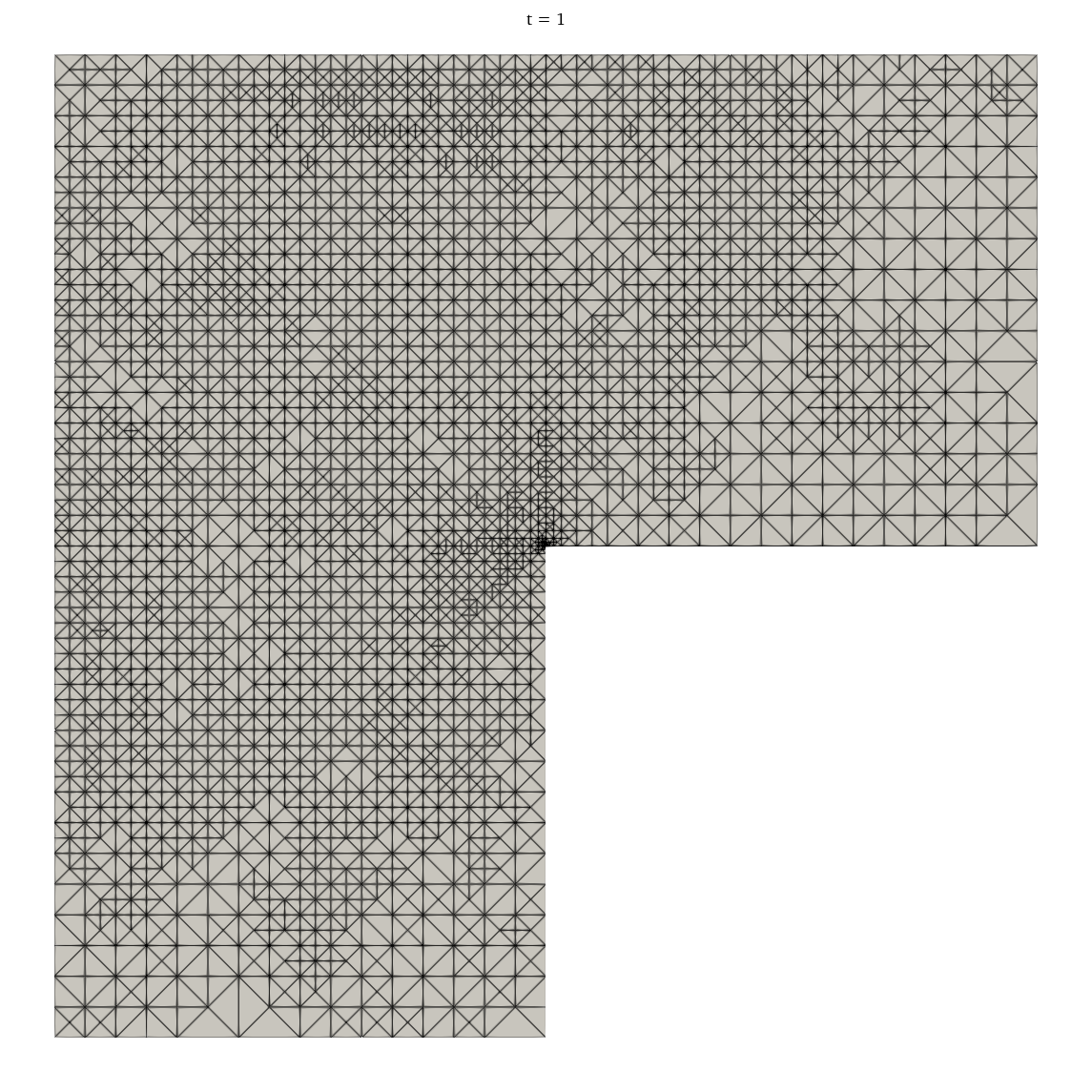}%
	\caption{Example~4: Space-time meshes in the first two rows; $(x_1,x_2)$-planes at $t=0.5$, $t=0.625$ and $t=1$ in the bottom row. The upper row shows the initial mesh, while the remaining meshes are obtained after 18 adaptive refinements using Algorithm~\ref{Alg: adapt. multigoal Algorithm}.}\label{fig:ex3:2d+1:meshes}
\end{figure}

\subsubsection{Fichera Corner: Discussion and Interpretation of Our Findings}
For the Fichera corner $\Omega = (-1,1)^3 \setminus(-1,0]^3$ (see Figure~\ref{fig:ex3:domains} right), in addition to the corner singularity at the origin, there can also be edge singularities along the re-entrant edges. In particular, we choose the right hand side~\footnote{See \url{https://math.nist.gov/amr-benchmark/index.html} and select ``Fichera Corner with Vertex and Edge Singularities'', or see \cite{ApelNicaise1998a}.}
\begin{equation*}
	f(x,t) = \frac{\sin(t)}{\sqrt{x_1^2 + x_2^2 + x_3^3}}.
\end{equation*}
In this case, to the best of our knowledge, there is no explicit form of the solution $u$. Nevertheless, we will again apply our adaptive multi-goal Algorithm~\ref{Alg: adapt. multigoal Algorithm}, with the same functionals as for the L-shape domain. For the first functional, we now use the three-dimensional $ \Omega_I = (11/16, 15/16)^3 $, see the right of Figure~\ref{fig:ex3:domains} for a visualization. Instead of presenting the convergence behavior of the error, we plot the convergence behavior of the error indicator $ \eta_h^{(2)}$. As we have observed in the previous examples, after some refinements $\eta_h^{(2)} $ provides a good estimate for the actual error. 

In Figure~\ref{fig:ex3:3d+1:parts of the error estimator}, we present the parts of the error estimator we use to drive the adaptive mesh refinement. We observe that the primal and 
adjoint parts of the error estimator both decay with a rate of $ \mathcal{O}(\mathrm{DoFs}^{-2/4}) $, i.e. quadratic convergence. In Figure~\ref{fig:ex3:3d+1:functional values J1 and J2}, we compare the convergence history of the functional values obtained by uniform refinements and adaptive refinements. As a reference value for the ``exact'' functional values, we plot the value of the functionals evaluated at the enriched solution $ u_h^{(2)} $ on the finest adaptive mesh.

\newpage
\mbox{}

\begin{figure}[htb]
	\centering%
	\includegraphics[width=.495\linewidth]{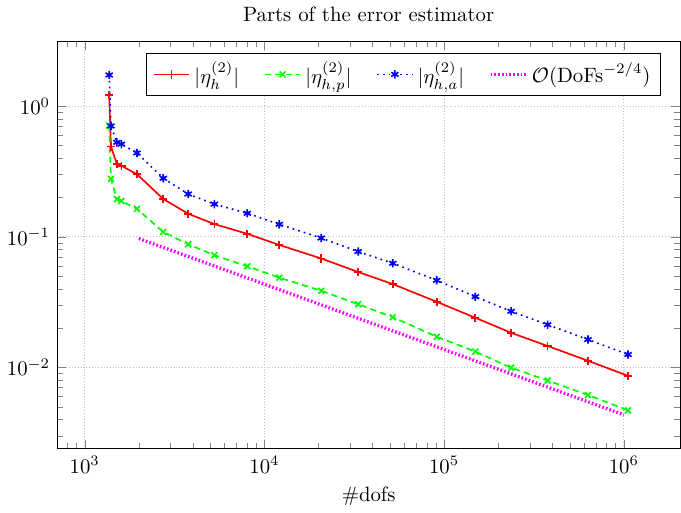}
	\caption{Example 4: Convergence history for the parts of the error estimator.}\label{fig:ex3:3d+1:parts of the error estimator}
\end{figure}

\begin{figure}[htb]
	\centering%
	\includegraphics[width=.95\linewidth]{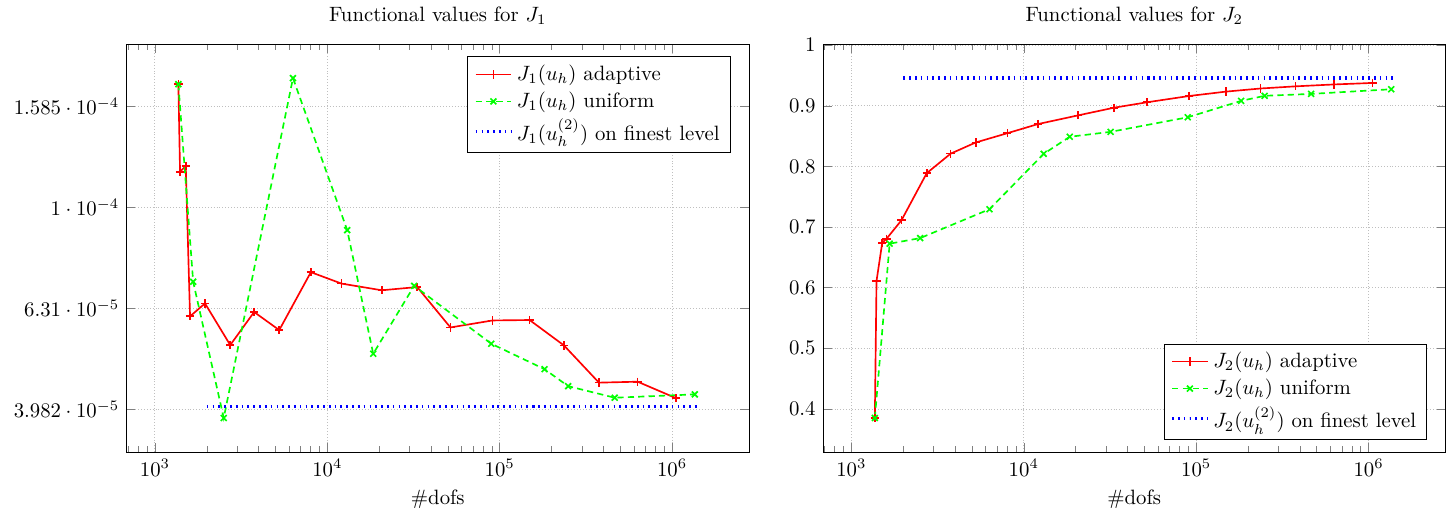}
	\caption{Example 4: Convergence history of the functional values.}\label{fig:ex3:3d+1:functional values J1 and J2}
\end{figure}

Finally, we present intersections of the four dimensional space-time mesh with the $(x_1,x_2,x_3)$-hyper-plane in Figure~\ref{fig:ex3:3d+1:meshes}. From top to bottom, the cuts were performed at $t=0$, $t=0.5$ and $t=1$, respectively. 
In the left column, we present the initial mesh, and in the right column the mesh after $18$ adaptive refinements. We observe that the mesh refinements toward the edge and corner singularities are much more prominent at final time $t=1$ as, e.g.,\ at time $t=0.5$. This behavior is similar to our earlier observations for the 2D+1 case.

\newpage
\mbox{}

\begin{figure}[!htb]
	\centering%
	\includegraphics[width=.4\linewidth]{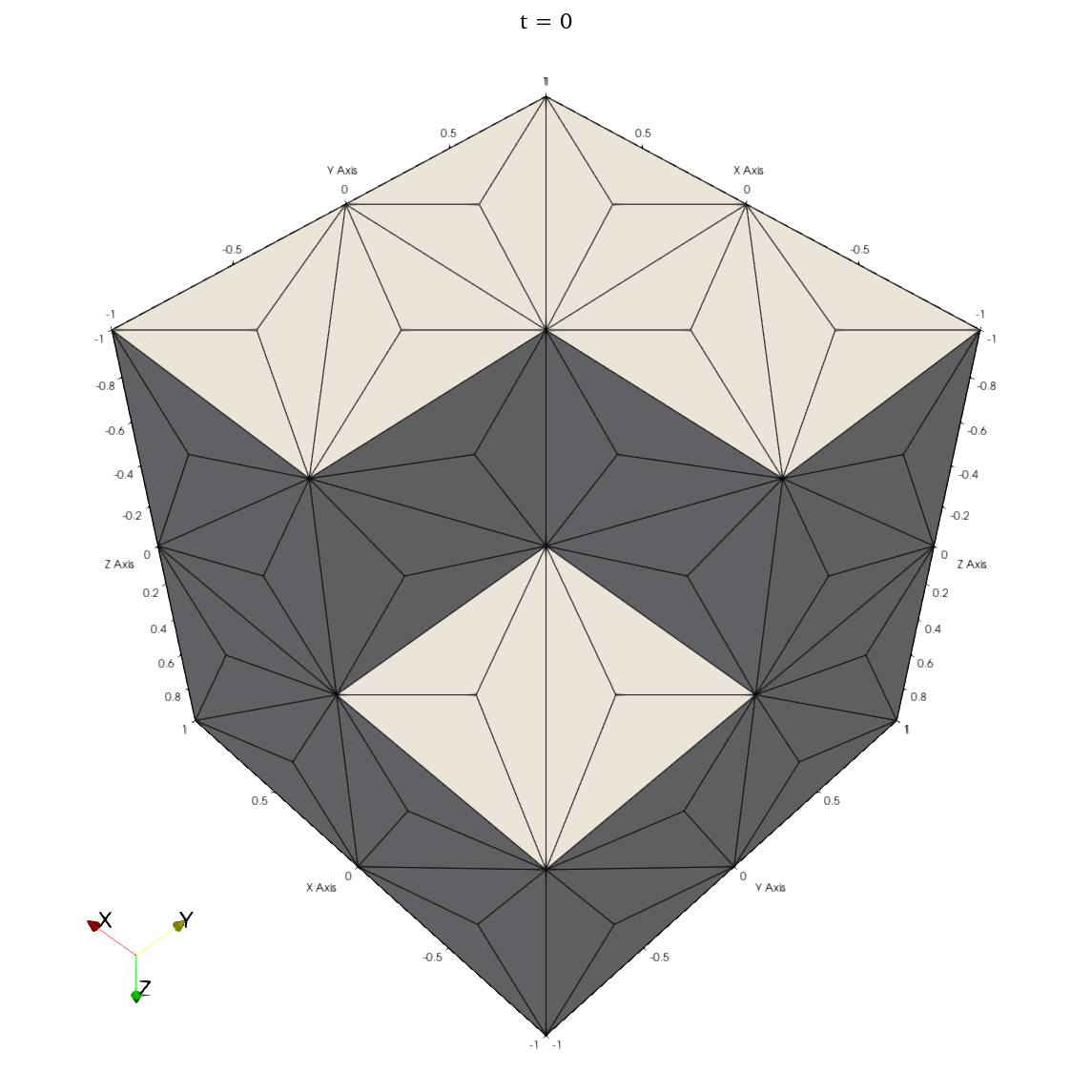}%
	\hspace*{.05\linewidth}%
	\includegraphics[width=.4\linewidth]{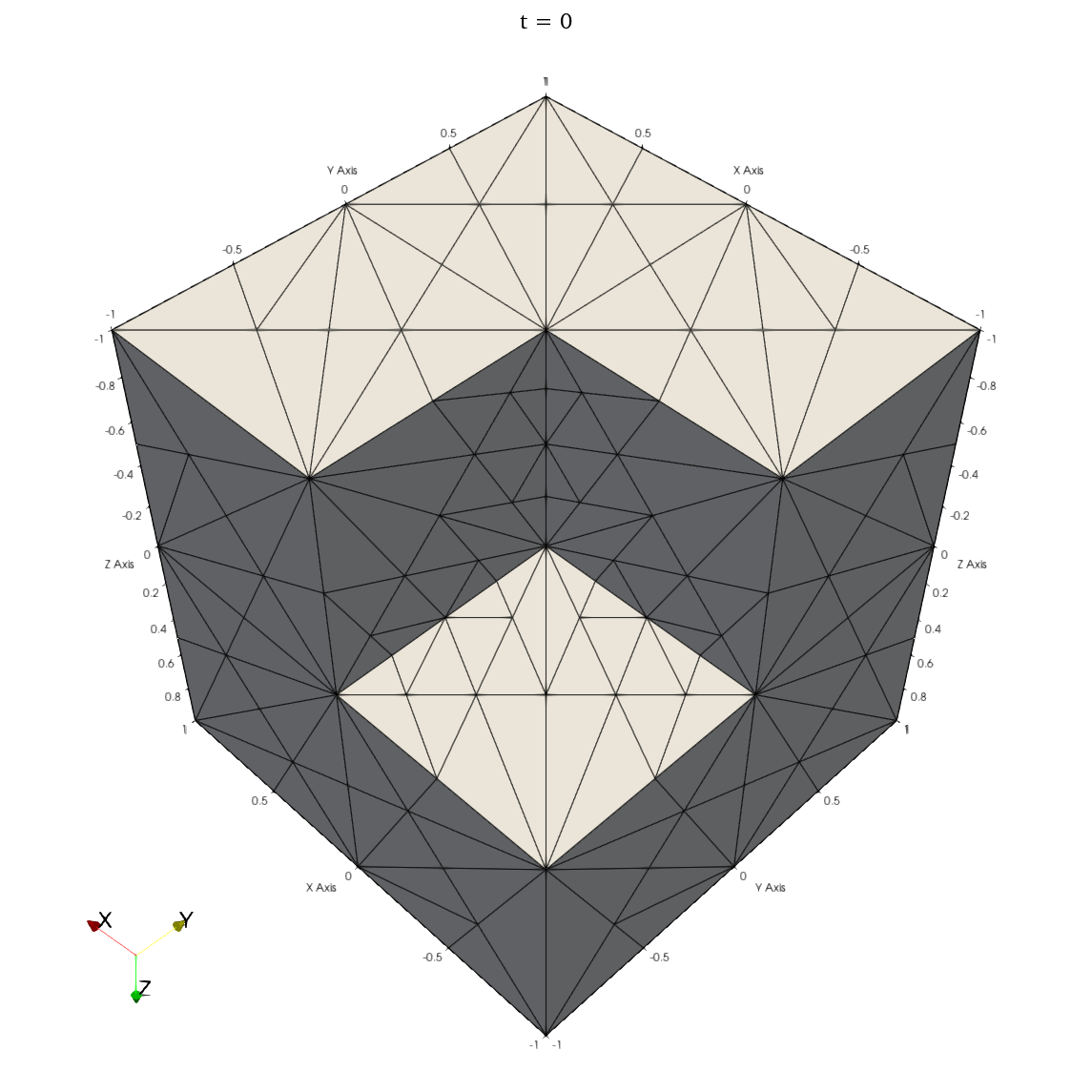}%
	\hfill%
	\includegraphics[width=.4\linewidth]{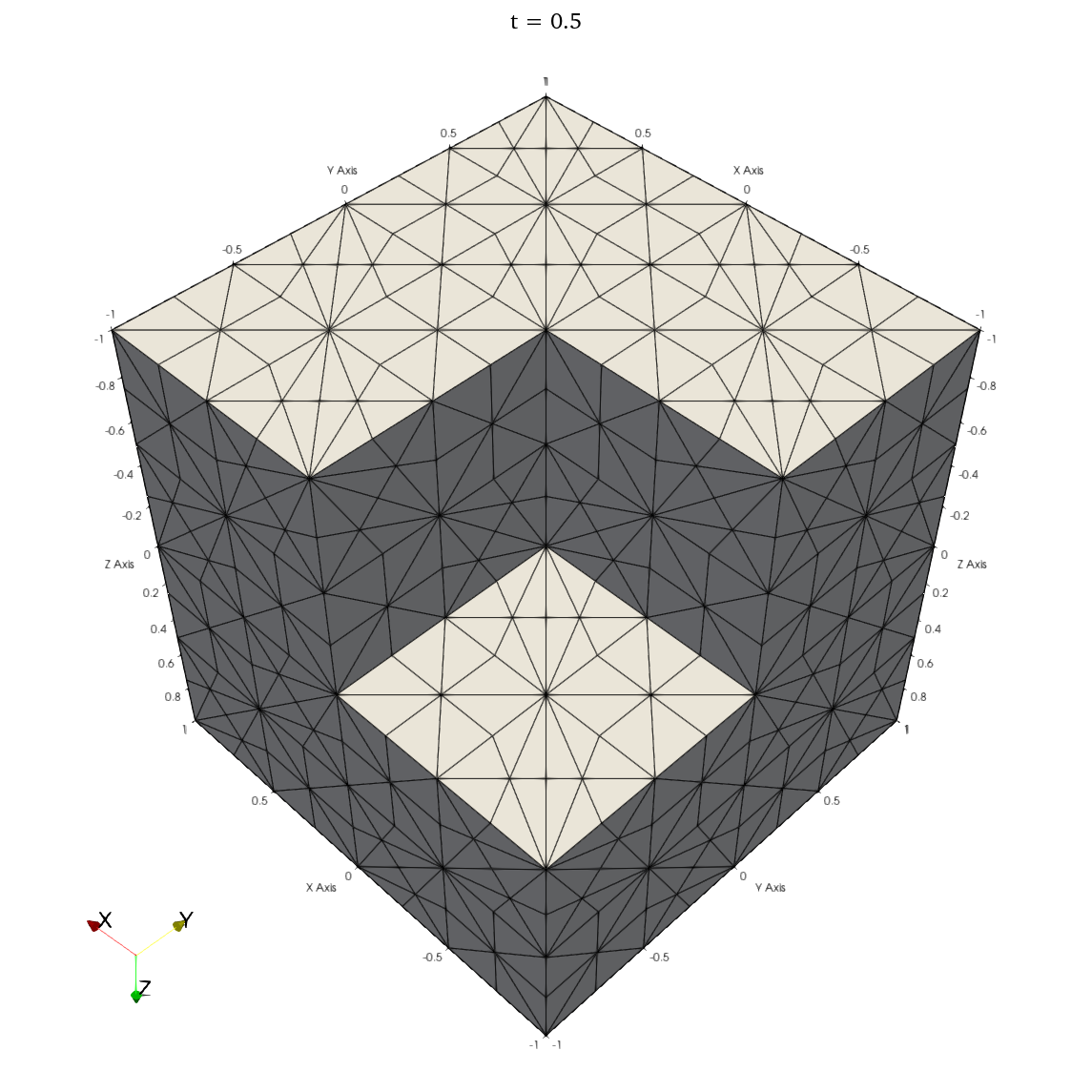}%
	\hspace*{.05\linewidth}%
	\includegraphics[width=.4\linewidth]{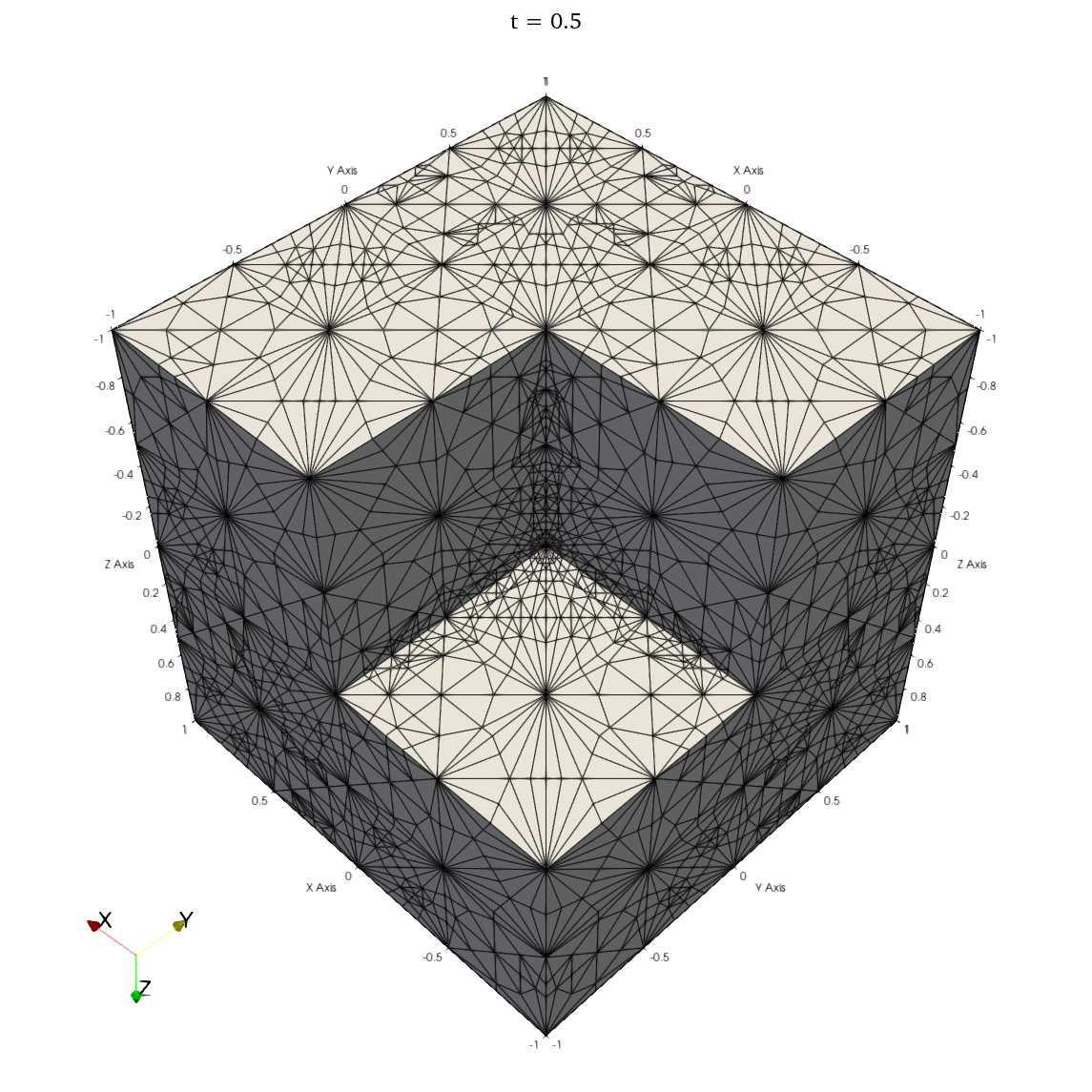}%
	\hfill%
	\includegraphics[width=.4\linewidth]{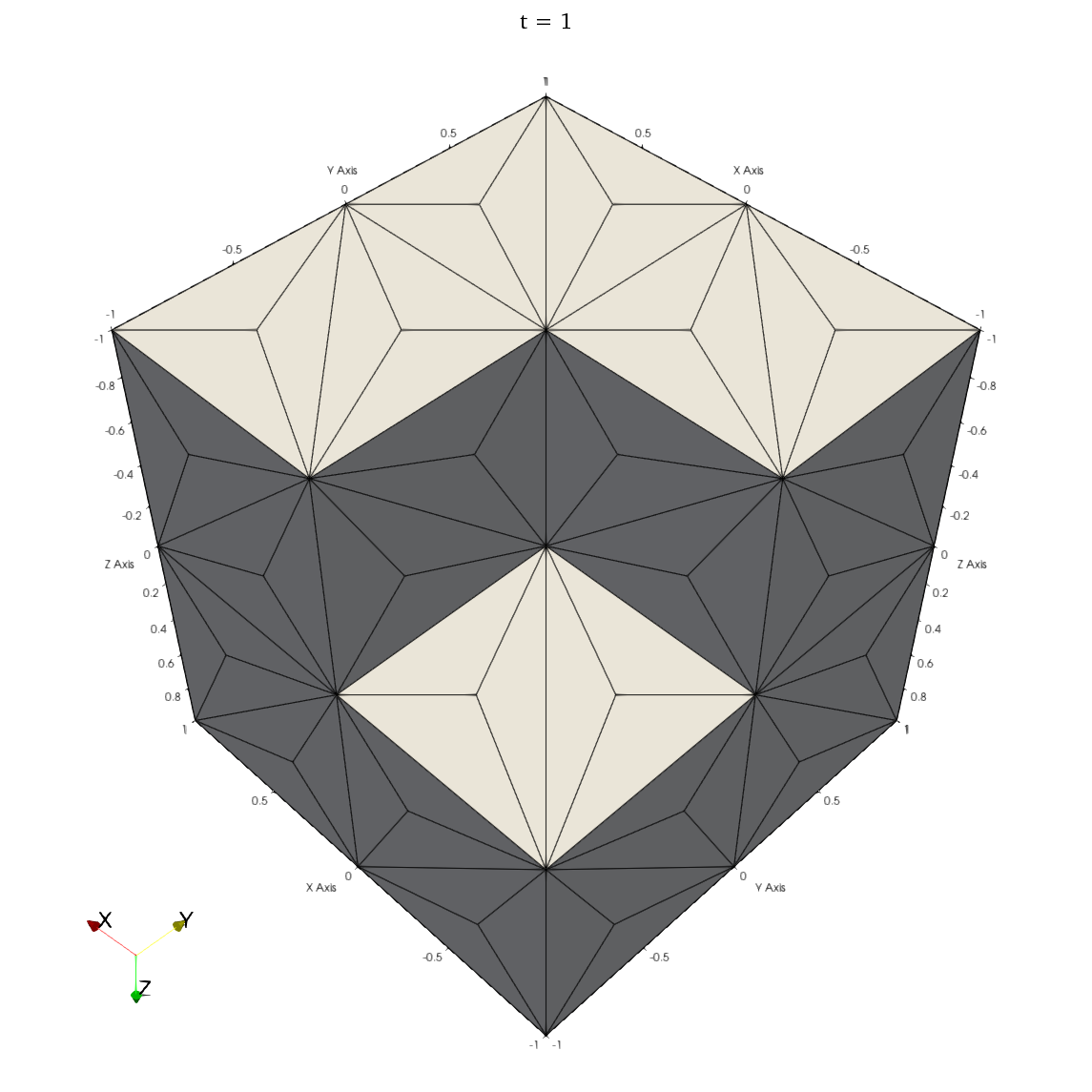}%
	\hspace*{.05\linewidth}%
	\includegraphics[width=.4\linewidth]{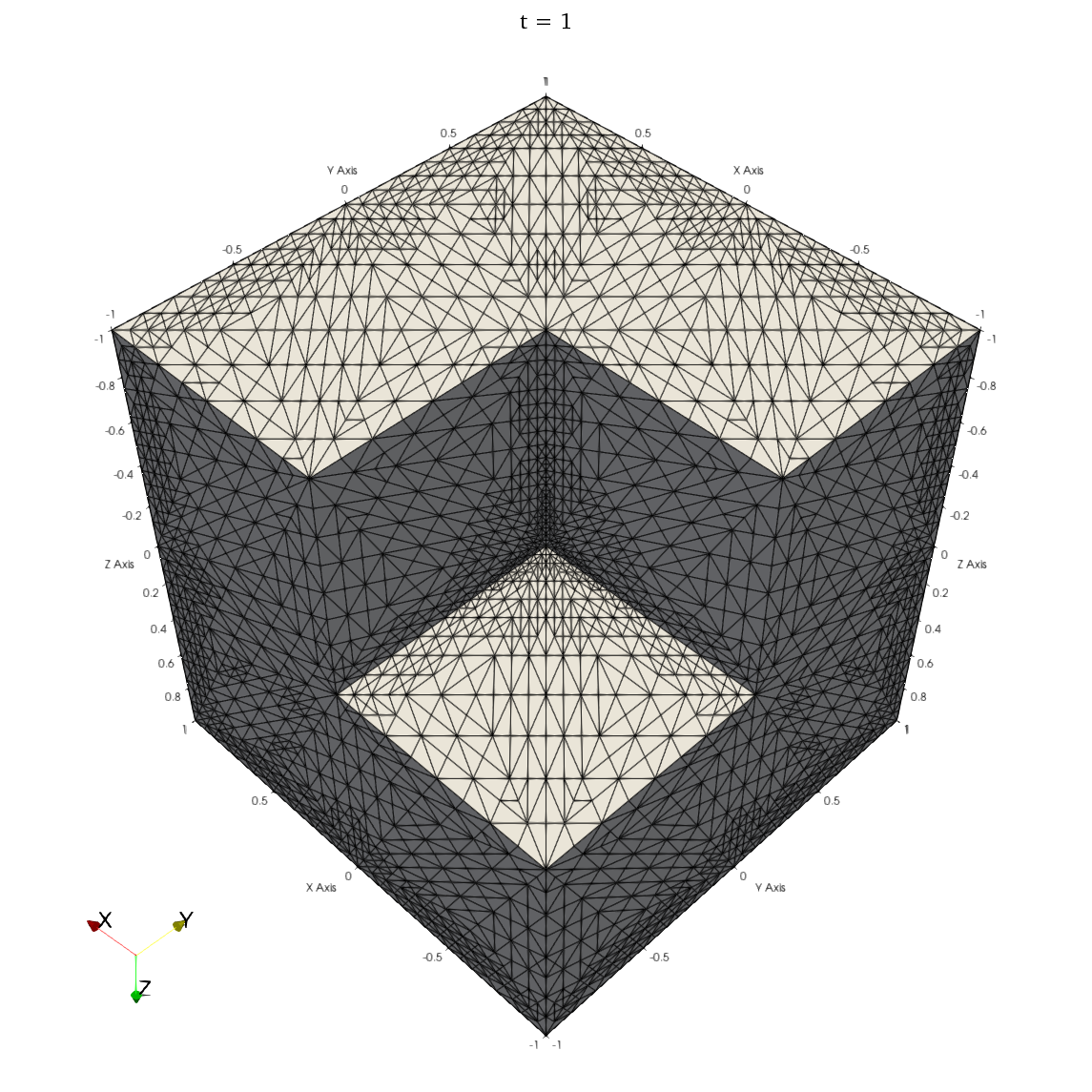}%
	\caption{Example 4: Intersections of the space-time mesh with the $(x_1,x_2,x_3)$-hyper-plane; at $t=0$ (top row), at $t=0.5$ (middle row), and at $t=1$ (bottom row). The left column shows the initial mesh, the right column the mesh after 18 adaptive refinements.}\label{fig:ex3:3d+1:meshes}
\end{figure}

%
\newpage
\section{Conclusions and Outlook}
\label{Sec:Conclusions}
In this work, we reviewed single- and multigoal-oriented error control 
and adaptivity. Our notation was kept in general terms such that 
stationary and non-stationary (space-time) situations are covered. 
First, we explained in detail single goal-oriented error estimation
with the help of the dual-weighted residual method. The error estimators 
cover both discretization and non-linear iteration errors. 
Therein, some new theoretical results were presented as well. Prior 
efficiency and reliability results only require one saturation assumption, rather 
a strengthened condition. Next, we considered non-standard discretizations such 
as stabilization terms and non-consistencies (e.g., classical finite difference based 
time-stepping schemes) or non-conformal methods (e.g., finite differences or neural 
network approximations) of numerical schemes in the frame of 
goal-oriented error estimation.
Third, we concentrated on multigoal-oriented error estimation in which 
we have put a lot of efforts in the last eight years due to advances and increasing interest of solving multiphysics partial differential equations
and coupled variational inequality systems.
Besides theoretical results, we provided several adaptive algorithms
for single and multiple goal functional evaluations.
In order to substantiate our developments, four numerical 
examples were designed and computationally analyzed. In the first 
example, namely Poisson's problem, 
the implementation is fully open-source and follows the classical 
structure of deal.II tutorial steps. Therefore, this example has a flavor 
of educational purpose. In the last example, recent work 
on the space-time PU-DWR method is further extended to multiple-goal functionals 
and singularities in the solution.
Ongoing and future work 
is concerned with the extension to cover more terms in the error estimators 
that besides discretization and non-linear iteration errors, as well
linear iteration errors, model errors, and model order reduction techniques
are covered. In conjunction with high-performance parallel computing these yield
important components to continue to solve efficiently 
and with desired accuracies multiphysics problems with physics-based discretizations and fast physics-based numerical solvers.

\section*{Acknowledgments}
BE and TW have been supported by the Cluster of Excellence PhoenixD (EXC 2122, Project ID 390833453).
Furthermore, BE greatly acknowledges his current Humboldt Postdoctoral 
Fellowship funding. Furthermore, we would like to thank 
the Johann Radon Institute for Computational and  Applied Mathematics
(RICAM) for providing 
computing resource support of the high-performance computing cluster Radon1
\footnote{\url{https://www.oeaw.ac.at/ricam/hpc}}.
%
%
\bibliographystyle{abbrv}
\bibliography{EndtmayerLangerRichterWickSchafelner}
%
\end{document}